\documentclass[11pt]{amsart}
\usepackage[english]{babel}

\usepackage{pdfsync}
\usepackage{amsfonts,amssymb}
\usepackage{wasysym}
\usepackage{enumerate}
\usepackage{a4wide}
\usepackage[T1]{fontenc}

\newtheorem{theorem}{Theorem}[section]
\newtheorem{definition}[theorem]{Definition}
\newtheorem{lemma}[theorem]{Lemma}
\newtheorem{corollary}[theorem]{Corollary}
\newtheorem{proposition}[theorem]{Proposition}
\newtheorem{rem}[theorem]{Remark}

\newenvironment{remark}{\begin{rem}}{$\square$\end{rem}\normalfont}

\newcommand{\C}{\mathbb{C}}
\newcommand{\CC}{\mathcal{C}}
\newcommand{\HH}{\mathbb{H}}
\newcommand{\OO}{\mathbb{O}}
\newcommand{\R}{\mathbb{R}}
\newcommand{\esf}{\mathbb{S}}
\newcommand{\esfera}{\mathbb{S}^{4n+3}}

\newcommand{\Ad}{\mathrm{Ad}}
\newcommand{\Sp}{\mathrm{Sp}}
\newcommand{\SU}{\mathrm{SU}}
\newcommand{\SO}{\mathrm{SO}}
\renewcommand{\j}{\mathbf{j}}
\renewcommand{\hom}{\mathrm{Hom}}
\def\dim{\mathop{\hbox{\rm dim}}}
\def\der{\mathop{\rm Der}}
\def\tr{\mathop{\rm tr}}
\def\T{\mathop{\mathcal{T}}}
\newcommand{\spf}{\mathfrak{sp}}
\newcommand{\suf}{\mathfrak{su}}
\newcommand{\sof}{\mathfrak{so}}
\newcommand{\slf}{\mathfrak{sl}}
\newcommand{\ad}{\mathop{\mathrm{ad}}}
\newcommand{\id}{\mathrm{id}}
\newcommand{\mm}{\mathfrak{m}}
\newcommand{\hh}{\mathfrak{h}}
\newcommand{\g}{\mathfrak{g}}
\newcommand{\V}{\mathcal{V}}
\newcommand{\U}{\mathcal{U}}
\newcommand{\ii}{\mathbf {i}}
\newcommand{\ef}{\mathfrak{e}}
\newcommand{\TT}{\mathbf{T}}
\newcommand{\s}{\mathbf{S}}
\newcommand{\estS}{{S}}

\begin{document}

\title[Affine Connections on 3-Sasakian Homogeneous Manifolds]{Affine Connections \\on 3-Sasakian Homogeneous Manifolds}

\date{}

\author[C.~Draper]{Cristina Draper${}^\dagger$}
\address[C.~Draper]{\normalfont Partially supported by the Spanish Ministery of Economy and Competitiveness (MEC), and European Region Development Fund (ERDF), project MTM2016-76327-C3-1-P and by the Junta de Andaluc\'{\i}a grant FQM-336.}
\curraddr{}
\email{}

\thanks{${}^\dagger$ C.~Draper and F.J.~Palomo: Departamento de Matem\'atica Aplicada, Universidad de M\'alaga,
Ampliaci\'on Campus de Teatinos, 29071 M\'alaga (Spain).\, C.~Draper: \textit{cdf@uma.es}. ORCID 0000-0002-2998-7473.\, F.J.~Palomo: \textit{fjpalomo@ctima.uma.es}. ORCID 0000-0002-1852-0667.}

\author[F.~J.~Palomo]{Francisco J. Palomo} 
\curraddr{}
\address[F.~J.~Palomo]{\normalfont Partially supported by  the Spanish MEC, and ERDF, project MTM2016-78807-C2-2-P.}

\author[M.~Ortega]{Miguel Ortega${}^\star$}
\address[M.~Ortega]{\normalfont Partially supported by the Spanish MEC, and ERDF, project   MTM2016-78807-C2-1-P.  Also, M.~Ortega and F.J.~Palomo are partially supported by the Junta de Andaluc\'{\i}a grant FQM-324.}
\curraddr{}
\thanks{${}^\star$ Instituto de Matemáticas IEMathUGR,
Departamento de Geometr\'ia y Topolog\'ia, Facultad de Ciencias, Universidad de Granada, 18071 Granada (Spain). \textit{miortega@ugr.es}; ORCID 0000-0002-1390-9980.} 

\begin{abstract}
The space of invariant affine connections  on every $3$-Sasakian  homogeneous manifold of dimension at least $7$ is  described. In particular, the  subspaces of invariant affine metric connections,  and the subclass   with skew-torsion, are also determined. To this aim, an explicit construction of all $3$-Sasakian homogeneous manifolds is  exhibited.  It is shown that   the    $3$-Sasakian homogeneous manifolds which admit nontrivial Einstein with skew-torsion invariant affine  connections are those of dimension $7$, that is, $\mathbb{S}^7$, 
$\R P^7$ and the Aloff-Wallach space $\mathfrak{W}^{7}_{1,1}$. On $\mathbb{S}^7$ and $\R P^7$, the set of such  connections is bijective to two copies of the  conformal linear transformation group of the Euclidean space, while it is  strictly bigger on  $\mathfrak{W}^{7}_{1,1}$. The set of invariant connections with skew-torsion whose Ricci tensor satisfies that its eigenspaces are the canonical vertical and horizontal  distributions, is fully described on $3$-Sasakian homogeneous manifolds.  An affine connection satisfying  these conditions is distinguished, by parallelizing all  the Reeb vector fields associated with the $3$-Sasakian  structure, which is also Einstein with skew-torsion  on the $7$-dimensional  examples. The invariant metric affine connections on $3$-Sasakian homogeneous manifolds with parallel skew-torsion have been found. Finally, some  results have been adapted to the non-homogeneous  setting.
\end{abstract}

\maketitle 



\section{Introduction}

Almost contact metric structures on $(2n+1)$-dimensional (smooth) manifolds may be regarded 
as an analogue of Hermitian structures for odd dimensional manifolds. Recall that, for every almost contact metric structure on a  $(2n+1)$-dimensional manifold, there exists a general method, \textit{the cone construction}, which permits to obtain a $(2n+2)$-manifold endowed with an Hermitian structure. Of course, the family of differentiable manifolds with Hermitian structures is wider than the set of K\"ahler manifolds. In this setting,  the following natural question raised. When does the cone construction produce a K\"ahler manifold? This very special kind of almost contact metric structure is now known as a Sasakian manifold in honour to the Japanese geometer Shigeo Sasaki, who introduced it in 1960. An extensive and complete study of Sasakian manifolds and related topics can be found in the excellent monograph \cite{galickiboyer}. 

In the late sixties, the new notion of the $3$-Sasakian manifold was introduced as a $(4n+3)$-dimensional manifold with a family of Sasakian structures parametrized by points on the $2$-dimensional unit sphere 
$\mathbb{S}^{2}$ and satisfying several compatibility conditions (Section \ref{seccion2}). Every $3$-Sasakian manifold $M$ carries $3$-orthonormal Killing vector fields which span, at every tangent vector space to $M$, a copy of the Lie algebra $\mathfrak{sp}(1)$, and therefore a $3$-dimensional foliation $\mathcal{F}_{Q}$. Under the assumption that $\mathcal{F}_{Q}$ is regular, the space of leaves $M/\mathcal{F}_{Q}$ inherits a hyper-K\"{a}hler structure with positive scalar curvature, \cite{IK}. That is, $M/\mathcal{F}_{Q}$ is endowed with a suitable family of integrable complex structures which satisfy     the quaternionic identities (see details in \cite[Chapter 13]{galickiboyer}). 
The relationship with hyper-K\"{a}hler structures goes in two ways. In fact, in this setting, starting now with a $3$-Sasakian manifold, the   cone construction produces a 
hyper-K\"{a}hler structure. 

From 1970 to 1975, the study of $3$-Sasakian geometry was mainly investigated by the Japanese school. In 1971, a remarkable property for our objectives was achieved:  
every $3$-Sasakian manifold is an Einstein space with positive scalar curvature,  \cite{kashiwada}. Despite of that particularly relevant result and others, in Boyer and Galicki's words, {\it ``1975 seems to be the year when $3$-Sasakian manifolds are relegated to an almost complete obscurity which lasted for about 15 years... The authors [Boyer and Galicki] have puzzled over this phenomenon without any sound explanation''} \cite[Chapter 13]{galickiboyer}. Anyway, at the beginning of the nineties, a renewed interest in $3$-Sasakian geometry arose in several areas. Let us briefly recall  some examples which supported this growing interest. The existence of two different Einstein metrics on $3$-Sasakian manifolds was obtained in \cite{BoGaMan98}, but only one is $3$-Sasakian. In the $7$-dimensional case, both Einstein metrics have $G_2$ weak holonomy. Moreover, for $7$-dimensional Riemannian manifolds, the existence of three Killing spinors is equivalent to the existence of a $3$-Sasakian structure,  \cite{FriKath}.  Bearing in mind  that cones over $3$-Sasakian manifolds produce Calabi-Yau manifolds, recent developments on $3$-Sasakian geometry also include the Yang-Mills equations on cones over $3$-Sasakian manifolds (\cite{GeiSpe} and references therein). Finally, in the study of the control system of a $n$-dimensional Riemannian manifold $M$ rolling on the sphere $\mathbb{S}^n$, without twisting or slipping, and under certain assumption, the manifold $M$ can be endowed with a $3$-Sasakian structure, \cite{Chitour}.
Several generalizations of 3-Sasakian structures have been recently studied, as the 3-$(\alpha,\delta)$-Sasaki manifolds, \cite{AgriDileo},
which are 3-Sasakian manifolds when $\alpha=\delta=1$, or the 3-quasi-Sasakian   manifolds \cite{quasi}. In general lines, they deal with geometric structures which are less rigid than the 3-Sasakian ones, which permits to clarify some geometrical properties of the 3-Sasakian manifolds.

We should recall a key difference between Sasakian and $3$-Sasakian manifolds. Indeed, every Sasakian manifold admits a unique metric connection   with totally skew-symmetric torsion such that all the tensors involved in the Sasakian structure are parallel, \cite{FriIva} (see   Remark~\ref{re_labuena}). Needless to say, this is not the case for a $3$-Sasakian manifold $M^{4n+3}$. In fact, for every $\tau\in \mathbb{S}^{2}$, the corresponding Sasakian structure on $M$ admits  such  connection $\nabla^{ch}_{\tau}$, but they do not coincide for different values of the parameter $\tau\in \mathbb{S}^2$. Therefore, the $3$-Sasakian structure is not {\it parallel} for any metric connection with skew  torsion. Thus, it is natural to ask whether there is a \emph{best} affine metric connection on a $3$-Sasakian manifold. This   question   in a general setting   was posed by Cartan in 1924, \cite{Cartancita}, to look  for a connection adapted to the geometry of the space under consideration. 


In this paper, we will focus on the existence (or not) of remarkable affine connections on $3$-Sasakian manifolds. A specially interesting case is to look for affine metric connections with \textit{skew  torsion}. In fact, these affine connections share geodesics with the Levi-Civita connection. The nice survey \cite{surveyagricola} includes both Mathematical and Physical motivations, as well as a wide variety of examples of affine metric connections with torsion. From our approach, among all the affine metric connections with skew  torsion, the nicest choice has been proposed in \cite{AgriFerr} in the following terms. A triple $(M,g,\nabla)$, where $g$ is a Riemannian metric and $\nabla $ is a metric affine connection with skew  torsion, is said to be  \emph{Einstein with skew  torsion} whenever
\begin{equation*}\small{
\mathrm{Sym}(\mathrm{Ric}^{\nabla})=\frac{s^{\nabla}}{\dim M}\, g,}
\end{equation*}
where $\mathrm{Sym}(\mathrm{Ric}^{\nabla})$ denotes the symmetric part of the Ricci tensor and $s^\nabla$ is the scalar curvature of $\nabla$ (Section \ref{seccion2}). Clearly, this notion is a wide generalization of the usual Einstein spaces. The main purpose of this paper is to determine when there are metric affine connections $\nabla$ on a $3$-Sasakian manifold $M$ such that, with the underlying Riemannian metric $g$, the triple $(M,g,\nabla)$ is Einstein with skew  torsion. To face this problem, we will pay attention to invariant metric affine connections on $3$-Sasakian  homogeneous manifolds. At this point, it is a remarkable fact that every $(4n+3)$-dimensional compact regular $3$-Sasakian manifold with $n< 4$ is homogeneous. Moreover, it has been conjectured that every regular $3$-Sasakian manifold must be homogeneous, \cite[p.~498]{galickiboyer}.


The classification of $3$-Sasakian homogeneous manifolds  
(see details in \cite[Theorems~13.4.6 and 13.4.7]{galickiboyer}) shows four families and five exceptional cases, as follows
$$
\begin{array}{c}
\frac{\Sp(n+1)}{\Sp(n)},\quad
\frac{\Sp(n+1)}{\Sp(n)\times\mathbb Z_2},\quad
\frac{\SU(m)}{S(\mathrm{U}(m-2)\times \mathrm{U}(1))},\quad
\frac{\SO(k )}{\SO(k-4)\times \Sp(1)},\vspace{3pt}\\
\frac{G_2 }{\Sp(1) },\quad  
\frac{F_4 }{\Sp(3) },\quad
\frac{E_6 }{\SU(6) },\quad  
\frac{E_7 }{\mathrm{Spin}(12) },\quad 
\frac{ E_8}{ E_7},
\end{array}
$$
for $n\ge0$, $m\ge3$ and $k\ge7$.
In particular, there is a one-to-one correspondence between   compact simple Lie algebras and   simply-connected $3$-Sasakian homogeneous manifolds. In some sense, $3$-Sasakian geometry seems to be an interesting geometric point of view to approach  {the} compact simple Lie algebras, specially  the exceptional ones. 

Taking into account that every $3$-Sasakian homogeneous manifold $M=G/H$ admits a reductive decomposition (though not naturally), the space of invariant affine connections on $M$ can be described in algebraical terms. In fact, in the classical paper \cite{teoNomizu}, starting from a fixed reductive decomposition $\mathfrak{g}=\mathfrak{h}\oplus \mathfrak{m}$ of the Lie algebra $\mathfrak{g}$ of $G$, Katsumi Nomizu established a very fruitful one-to-one correspondence between the set of all invariant connections on $M$ and the set of all bilinear functions $\alpha$ on $\mathfrak{m} $ with values in $\mathfrak{m}$ which are invariant by $\mathrm{Ad}(H)$. We will extensively use this correspondence for our purposes.


The paper is organized as follows. Section~\ref{seccion2} is devoted to introducing the basic definitions and properties of $3$-Sasakian manifolds in order to fix the notations. This section is mainly indebted to \cite[Chapter 13]{galickiboyer}. In particular, we have adopted the {\it very geometric} definition of $3$-Sasakian structure ${\estS}=\{\xi_{\tau} ,\eta_{\tau}, \varphi_{\tau}\}_{\tau\in \mathbb{S}^{2}}$ as in 
\cite[Definition 13.1.8]{galickiboyer}, where ${\estS}$ is a family of Sasakian structures parametrized  on the points of the $2$-dimensional sphere $\mathbb{S}^2$. Thus, the compatibility conditions on the family of Sasakian structures reduce to
the following    
$$
g(\xi_{\tau}, \xi_{\tau '})=\tau\cdot \tau'  \quad \mathrm{and}\quad \quad [\xi_\tau, \xi_\tau']=2 \xi_{\tau \times \tau'},
$$
 where  \lq\lq$ \cdot $\rq\rq\ and \lq\lq$\times$\rq\rq\ are   the standard inner and cross products in $\mathbb{R}^{3}$. 
 This section also includes the statement of the classification theorem for 
$3$-Sasakian homogeneous manifolds (Theorem~\ref{th_lasSashomogeneas}) and several basic facts on Riemann-Cartan manifolds and Einstein with skew  torsion connections. Riemann-Cartan manifolds can be seen as a generalization of the Riemann manifolds where the Levi-Civita connection is replaced with a metric affine connection with  nonvanishing torsion tensor, in general. 

Section~\ref{se_3} continues to give the necessary  background, since we have tried to keep the paper as self-contained as possible.
 Nomizu\rq{s} Theorem on invariant connections on homogeneous reductive spaces is stated 
in the way that we will use it (Theorem~\ref{nomizu}).  Several sets of invariant connections are algebraically characterized. We have enclosed a technical but powerful lemma, which allows us to properly handle covariant derivatives with respect to invariant connections; and several {\it ad hoc} algebraical  lemmata, in order to compute later the dimensions of the different spaces of invariant connections we are interested in. 

   Section~\ref{homogeneas} is, in a sense, a Lie-algebraic look at the classification by Boyer-Galicki, which is used for applying Nomizu's Theorem to any 3-Sasakian homogeneous manifold.  In  Theorem~\ref{le_laestructura!}, we exhibit an explicit and complete construction of  all the $3$-Sasakian homogeneous manifolds $M=G/H$ in a unified way. As far as we know, this   construction was not available in the literature. In particular, the Riemannian metric $g$ and the multiplication map $\alpha^{g}$ corresponding to the Levi-Civita connection via Nomizu's Theorem are given. Up to a factor, the explicit description of the metrics on $3$-Sasakian manifolds was given in \cite{bialowski} (see Remark~\ref{re_lametrica}). The reductive decompositions $\mathfrak{g}=\mathfrak{h} \oplus \mathfrak{m}$ are studied in order to determine the number of free parameters involved in the spaces of invariant connections, metric invariant connections and metric with skew  torsion invariant connections, because these numbers coincide, respectively, with  the dimensions of the following spaces of $\mathfrak{h}$-module homomorphisms 
$$
\small
\hom_{\mathfrak h}(\mathfrak m\otimes \mathfrak m,\mathfrak m),\quad
\hom_{\mathfrak h}(\mathfrak m,\mathfrak m\wedge\mathfrak m),\quad
\hom_{\mathfrak h}(\wedge^3\mathfrak m,\R),\large
$$
when $H$ is connected. Since they are computed by complexification, it is necessary to know in detail the    decomposition of  the  $\hh^\C$-module $\mm^\C$ as a sum of irreducible submodules, in each case. It is a remarkable fact that these dimensions do not depend on the particular choice of the invariant metric. In particular, these computations can be applied to every {\it canonical variation along the fibers} of the metric $g$ on each $3$-Sasakian homogeneous manifold. 
Next, we describe all $3$-Sasakian homogeneous manifolds case-by-case, in order to compute the dimensions of   the related sets of homomorphisms. 
A surprising fact on these vector spaces happens: except for the family $\SU(m)/S(\mathrm U(m-2)\times \mathrm U(1))$, and for dimension at least $7$, these dimensions are always the same, 
$$
\dim\hom_{\mathfrak h}(\mathfrak m\otimes \mathfrak m,\mathfrak m)=63,\,\,\dim\hom_{\mathfrak h}(\mathfrak m,\mathfrak m\wedge\mathfrak m)=30, \,\, \dim\hom_{\mathfrak h}(\wedge^3\mathfrak m,\R)=10.\large
$$ 
That is, these numbers do not depend on the concrete example we are dealing with. This observation was one of the seminal ideas for this paper. In fact, at the beginning, we were only interested in the family of spheres $\mathbb{S}^{4n+3}=\Sp(n+1)/\Sp(n)$ and we realized that for $\mathbb{S}^7$ and $G_{2}/\Sp(1)$ the \textit{numbers} $63$, $30$ and $10$ were the same, \cite{esphomogeneosdeG2,DraperPalomoPadge}. 
This feature made us think that the 3-Sasakian structure was behind this numerical coincidence. 
Then, we found  that several remarkable   algebraical properties of the reductive decompositions $\mathfrak{g}=\mathfrak{h} \oplus \mathfrak{m}$ are common for all $3$-Sasakian homogeneous manifolds   (see Eqs.~\eqref{eq_imp1} and \eqref{eq_imp2}, with Remark~\ref{remark6}).  Section~\ref{homogeneas}  finishes by analysing the remaining family, $G=\SU(m)$  with $m\geq 3$, and we found that
$$
\dim\hom_{\mathfrak h}(\mathfrak m\otimes \mathfrak m,\mathfrak m)=99,\,\,\dim\hom_{\mathfrak h}(\mathfrak m,\mathfrak m\wedge\mathfrak m)=45, \,\, \dim\hom_{\mathfrak h}(\wedge^3\mathfrak m,\R)=13.\large
$$
Thus,  
the \emph{exceptional} cases now are the manifolds in the family $\SU(m)/S(\mathrm U(m-2)\times \mathrm U(1))$.
The concrete computations made to determine the pieces of $\mm^\C$ (Section~\ref{su}) shed light on  the new invariant torsion tensors.

Next, the main aim of  Section~\ref{affine} is to read the above computations from a more geometrical point of view, and obtain further consequences on different types of invariant  metric connections with skew  torsion. Section~\ref{sec_invariantes} provides  explicit expressions for all the invariant  metric connections with skew  torsion  on $3$-Sasakian homogeneous manifolds in Corollary~\ref{co_lasskew}. We study Einstein connections with skew  torsion in Section~\ref{sec_Einstein}. By means of  Proposition~\ref{prop_Sys}, where we compute the symmetric part of the Ricci tensor, we obtain our main result (Theorem~\ref{dimn=1}), namely, 
\begin{quote}\it 
Let $M=G/H$ be a $(4n+3)$-dimensional $3$-Sasakian homogeneous manifold ($n\ne 0$).  
Assume that $M$ admits an    invariant metric connection $\nabla$ such that $(M,g, \nabla)$ is Einstein with  nonzero skew  torsion. Then $n=1$, that is, either $M=\mathbb{S}^{7}$, or $M=\mathbb{R}P^{7}$, or $M$ is the Aloff-Wallach space 
$\mathfrak{W}^{7}_{1,1}=\SU(3)/\mathrm U(1)$  \cite{AloWa}. Moreover, for $\mathbb{S}^7$ (and $\mathbb{R}P^{7}$), the set of such connections is parametrized by two copies of the conformal linear transformation group of the Euclidean space, while for $\mathfrak{W}^{7}_{1,1}$, by (the bigger set of) two copies of  nonzero elements in
$$
\Big\{ (c,B)\in \mathbb{R}^{3}\times \mathcal{M}_{3}(\mathbb{R}): BB^{t}+ cc^{t}\in \mathbb{R} {I}_{3}, \,\, c^{t} B=0\Big\}.
$$ 
\end{quote}
As far as we know, there are no precedents in the literature of homogeneous manifolds in which the set of Einstein connections with skew  torsion is so big and structured. 

As already pointed out, the Levi-Civita connection $\nabla^g$ is not {\it adapted} to the $3$-Sasakian structure in the sense that $\nabla^{g}\xi_{\tau}\neq 0$ for all  the Reeb vector fields $\xi_{\tau}$. Besides,
the $3$-Sasakian structure is not {\it parallel} for any metric connection with skew  torsion. Then we look in Section~\ref{distinguida} for a \emph{nontrivial} invariant affine  connection $\nabla$ (that is, $\nabla$ has nonvanishing torsion tensor), Einstein with skew  torsion and such that the Reeb vector fields are $\nabla$-parallel.  We show in Theorem~\ref{new} that there exists a well-adapted one.
\begin{quote}\it 
Let $M$ be a $(4n+3)$-dimensional $3$-Sasakian manifold ($n\ne0$) with Reeb vector fields $\{\xi_{\tau}\}_{\tau \in \mathbb{S}^2}$. Then, there exists an affine connection $\nabla^{\estS}$ on $M$ such that $\nabla^{\estS} \xi_{\tau}=0$ for any $\tau$, with skew  torsion  given by
$$
\omega_{_{\nabla^{\estS}}}= \eta_{1}\wedge d\eta_{1}+   \eta_{2}\wedge d\eta_{2}+ \eta_{3}\wedge d\eta_{3} +4  \eta_{1} \wedge \eta_{2}\wedge \eta_{3}.
$$
\end{quote}
\normalfont   
Here $\eta_{k}=g(\xi_{k}, -)$, for $\{\xi_{k}\}_{k=1,2,3}$   an orthonormal basis of Reeb vector fields with $[\xi_{1},\xi_{2}]=2\xi_{3}$.  
\begin{quote}\it
Moreover, if $n=1$, then $(M,g , \nabla^{\estS})$ is Einstein with skew  torsion; and when $M$ is homogeneous, such  affine connection is unique among the invariant ones. 
\end{quote}
\normalfont 
The affine connection $\nabla^{\estS}$  
is not Einstein with skew  torsion for $n> 1$, but it has interesting properties. In fact, its Ricci tensor   is always symmetric and satisfies
\[
\mathrm{Ric}^{\nabla}=\alpha\,   g+ \beta\,\sum_{k=1}^{3}\eta_{k}\otimes \eta_{k}
\]
for $\alpha,\beta\in \mathbb{R}$. Recall the following definition   \cite[Definition 11.1.1]{galickiboyer}: A contact metric structure $ \{\xi,\eta,\varphi,g\}$ on $M$ is said $\eta$-Einstein if there are constants $\alpha,\beta$ such that $\textrm{Ric}^g=\alpha g+\beta\eta\otimes\eta$.
Observe that the Ricci tensors of the Robertson-Walker metrics satisfy a similar property \cite[12.10]{ONeill}.
In addition, a similar condition was considered for real hypersurfaces in complex space forms in 
\cite{CecilRyan}, \cite{IveyRyan},  \cite{KimRyan}  and \cite{Montiel}, and for real hypersurfaces in quaternionic space forms in \cite{MartinezPerez} and \cite{OrtegaPerez}.
Motivated for these properties, we introduce in Section~\ref{sec_S-Einstein} the notion of ${\estS}$-Einstein affine connection on a manifold with a $3$-Sasakian structure ${\estS}$. 
In order to find these connections, we first find out when the Ricci tensor is symmetric in Corollary~\ref{co_ricsime}, and then,  in Theorem~\ref{segundo}, we describe with precision the set of ${\estS}$-Einstein   invariant affine connections on any  $3$-Sasakian homogeneous manifold. As a consequence of Theorem~\ref{segundo}, any 3-Sasakian manifold (without conditions on the dimension $\ge7$) possesses a great amount of ${\estS}$-Einstein connections, being $\nabla^{\estS}$ distinguished among them.

Finally, Section~\ref{sec_paralela} is devoted to finding the connections with parallel skew  torsion. 
For a metric affine connection $\nabla$ with totally skew-symme\-tric torsion $T$, the condition $\nabla T=0$   provides remarkable properties on the curvature and Ricci tensors of $\nabla$ \cite[Appendix A]{surveyagricola}. For example, the curvature tensor is pair symmetric and satisfies the second Bianchi identity,  similarly to the well-known case of the Levi-Civita connection. From a physical point of view,  $\nabla T=0$ simplifies the equations of superstring theory (see \cite{surveyagricola} and references therein).   It is interesting to point out the recent study of the local structure of Riemannian manifolds with nonvanishing parallel  skew  torsion
in \cite{Cleyton}. Our main result concerning parallel skew  torsion is Theorem~\ref{th_paralela}.
\begin{quote}\it 
Let $M$ be a $(4n+3)$-dimensional $3$-Sasakian homogeneous manifold ($n\geq 1$). The only invariant metric connections with parallel skew  torsion are the Levi-Civita connection, the family of the characteristic connections  $\nabla^{ch}_{\tau}$   of each Sasakian structure $\{\xi_{\tau} ,\eta_{\tau}, \varphi_{\tau}\}$, $\tau\in \mathbb{S}^2$ (Remark~\ref{re_labuena}), the canonical connection $\nabla^{c}$ of the    $3$-Sasakian  manifold (Remark~\ref{G2}~iv)), and certain family parame\-trized by  $\R P^2$. The last family only appears when $n=1$.  
\end{quote}

\section{Set up}\label{seccion2}
All the manifolds, maps, tensor fields, etc,   are assumed to be smooth. Let us briefly recall the basic notions on $3$-Sasakian geometry in order to fix some notations. This section is indebted to \cite[Chapter~13]{galickiboyer}.

Among the alternative definitions of Sasakian structure, we take the following one.
Let $(M,g)$ be a Riemannian manifold with Levi-Civita connection $\nabla^g$. The triple ${\estS}=\{\xi, \eta, \varphi\}$ is called a \emph{Sasakian structure} on $(M,g)$ when $\xi\in \mathfrak{X}(M)$ is a unit Killing vector field, $\varphi$ is the endomorphism field given by $\varphi (X)=- \nabla^g_{X}\xi$ for all $X\in \mathfrak{X}(M)$, $\eta$ is the $1$-form on $M$ metrically equivalent to $\xi$, i.e., $\eta(X)=g(X, \xi)$, and the following condition is satisfied
$$
(\nabla^g_{X}\varphi)(Y)=g(X,Y)\xi-\eta (Y) X 
$$
for $X, Y\in \mathfrak{X}(M)$.
The vector field $\xi$ and the $1$-form $\eta$ are called the Reeb vector field and the characteristic $1$-form of the Sasakian structure, respectively. A \emph{Sasakian manifold} is a Riemannian manifold  $(M,g)$ endowed with a fixed Sasakian structure ${\estS}$. The following consequences of the definition allow to handle properly  Sasakian manifolds (see details in \cite{galickiboyer}),
$$
\varphi^{2}=- \mathrm{id}+ \eta \otimes \xi,\quad g(\varphi X, \varphi Y)=g(X,Y)-\eta (X)\eta(Y), \quad \varphi(\xi)=0, \quad \eta (\varphi)=0,
$$
$$
g(X, \varphi(Y))+g(\varphi(X), Y)=0 , \quad d\eta (X,Y)=2 g(X, \varphi(Y)),  
$$
{for all } $X,Y\in \mathfrak{X}(M)$.\footnote{Our convention is $d\eta (X,Y)=X(\eta(Y))-Y(\eta(X))-\eta([X,Y])$.}

A \emph{$3$-Sasakian structure} on $(M,g)$ will be a family of Sasakian structures  ${\estS}=\{\xi_{\tau} ,\eta_{\tau}, \varphi_{\tau}\}_{\tau\in \mathbb{S}^{2}}$ on $(M,g)$ parametri\-zed by points $\tau\in \mathbb{S}^{2}$ on the $2$-dimensional unit sphere and such that, for $\tau, \tau'\in \mathbb{S}^{2}$, the following compatibility conditions hold 
\begin{equation}\label{eq_compatibilityconditions}
g(\xi_{\tau}, \xi_{\tau '})=\tau\cdot \tau'  \quad \mathrm{and}\quad \quad [\xi_\tau, \xi_{\tau'}]=2 \xi_{\tau \times \tau'},
\end{equation}
where  \lq\lq$ \cdot $\rq\rq\ and \lq\lq$\times$\rq\rq\ are the standard inner and cross products in $\mathbb{R}^{3}$,  and we extend the Reeb vector fields from $ \mathbb{S}^{2}$ to $\mathbb{R}^{3}$ by linearity.
The compatibility conditions imply, for all $\tau, \tau'\in \mathbb{S}^{2}$,
$$
\begin{array}{c}
\varphi_{\tau}\circ \varphi_{\tau' }-\eta_{\tau'}\otimes \xi_{\tau }=\varphi_{\tau \times \tau'}- (\tau \cdot \tau')\mathrm{id}, \\
\varphi_\tau(\xi_{\tau'})=\xi_{\tau \times \tau'}, \qquad \eta_{\tau}\circ\varphi_{\tau'}=\eta_{\tau \times \tau'}.
\end{array}
$$
In order to construct a $3$-Sasakian structure on a Riemannian manifold $(M,g)$, we need only to fix three Sasakian structures ${\estS}_{k}=\{\xi_{k} ,\eta_{k}, \varphi_{k}\}$, for $k=1,2,3$, such that $g(\xi_{i},\xi_{j})=\delta_{ij}$ and
$[\xi_{i},\xi_{j}]=2 \epsilon_{ijk}\xi_{k}$ (here $ \epsilon_{ijk}$ denotes the sign of the permutation). In fact, it is an easy matter {\it to extend} the Reeb vector fields to the whole sphere $\mathbb{S}^2$ by means of the   formula $\xi_{v}=\sum _{k=1}^3 v_{k}\xi_{k}$ for $v=(v_{1}, v_{2}, v_{3})^{t}\in \mathbb{S}^2$. Throughout this text, we will call 3-Sasakian structure both to ${\estS}=\{{\estS}_{\tau}\}_{\tau\in\mathbb{S}^2}$ and to $\{{\estS}_{k}\}_{k=1,2,3}$. The $3$-dimensional {\it fundamental foliation} $\mathcal{F}_{Q}$ is generated by the Reeb vector fields $ \{\xi_{\tau}\}_{\tau \in  \mathbb{S}^{2}}$
and the distribution $ Q=\mathcal{F}_{Q}^{\perp}$ is called the \textit{quaternionic distribution} of $M$.  They are also well-known as the \textit{vertical} distribution  $ {Q}^{\perp}$ and the \textit{horizontal} distribution $ Q$.
 
A remarkable result stated by Kashiwada (cf. \cite{kashiwada}) shows that every $3$-Sasakian manifold of dimension $4n+3$ is Einstein, with Ricci tensor satisfying $\mathrm{Ric}^g=2(2n+1)  g$. 

Given $(M,g,{\estS})$ a $3$-Sasakian manifold, the automorphism group $\mathrm{Aut}(M,g, {\estS})$ is given by
$$
\mathrm{Aut}(M,g, {\estS})=\bigcap_{\tau\in \mathbb{S}^2}\mathrm{Aut}(M,g, {\estS}_\tau),
$$
where $\mathrm{Aut}(M,g, {\estS}_\tau)\subset \mathrm{Iso}(M,g)$ is the subgroup of isometries $f$ of $M$ which preserve $\xi_{\tau}$ (and then $\eta_{\tau}$ and $\varphi_{\tau}$). That is, $f_{*}( \xi_{\tau}(p))=\xi_{\tau}(f(p))$ for every $p\in M$. We will mainly focus on $3$-Sasakian \textit{homogeneous} manifolds, that is, $3$-Sasakian manifolds $M$  whose  $\mathrm{Aut}(M,g, {\estS})$ acts transitively on $M$. If $(M,g, {\estS})$ is a $3$-Sasakian homogeneous manifold, then the orbit space $M/\mathcal{F}_{Q}$ is a quaternionic K\"ahler homogeneous manifold \cite[Prop. 13.4.5]{galickiboyer}  (see Section 4.1 below for more details on this fact). Then,  the next classification theorem is achieved  in \cite{Ale}.

\begin{theorem}
\label{th_lasSashomogeneas}
Any $3$-Sasakian homogeneous manifold is one of the following coset manifolds:
$$
\frac{\Sp(n+1)}{\Sp(n)},\quad
\frac{\Sp(n+1)}{\Sp(n)\times\mathbb Z_2},\quad
\frac{\SU(m)}{S(\mathrm{U}(m-2)\times \mathrm{U}(1))},\quad
\frac{\SO(k )}{\SO(k-4)\times \Sp(1)},$$
$$\frac{G_2 }{\Sp(1) },\quad  
\frac{F_4 }{\Sp(3) },\quad
\frac{E_6 }{\SU(6) },\quad  
\frac{E_7 }{\mathrm{Spin}(12) },\quad 
\frac{ E_8}{ E_7},
$$
for $n\ge0$, $m\ge3$ and $k\ge7$ ($\Sp(0)$ denoting the trivial group).
\end{theorem}

For the explicit description of the metrics of the 3-Sasakian homogeneous manifolds, see Theorem~\ref{le_laestructura!} below. 

Observe  that $\dim G/H=4n+3$ in all the cases, with $n=m-2$ for the quotients of the unitary groups, $n=k-4$ in the orthogonal case, and $n=2$, $7$, $10$, $16$ and $28$, respectively, in the exceptional cases.
The only $3$-dimensional $3$-Sasakian homogeneous manifolds are the sphere $\esf^3$ and the projective space $\mathbb{R}P^3$, which behave different from the rest, because their 
horizontal distributions reduce to $\{0\}$. 

As a consequence of Theorem~\ref{th_lasSashomogeneas}, all $3$-Sasakian homogeneous manifolds are simply-connec\-ted except  the real projective spaces $\mathbb{R}{P}^{4n+3}\simeq\frac{\Sp(n+1)}{\Sp(n)\times\mathbb Z_2}$. In particular, there is a one-to-one correspondence between   compact simple Lie algebras and   simply-connected $3$-Sasakian homogeneous manifolds. 


Let us recall the notion of \textit{Riemann-Cartan manifold}, since this is the second main topic of this paper (cf.~\cite{surveyagricola}). A Riemann-Cartan manifold is a triple $(M,g,\nabla)$, where $(M,g)$ is a Riemannian manifold and $\nabla$ is a \textit{metric} affine connection, that is to say, $\nabla g=0$. It is also commonly said that $\nabla$ is \textit{compatible} with the metric $g$. The \textit{torsion} tensor field of $\nabla$ is defined as usual by $T^{\nabla}(X,Y)=\nabla_{X}Y-\nabla_{Y}X-[X,Y]$, for $X,Y\in \mathfrak{X}(M)$, and it does not vanish in general. Clearly, these manifolds can be seen as a generalization of Riemannian manifolds, since the considered metric affine connection may be different from the Levi-Civita connection $\nabla^{g}$, which is characterized by the condition $T^{\nabla^{g}}=0$. 
For $(M,g,\nabla)$ a Riemann-Cartan manifold, we set
\begin{equation}\label{tos}
\omega_{_\nabla}(X,Y,Z):=g(T^{\nabla}(X,Y),Z),
\end{equation}
for $X,Y,Z\in \mathfrak{X}(M)$. Then, the (metric) connection $\nabla$ is said
to have \emph{totally skew-symmetric torsion} or briefly, \emph{skew  torsion}, if $\omega_{_\nabla}$ defines a differential $3$-form on $M$. This  characterizes the remarkable fact that $\nabla$ and $\nabla^g$ share their (parametrized) geodesics.

\begin{remark}\label{re_labuena}
\normalfont 
For every  Sasakian structure $\{\xi, \eta, \varphi\}$ on $(M,g)$, there is a unique metric connection with totally skew-symmetric torsion $\nabla^{ch}$ such that  $\nabla^{ch} \xi=0, \nabla^{ch} \eta=0$ and $\nabla^{ch} \varphi=0$, \cite{FriIva}.
An explicit formula for this connection is given by
\[
g(\nabla^{ch}_{X}Y, Z)=g(\nabla^{g}_{X}Y, Z)+ \frac{1}{2}\eta \wedge d\eta(X,Y,Z).
\] 
Then, in a 3-Sasakian manifold, each Sasakian structure $\{\xi_{\tau} ,\eta_{\tau}, \varphi_{\tau}\}$ has a distinguished metric connection with totally skew-symmetric torsion as above, denoted by $\nabla_{\tau}^{ch}$ and called the \emph{characteristic connection} of the corresponding Sasakian structure. It has been   studied in \cite{new}. 
\end{remark}  

For any metric affine connection $\nabla$, we define the difference $(1,2)$-tensor $D=\nabla-\nabla^g$ on $M$. The torsion $T^\nabla$ satisfies $T^\nabla(X,Y)=D(X,Y)-D(Y,X)$, for any $X,Y\in\mathfrak{X}(M)$.
The connections $\nabla$ and $\nabla^g$ share the same geodesics if, and only if,  the difference tensor $D$ is skew-symmetric. In such case, we have $$\nabla=\nabla^g+\frac12T^\nabla.$$

Given $\nabla$ an affine connection on an $n$-dimensional manifold $M$, it is always possible to compute its Ricci tensor, but it is not symmetric in general. Thus, a Riemann-Cartan manifold $(M,g,\nabla)$ is said to be \emph{Einstein with skew  torsion} (cf. \cite{AgriFerr}) if the metric affine connection $\nabla$ has  totally skew-symmetric torsion and satisfies 
\begin{equation} \label{einstein}
\mathrm{Sym}(\mathrm{Ric}^{\nabla})=\frac{s^{\nabla}}{\dim M}\, g,
\end{equation}
where $s^{\nabla}$ is the corresponding scalar curvature given by $s^{\nabla}=\sum_{i,j=1}^{n}g (R^{\nabla}(e_{i}, e_{j})e_{j}, e_{i})$ for an orthonormal basis  and, following \cite{surveyagricola}, $\mathrm{Sym}(\mathrm{Ric}^{\nabla}) $ denotes the symmetric part of the Ricci (curvature) tensor of $\nabla$.    
The metric connections   
such that $(M,g,\nabla)$ is Einstein with skew  torsion are the critical points  of a variational problem which involves the scalar curvature of the Levi-Civita connection $\nabla^g$ and the torsion of $\nabla$ (see details in \cite{AgriFerr}). For brevity, we will also say that $\nabla$ is an Einstein with skew  torsion  {(affine)} connection.

We recall from \cite{FriIva} some curvature identities 
on Riemann-Cartan manifolds with totally skew-symmetric torsion.
Let $\s\in \mathcal{T}^{(0,2)}(M)$ be the tensor given at $p\in M$ by
\begin{equation}\label{eq_defS}
\s(X,Y)_p:=\sum_{j=1}^{n}g(T^{\nabla}(e_{j},X_p),T^{\nabla}(e_{j},Y_p)),
\end{equation}
where $\{e_{1},\ldots,e_{n}\}$ is an orthonormal basis of $T_pM$ and $X,Y\in \mathfrak{X}(M)$. The Ricci tensor of the Levi-Civita connection, denoted by $\mathrm{Ric}^g$, and  
the Ricci tensor of $\nabla$ are related by
\[
 \mathrm{Ric}^{\nabla} =\mathrm{Ric}^{g}-\frac{1}{4}\s+\frac12\,\textrm{div}(T^{\nabla}),  
\]
where $\textrm{div}(T^{\nabla}) $ denotes the divergence of the torsion form given by
$$
\textrm{div}(T^{\nabla})(X,Y)=\sum_{i=1}^{n}\left(\nabla^{g}_{e_{i}}\omega_{_{\nabla}}\right)(X,Y,e_{i}). 
$$
In particular, since $\mathrm{Ric}^{g}$ and $\s$ are symmetric tensors  but $\textrm{div}(T^{\nabla})$ is skew-symmetric, we get 
\begin{equation}\label{formulicas}
\mathrm{Sym}(\mathrm{Ric}^{\nabla})=\mathrm{Ric}^{g}-\frac{1}{4}\s,\qquad 
\mathrm{Skew}(\mathrm{Ric}^{\nabla})=\frac12\,\textrm{div}(T^{\nabla}).
\end{equation}
Then,  $ \mathrm{Ric}^{\nabla}$ is symmetric if, and only if, $\textrm{div}(T^{\nabla})=0$. For instance, this holds if the torsion is $\nabla$-parallel.
\begin{remark}
{\rm Our conventions on the signs are different from \cite{FriIva}. For example, we take $\textrm{div}(T^{\nabla})=-\delta^{g}(\omega_{_{\nabla}})$, where 
$\delta^g(\omega_{_{\nabla}})=-\sum_{i=1}^n \iota_{ e_{i}}\nabla^{g}_{e_{i}}\omega_{_{\nabla}}$ denotes 
the codifferential of the $3$-form $\omega_{_{\nabla}}$. Also, let us recall that $\delta^{g}(\omega_{_{\nabla}})=\delta^{\nabla}(\omega_{_{\nabla}})$ under our assumption that $\nabla$ has  
skew  torsion  (\cite[p.~305-306]{FriIva}).}
\end{remark}

\section{Nomizu's Theorem on Invariant Connections and Algebraical Tools} \label{se_3}

The study of invariant affine connections in homogeneous spaces is a \emph{simple} matter in the reductive cases thanks to  Nomizu's Theorem (cf. \cite{teoNomizu}), which allows to translate the geometric problem to an algebraic setting. Roughly speaking, it establishes a bijective correspondence between these connections and certain homomorphisms of modules. This allows to compute the \emph{size} of the set of invariant connections. 

Let $G$ be a Lie group acting transitively on a manifold $M$. We write a dot to denote the action of $G$ on $M$ and so, for $\sigma \in G$, the left translation by $\sigma$ will be given by $\tau_{\sigma}(p)=\sigma \cdot p$ for all $p\in M$.
For each $\sigma \in G$ and $X\in \mathfrak{X}(M)$, the vector field $\tau_{\sigma}(X)\in \mathfrak{X}(M)$ is defined at each $p\in M$ by
\[
 (\tau_{\sigma}(X) )_p:=(\tau_{\sigma})_{*}(X_{\sigma^{-1}\cdot p}).
\]
An affine connection $\nabla$ on $M$ is said to be $G$-\textit{invariant}  if, for each $\sigma \in G$ and for all $X,Y\in\mathfrak{X}(M)$,
$$
\tau_{\sigma}(\nabla_{X}Y)=\nabla_{_{\tau_{\sigma}(X)}}\tau_{\sigma}(Y).
$$
Let $H$ be the isotropy subgroup at a fixed point $o\in M$, so that there exists a diffeomorphism between $M$ and $G/H$. The homogeneous space $M=G/H$ is said to be \textit{reductive} if the Lie algebra $\mathfrak{g}$ of $G$ admits a vector space decomposition
\begin{equation}\label{eq_descomposicionreductiva}
\mathfrak{g}=\mathfrak{h}\oplus\mathfrak{m},
\end{equation}
for $\mathfrak{h}$ the Lie algebra of $H$ and $\mathfrak{m}$ an $\mathrm{Ad}(H)$-invariant subspace (i.e., $\mathrm{Ad}(H)(\mathfrak{m})\subset \mathfrak{m}$). In this case, $\mathfrak{g}=\mathfrak{h}\oplus\mathfrak{m}$ is called a \emph{reductive decomposition} of $\mathfrak{g}$. The condition $\mathrm{Ad}(H)(\mathfrak{m})\subset \mathfrak{m}$ implies that $[\mathfrak{h},\mathfrak{m}]\subset\mathfrak{m}$, and both are equivalent conditions when $H$ is connected. The differential map $\pi_{*}$ of the projection $\pi\colon G\to M=G/H$ gives a linear isomorphism
$(\pi_{*})_e\vert_{\mathfrak{m}}\colon\mathfrak{m}\to T_{o}M$, where $o=\pi(e)$. Nomizu's Theorem,   \cite{teoNomizu}, can be stated as follows:

\begin{theorem}\label{nomizu}
Let $G/H$ be a reductive homogeneous space with a fixed reductive decomposition as in  (\ref{eq_descomposicionreductiva}). Then, there is a one-to-one correspondence between the set of   $G$-invariant linear connections $\nabla$ on $G/H$ and the vector space of bilinear maps $\alpha\colon \mathfrak{m}\times\mathfrak{m}\to\mathfrak{m}$ such that $\mathrm{Ad}(H)\subset\mathrm{Aut}(\mathfrak{m},\alpha)$.
In case $H$ is connected, this condition $\mathrm{Ad}(H)\subset\mathrm{Aut}(\mathfrak{m},\alpha)$ is equivalent to
\begin{equation}\label{casoconexo}
[A,\alpha(X,Y)]=\alpha([A,X],Y)+\alpha(X,[A,Y])
\end{equation}
for all $X,Y\in\mathfrak{m}$ and $A\in\mathfrak{h}$ (i.e., $\mathrm{ad}(\mathfrak{h})\subset \mathrm{Der}(\mm,\alpha)$).
\end{theorem}

We should bear in mind that, under the above identification $(\pi_{*})_e\vert_{\mathfrak{m}}$, the correspondence given by Nomizu's Theorem works as follows
\begin{equation}\label{eq_elalfaasociado}
\nabla \longmapsto \alpha_{_{\nabla}}(X,Y)=\nabla_{X_{o}}Y- [X,Y]_{o}, \quad X,Y\in \mathfrak{m},
\end{equation} 
for $\nabla$ a  $G$-invariant affine connection on $G/H$ and $\alpha_{_{\nabla}}$  the associated bilinear map on $\mathfrak{m}$ (the key point is that  $L(X,Y)= \nabla_XY- [X,Y]$ defines a tensor).

The reductive complement $\mathfrak{m}$ is an $\mathfrak{h}$-module and so, in a natural way, the tensor product $\mathfrak{m}\otimes\mathfrak{m}$ is also  an $\mathfrak{h}$-module for $A\cdot X\otimes Y= [A,X]\otimes Y + X\otimes [A,Y]$. Note that any bilinear map $\alpha$ as in \eqref{casoconexo} allows to construct a homomorphism of $\mathfrak{h}$-modules
$\widetilde\alpha\colon \mathfrak{m}\otimes\mathfrak{m}\to\mathfrak{m}$ by means of $\widetilde\alpha(X\otimes Y) =\alpha(X,Y)$ for all $X,Y\in\mathfrak{m}$, and conversely. We will usually identify $\alpha$ and   $\widetilde\alpha$. So, Nomizu's Theorem can be read in the following terms.

\begin{corollary} \label{atilde} 
Let $G/H$ be a reductive homogeneous space with a fixed reductive decomposition as in  (\ref{eq_descomposicionreductiva}) and $H$ connected. Then, there is a bijective correspondence between the 
set of  $G$-invariant affine connections on $G/H$
and the vector space  
$\mathrm{Hom}_{\mathfrak{h}}(\mathfrak{m}\otimes\mathfrak{m},\mathfrak{m})$.
\end{corollary}
 
From now on, we extensively  use the identification $(\pi_{*})_e\vert_{\mathfrak{m}}\colon\mathfrak{m}\to T_{o}M$.
Given $\nabla$ a  $G$-invariant affine connection on $G/H$ 
and $\alpha_{_{\nabla}}\ $ the associated bilinear map as in \eqref{eq_elalfaasociado},   the torsion and curvature tensors of the $G$-invariant affine connection $\nabla$   are   computed in \cite{teoNomizu} as follows:
\begin{align*}
T^{\nabla}(X,Y)= &\ \alpha_{_{\nabla}} (X,Y)-\alpha_{_{\nabla}} (Y,X)-[X,Y]_\mathfrak{m}, \\
\label{cur}
R^{\nabla}(X,Y)Z=& \ \nonumber 
 \alpha_{_{\nabla}} (X,\alpha_{_{\nabla}} (Y, Z))-\alpha_{_{\nabla}} (Y,\alpha_{_{\nabla}} (X, Z)) 
  -\alpha_{_{\nabla}} ( [X , Y]_{\mathfrak{m}}, Z)-[[X , Y]_{\mathfrak{h}} , Z],
\end{align*}
for any $X, Y, Z\in \mathfrak{m}$, where   $[\ ,\ ]_{\mathfrak{h}}$ and   $[\ ,\ ]_{\mathfrak{m}}$ denote the composition of the bracket
($ [\mathfrak{m},\mathfrak{m}]\subset \mathfrak{g}$) with  the projections $\pi_{\mathfrak{h}}$ and $\pi_{\mathfrak{m}}$  of $\mathfrak{g}= \mathfrak {h}\oplus \mathfrak{m}$ on each factor, respectively. These expressions give $T^{\nabla}$  and $R^{\nabla}$ at the point $o=\pi(e)$, but the invariance permits to recover the whole tensors.

To deal with the covariant derivative of arbitrary tensors at the point $o$ with  respect to an invariant affine connection $\nabla$, we include the  following technical result.

\begin{lemma} \label{le_comoderivar}
Let $M=G/H$ be a reductive homogeneous space with a fixed reductive  decomposition as in  (\ref{eq_descomposicionreductiva}) and $\nabla$ a 
$G$-invariant affine connection. For every  tensor field of type $(1,k)$, $\TT\in \mathcal{T}^{(1,k)}(M)$, the following formula holds
\begin{equation}\label{derivasdeA}
(\nabla_{Z} \TT)(X_{1},...,X_{k})=\alpha_{_{\nabla}}(Z, \TT(X_{1},...,X_{k}))-\sum_{i=1}^{k}\TT(X_{1},..., \alpha_{_{\nabla}}(Z,  X_{i}),...,X_{k}),
\end{equation}
where $Z, X_{1},...,X_{k}\in \mathfrak{m}\approx T_{o}M$.
\end{lemma}  

\begin{proof}
For every $X\in \mathfrak{g}$, let us consider the fundamental vector  field $X^{+}\in \mathfrak{X}(M)$ defined by 
$$
X^{+}_{p}:=\left.\frac{d}{dt}\right \vert_{0}(\mathrm{exp}(tX)\cdot p).
$$
 A direct computation shows
$
X^{+}_{o}=(\pi_{*})_e (X)
$
and therefore $X\in \mathfrak{m}$ corresponds with $X^{+}_{o}$ via the  identification $(\pi_{*})_e\vert_{\mathfrak{m}}$. In these terms, the  bilinear map $\alpha_{_{\nabla}}$ corresponding to $\nabla$ via  Nomizu's Theorem is given by
$$
\alpha_{_{\nabla}}(X,Y)=\nabla_{X}W^{Y}-[X^{+}, W^{Y}]_{o},
$$
for $X\in \mathfrak{m}$ and $W^{Y}\in \mathfrak{X}(M)$   an arbitrary  extension of $Y\in T_{o}M$. The tangent vector $Y\in \mathfrak{m} \approx T_{o}M$ can be extended to the vector field $W^Y$ in a  convenient way. Indeed, we choose the extension $W^Y$ such that $ [X^{+}, W^{Y}]_{o}=0$. For $X=0$, the vector field $X^+$ vanishes  identically and any extension $W^Y$ satisfies the required property.  On the contrary, when $X\neq 0$, we take a local coordinate system $ (x_{1},..., x_{n})$ centered at $o\in M$ and such that $X^{+}_{o}= \frac{\partial}{\partial x_{1}}\vert_{o}$. We have the local coordinate  expressions $X^{+}=\sum f_{i}\frac{\partial }{\partial x_{i}}$ and  $W^{Y}=\sum h_{i}\frac{\partial }{\partial x_{i}}$. Thus, the  conditions $[X^{+}, W^{Y}]_{o}=0$ and $W^{Y}_{o}=Y$ become
$$
\left.\frac{\partial h_{i}}{\partial x_{1}}\right\vert_{o}=\sum_{j=1}^{n} y_{j}\left. \frac{\partial f_{i}}{\partial x_{j}}\right \vert_{o},\quad  h_{i}(o)=y_{i},  \quad (i=1,...,n)
$$
where $Y=\sum_{i=1}^n y_{i}\frac{\partial }{\partial x_{j}}\vert_{o}$.  Now, it is clear how to construct the vector field $W^Y$, and we obtain the  short formula
\begin{equation}\label{alphafacil}
\alpha_{_{\nabla}}(X,Y)=\nabla_{X}W^{Y}
\end{equation}
for our convenient extension $W^Y$.
Finally, the proof of  (\ref{derivasdeA}) is an easy matter  by using (\ref{alphafacil}). Indeed, a straightforward computation  shows
\begin{align*}
(\nabla_{Z} \TT)(X_{1},...,X_{k})&=\nabla_{Z}(W^{\TT(X_{1},...,X_{k})})- \sum_{i=1}^{k}\TT(X_{1},..., \nabla_{Z}W^{X_{i}},...,X_{k})\\
&=\alpha_{_{\nabla}}(Z, \TT(X_{1},...,X_{k}))-\sum_{i=1}^{k}\TT(X_{1},...,  \alpha_{_{\nabla}}(Z, X_{i}),...,X_{k}),
\end{align*}
by using here suitable extensions of the tangent vectors $X_{1},...,X_{k}, \TT(X_{1},...,X_{k})\in T_{o}M$.
\end{proof}

If our reductive homogeneous space $G/H$ is endowed with a $G$-invariant Riemannian metric $g$, the identification of $\mathfrak{m}$ with $T_oM$ via $\pi_{*}$ allows to consider the orthogonal Lie algebra $ \mathfrak{so}(\mathfrak{m},g)$ corresponding to the $\mathrm{Ad}(H)$-invariant  nondegenerate symmetric bilinear map induced by $g$. In particular, we have $\mathrm{ad}(\hh)\subset \mathfrak{so}(\mathfrak{m},g)$, 
so we can consider the
$\hh$-module structure on $\mathfrak{so}(\mathfrak{m},g)$ given by $(A\cdot \psi) (X)=[A,\psi(X)]- \psi([A,X])$. Recall also the usual $\hh$-module structure
on $\Lambda^{2}\mathfrak{m}$ given by $A\cdot (X\wedge Y )=[A,X]\wedge Y + X\wedge [A,Y].$ Then, $\Lambda^{2}\mathfrak{m}$ and $\mathfrak{so}(\mathfrak{m},g)$ are isomorphic $\hh$-modules by means of 
$X\wedge Y \mapsto g(X,-)Y- g(Y,-)X\in\mathfrak{so}(\mathfrak{m},g).$
The outcome is  another version of Nomizu's Theorem (see \cite[Theorem~2.7, Remark~2.8]{DraperGarvinPalomo}).
 \begin{lemma} 
Let $(G/H, g)$ be a reductive homogeneous Riemannian manifold with reductive decomposition $\mathfrak{g}=\mathfrak{h}\oplus\mathfrak{m}$
and $H$ connected. Then, a $G$-inva\-riant affine 
 connection $\nabla$ is metric ($\nabla g=0$) if, and only if,  the related bilinear map $ \alpha_{_{\nabla}}\colon \mathfrak{m}\times\mathfrak{m}\to\mathfrak{m}$ satisfies
$\alpha_{_{\nabla}}(X,-)\in\mathfrak{so}(\mathfrak{m},g)$ for all $X\in\mathfrak{m}$. Therefore, there is a bijective correspondence between the set of connections on $G/H$ which are $G$-invariant and compatible with $g$, and the vector space $\mathrm{Hom}_{\mathfrak{h}}(\mathfrak{m}, \mathfrak{so}(\mathfrak{m},g))$ (or alternatively with $\hom_{\mathfrak{h}}(\mathfrak{m}, \Lambda^2\mm)$).
\end{lemma}

It will be specially useful for our purposes to also set Nomizu's Theorem for the case of invariant metric  connections with skew  torsion on $G/H$. 
Let us denote such set by $\mathcal{C}_{S}(G/H,g)$.
\begin{lemma}  \label{igualdadconjuntos} 
Let $(G/H, g)$ be a reductive homogeneous Riemannian manifold with reductive decomposition $\mathfrak{g}=\mathfrak{h}\oplus\mathfrak{m}$
and $H$ connected. There is a one-to-one correspondence between $\mathcal{C}_{S}(G/H,g)$ and $\mathrm{Hom}_{\mathfrak{h}}(\Lambda^{3} \mathfrak{m},\R)$ ($\R$ as trivial $\hh$-module). 
\end{lemma}
\begin{proof}
Let us consider the map $\Theta: \mathcal{C}_{S}(G/H,g) \rightarrow \mathrm{Hom}_{\mathfrak{h}}(\Lambda^{3} \mathfrak{m}, \R)$ given by
\begin{align*}
\Theta(\nabla)(X\wedge Y \wedge Z)=g(\alpha_{_{\nabla}}(X,Y)-\alpha^g(X,Y),Z),
\end{align*} 
where $\alpha^g=\alpha_{_{\nabla^g}}$ denotes the bilinear map associated with the Levi-Civita connection by Nomizu's Theorem. It is easy to conclude that $\Theta(\nabla)$ is an $\hh$-module homomorphism, since $\alpha_{_{\nabla}}$, $\alpha^g$ and $g$ so are. The injectivity of $\Theta$ follows from    the nondegeneracy of $g$.
Finally, given $\omega\in  \mathrm{Hom}_{\mathfrak{h}}(\wedge^{3} \mathfrak{m}, \R)$, there exists a unique bilinear map $T\colon\mathfrak{m}\times\mathfrak{m}\rightarrow \mathfrak{m}$ such that $\omega(x\wedge y \wedge z)=g(T(x,y),z)$  for any $x,y,z\in\mathfrak{m}$.   We also denote by $\omega$ and $T$ the natural extensions of $\omega$ and $T$ to the whole manifold $G/H$ by using the $G$-invariance. Now,   the invariant connection $\nabla= \nabla^g+(1/2)T$ has torsion $T$,  is compatible with the metric $g$, and satisfies $\Theta(\nabla)=1/2\omega$. \end{proof}    

Our next target will be to compute the dimensions of 
\begin{equation}\label{eq_dimensionesqueridas}
\hom_{\mathfrak h}(\mathfrak m\otimes \mathfrak m,\mathfrak m),\quad
\hom_{\mathfrak h}(\mathfrak m,\mathfrak m\wedge\mathfrak m),\quad
\hom_{\mathfrak h}(\mathfrak m\wedge\mathfrak m\wedge\mathfrak m,\R),
\end{equation}
 for each    reductive decomposition  related to any  3-Sasakian homogeneous manifold $G/H$ in Theorem~\ref{th_lasSashomogeneas}. 
These vector spaces will be useful to describe, respectively, the set of $G$-invariant affine connections on $G/H$, those ones which are besides compatible with an invariant metric, and those ones that have also skew  torsion.
 We need  precise descriptions of such reductive decompositions as well as of the  involved modules, which will be provided case-by-case in the next section. As  mentioned in  the Introduction, except for the case $\mathrm{SU}(m)/ S(\mathrm{U}(m-2)\times \mathrm{U}(1))$, the dimensions of the above three vector spaces turn out to be the same for all $3$-Sasakian homogeneous manifolds  (see Proposition~\ref{pr_descriunificada} for a description of the first space).

We finish this section with some ad-hoc facts involving representation theory of complex Lie algebras.
(A textbook including the algebraic concepts relative to Lie algebras and their representations is \cite{Draper:Humphreysalg}, for instance.)
 These results   will be used in Section~\ref{homogeneas} to compute the   dimensions in \eqref{eq_dimensionesqueridas}. Although well-known, we include them here for the sake of completeness.  As usual, we denote by $pV$, $p\in\mathbb N$, to the module which is a direct sum of $p$ submodules all of them isomorphic to $V$.
\begin{lemma}\label{le_comocontar}
Let  $V$ and $W$ be (finite-dimensional) modules for a complex Lie algebra $L$. 
\begin{itemize}
\item[i)] If $\,V$ and $\,W$ are irreducible, then the vector space  $\hom_L(V,W)$ is   one-dimensional  if $\,V$ and $W$ are isomorphic $L$-modules and it is $0$ otherwise.
\item[ii)] If $\,V\cong (\oplus_{_{i}}p_iV_i)\oplus(\oplus_{_{j}}n_jU_j)$ and $W\cong (\oplus_{_{k}}q_kW_k)\oplus(\oplus_{_{j}}m_jU_j)$, with $V_i$, $W_k$ and $U_j$ (finite-dimensional) irreducible $L$-modules
which are not isomorphic, then
$$\dim\hom_L(V,W)=\sum_{j}n_jm_j.$$
\end{itemize}
\end{lemma}

Schur's lemma gives item i), while item ii) is an immediate consequence. If besides $L$ is semisimple, all finite-dimensional $L$-modules $V$ and $W$ are completely reducible (sum of irreducible submodules) and then we can always apply item ii) to them. In particular, $\dim\hom_L(V,W)=\dim\hom_L( W,V)$.
 So, for complex representations, such dimensions are computed by finding the irreducible submodules with their corresponding multiplicities.   We denote by   $S^nW$ and $\Lambda^nW$   the $n$th symmetric and alternating tensor power of the module $W$, respectively. 

 \begin{lemma}\label{le_contando}
Let  $U$ and $W$ be modules for a complex Lie algebra $L$.
 
\begin{itemize}
\item[a)] $U\otimes U\cong S^2U\oplus \Lambda^2U$.
 \item[b)] 
$
\Lambda^n(U\oplus W)\cong\oplus_{k=0}^n\,\Lambda^{n-k}U\otimes \Lambda^kW\ 
$
and
$
\ S^n(U\oplus W)\cong\oplus_{k=0}^n\,S^{n-k}U\otimes S^kW.  
$
\item[c)] In particular, for $\mathbb C$ the trivial $L$-module and $ W=2U\oplus 3\mathbb C$,
\begin{equation}\label{eq_landa3}
\begin{array}{l}
W\otimes W \cong 4(U\otimes U)\oplus 12 U\oplus 9\C , \\
\Lambda^2W\cong 3\Lambda^2U\oplus S^2U \oplus 6U\oplus 3\mathbb C, \\
\Lambda^3W\cong 2\Lambda^3U\oplus 2( \Lambda^2U\otimes U) \oplus 9\Lambda^2U\oplus 3S^2U\oplus 6U\oplus \mathbb C.
\end{array}
\end{equation}
\end{itemize}
\end{lemma}
We will check in Section~\ref{homogeneas} that for any 3-Sasakian homogeneous manifold  $M$,  the complexification of the $\hh$-module $\mm\cong T_oM$ is always  in the situation of item c). Moreover, under certain technical conditions, we will compute the  required dimensions in \eqref{eq_dimensionesqueridas} in a unified way, as follows.

\begin{corollary}\label{co_contar}
Let $L$ be a complex Lie algebra and $U$ an irreducible nontrivial $L$-module such that   
$\dim\hom_{L}(S^2U,U)=\dim\hom_{L}(   \Lambda^2U, U ) =\dim\hom_{L}(  S^2 U,\C )=0$, and also 
$\dim\hom_{L}(\Lambda^2U,\C)=1$. Then, for the $L$-module $ W=2U\oplus 3\mathbb C$,  we have
$$
\begin{array}{l}
 \dim \hom_{L}(  W\otimes W,W ) =63,\quad
  \dim \hom_{L}(W,\Lambda^2 W)=30. 
\end{array}
$$
If, besides,  neither $ \Lambda^3U$ nor $\Lambda^2U\otimes U$ contains any trivial submodule (that is, a submodule isomorphic to $\C$), then  
$$\dim \hom_{L}(\Lambda^3 W,\mathbb C)=10.$$
\end{corollary}
\begin{proof} Bearing in mind that $W\otimes W\cong4S^2U\oplus 4\Lambda^2U\oplus12U\oplus9\C$, our assumptions imply that $W\otimes W$ contains $4+9=13$ copies of $\C$ and 12 copies of $U$. Therefore,   Lemma~\ref{le_comocontar}\, ii) gives
$\dim \hom_{L}(  W\otimes W,W )=12\cdot 2+ 13\cdot 3 =63$. In a similar way, $\Lambda^2W\cong 3\Lambda^2U\oplus S^2U \oplus 6U\oplus 3\mathbb C$ contains $3+3=6$ copies of $\C$  and 6 copies of $U$, so that 
$\dim \hom_{L}(W,\Lambda^2 W)=6\cdot2 +6\cdot3=30.$
The third dimension is immediately deduced from Lemma~\ref{le_contando}\,c).
\end{proof}

\section{3-Sasakian homogeneous manifolds}\label{homogeneas}

Now we provide  an explicit and detailed construction of all $3$-Sasakian homogeneous manifolds  in algebraic terms.
These precise descriptions will be as  self-contained as possible and will be used throughout the text. The common information is summarized in Theorem~\ref{le_laestructura!}.
The following subsections contain the different reductive decompositions, module decompositions  and the computations of the dimensions in \eqref{eq_dimensionesqueridas}, for each case in Theorem~\ref{th_lasSashomogeneas}.

\subsection{The algebraical data}\label{sec_data}

The starting point is that each $3$-Sasakian homogeneous manifold is the
total space of a  $\mathrm{Sp}(1)$ or a $\SO(3)$ principal bundle over a symmetric space, \cite[Proposition 13.4.5]{galickiboyer}. 
 More specifically, for a $3$-Sasakian homogeneous manifold $M=G/H$, all the leaves of the $3$-dimensional integrable distribution  ${Q}^\perp$ are diffeomorphic and the 
orbit space  $M/\mathcal{F}_{Q}$ is a
quaternionic K\"ahler $G$-homogeneous manifold \cite[Prop. 13.4.5]{galickiboyer}. Moreover, the natural projection $M\to M/\mathcal{F}_{Q}$ is a principal bundle with structure group $F=\Sp(1)$ or $F=\SO(3)$ and $M/\mathcal{F}_{Q}$ has positive scalar curvature \cite[Sec.~4]{BoGaMan94}. Therefore, from a classical result of Alekseevsky, the orbit space $M/\mathcal{F}_{Q}$ must be symmetric \cite{Ale}. Thus, the natural projection on the orbit space of every $3$-Sasakian homogeneous manifold admit the form
$$M=G/H\to M/\mathcal{F}_{Q}=M/F=G/(H\cdot F).
$$

The reductive decomposition related to the $3$-Sasakian homogeneous manifold $M=G/H$ is introduced from the $\mathbb Z_2$-grading related to the corresponding symmetric space $M/\mathcal{F}_{Q}$.
To be precise,  let us consider the $\mathbb Z_2$-grading on $\g$, or symmetric decomposition   $\g=\g_{0}\oplus\g_{1}$ (that is, $[\g_{i},\g_{j}]\subset \g_{i+j}$). The even part $\g_{0}= \spf(1)\oplus\hh$ is sum of two   ideals,
which immediately implies that $\g=\hh\oplus\mm$ is a reductive decomposition for $\mm=\spf(1)\oplus\g_{1}$. 
Theorem~\ref{le_laestructura!} shows how to construct the 3-Sasakian   structure starting from the Lie algebraic data (see Definition~\ref{def_data}).

In order to simplify   computations, we would like to point out that, except for the projective spaces $\R P^{4n+3}=\mathrm{Sp}(n+1)/ \mathrm{Sp}(n)\times \mathbb{Z}_{2}$, for every $3$-Sasakian homogeneous manifold $G/H$, the Lie subgroup $H$ is connected. Therefore, from now on, {\it we assume that $H$ is connected} and the projective cases will be considered in Remark~\ref{remark_proyectivo} in a separate way. 

Before the next definition, recall that if $W_i$ is an $L_i$-module, $i=1,2$, for $L_i$ any Lie algebra, then $W_1\otimes W_2$ is an $L_1\oplus L_2$-module for the action $(x_1+x_2)\cdot w_1\otimes w_2=(x_1\cdot w_1)\otimes w_2+w_1\otimes (x_2\cdot w_2)$.

\begin{definition}\label{def_data} 
A \emph{$3$-Sasakian data} is a pair of Lie algebras $(\g, \hh)$ such that 
\begin{itemize}
\item
$\g=\g_{0}\oplus\g_{1}$  is a $\mathbb{Z}_{2}$-graded compact  simple  Lie algebra (that is, $[\g_{i},\g_{j}]\subset \g_{i+j}$, $i, j\in\{0,1\}$) whose even part is   sum of two   ideals,
\begin{equation}\label{eq_imp1}
\g_{0}= \spf(1)\oplus\hh;
\end{equation}
 \item
and there exists an  $\hh^\C$-module  $W$ such that the complexified    
  $\g_{0}^\C$-module $\g_{1}^\C$ is isomorphic to the tensor product of the natural $\spf(1)^\C\cong \mathfrak{sl}(2, \C)$-module $\C^2$
  given by multiplication on the columns and $W$, that is,
\begin{equation}\label{eq_imp2}
\g_{1}^\C\cong\C^2\otimes W.
\end{equation}
 \end{itemize}
\end{definition}
In particular,  $\dim_{\R}\mm=4n+3$ holds for $n=\dim_\C W$. A reader interested in the (nonassociative) algebraic structure attached to $W$ can consult Remark~\ref{re_sts}. Note that   Eq.~\eqref{eq_imp2} can be obtained as a consequence of Eq.~\eqref{eq_imp1} under certain conditions. For instance, this is the case if the algegra $\hh$ is semisimple and the $\hh$-module $\g_{1}$ is irreducible, as we will prove in Remark~\ref{remark6}. These conditions will be satisfied in most of our cases.

\begin{theorem}\label{le_laestructura!}
Let $M^{4n+3}=G/H$ be a homogeneous space with $H$ connected such that the pair $(\g , \hh)$ is a $3$-Sasakian data,
being $\g$ and $\hh$ the Lie algebras of $G$ and $H$ respectively.
Take $\mm= \spf(1)\oplus \g_{1}$ (for the subalgebra of type $\spf(1)$ in \eqref{eq_imp1}). 
\begin{enumerate}[i)]
\item The decomposition  $\g=\hh\oplus\mm$ is a reductive decomposition.

\item Let $g$ be the Riemannian metric on $M$ corresponding to the $\mathrm{Ad}(H)$-invariant inner product on $\mm$ given by
\begin{equation}\label{eq_nuestrag}
g\vert_{\spf(1)}=- \frac1{4(n+2)} \kappa,\quad g\vert_{\g_{1} }=-\frac1{8(n+2)}\kappa,\quad g\vert_{\spf(1)\times\g_{1} }=0,
\end{equation}
where $\kappa$ is the Killing form of $\g$. Then, the  multiplication $\alpha^g\colon\mm\times\mm\to\mm$ corresponding via Nomizu's Theorem   with the Levi-Civita connection $\nabla^g$ satisfies 
 \begin{equation}\label{eq_alfadeLevi}
 \alpha^g(X,Y)=\left\{\begin{array}{ll}
 0&\text{if $\,X\in\spf(1)$ and $\,Y\in\g_{1}$},\\
 \frac12[X,Y]_\mm&\text{if either $\,X,Y\in\spf(1)$ or $\,X,Y\in\g_{1}$},\\
{ [X,Y]}_\mm&\text{if $\,X\in\g_{1}$ and $\,Y\in\spf(1)$}.\\
 \end{array}
 \right.
\end{equation}

 \item  Let  $\{\xi_i\}_{i=1}^3$ be the $G$-invariant   vector fields on $M$ corresponding to the following basis of $\spf(1)=\suf(2)$, \begin{equation}\label{eq_losxis}
\xi_1=\left(\begin{array}{cc}\mathbf{i}&0\\0&-\mathbf{i} \end{array}\right),\quad
\xi_2=\left(\begin{array}{cc}0&-1\\ 1&0 \end{array}\right),\quad
\xi_3=\left(\begin{array}{cc}0&-\mathbf{i}\\ -\mathbf{i}&0 \end{array}\right).
\end{equation}
Then, the endomorphism field $\varphi_{i}=-\nabla^g \xi_{i}$ satisfies
\begin{equation}\label{derivada de xi}
\varphi_i\vert_{\spf(1)}=\frac12\ad \xi_i,\quad \varphi_i\vert_{\g_{1} }= \ad \xi_i ,
\end{equation}
for each $i=1,2,3$, and   ${\estS}_i=\{\xi_i,\eta_i,\varphi_i\}$   is a Sasakian structure for $\eta _i=g(\xi_{i},-)$. In addition, ${\estS}_1$, $ {\estS}_2$ and ${\estS}_3$ satisfy the compatibility  conditions \eqref{eq_compatibilityconditions}, 
hence providing a $3$-Sasakian structure on $M$.
\end{enumerate}
\end{theorem}

\begin{proof}
Item i) is clear.
For item ii), note that the inner product   on $\mm$ defined in \eqref{eq_nuestrag} is $\mathrm{Ad}(H)$-invariant,  so that it extends to a $G$-invariant Riemannian metric on $M$, also denoted by $g$.  Next, we show that $g(\alpha^g(X,Y),Z)+g(Y,\alpha^g(X,Z))=0$ for $X,Y,Z\in \mm.$ Indeed, the bilinear map $\alpha^g$ satisfies
\begin{equation}\label{alphacompleta}
\alpha^g (X,Y)=\frac{1}{2}[X^v , Y^v]+ \frac{1}{2}[X^h, Y^h]_{\spf(1)}+[X^h , Y^v],
\end{equation}
where the superscripts $v$ and $h$ denote the projections from $\mm$ on $\spf(1)$ and $\g_{1}$, respectively (note that $[X^h , Y^v]\in \g_{1}$). Now, from (\ref{alphacompleta}) and the associativity of the Killing form, one deduces that
\begin{align*}
g(\alpha^g (X,Y),Z)= & -\frac1{8(n+2)}\kappa \Big([X^v , Y^v]+[X^h, Y^h]_{\spf(1)}+[X^h , Y^v], Z\Big)\\  =&-g(\alpha^g (X,Z), Y),
\end{align*}
since $\kappa(\spf(1),\g_{1} )=0$.
Thus, the affine connection corresponding to $\alpha^g$ is compatible with the metric. Also, one easily shows that its torsion tensor vanishes identically.

The formulae in~\eqref{derivada de xi} for the endomorphism fields $\varphi_{i}=-\nabla^g \xi_{i}$ are a straightforward   computation from
(\ref{alphafacil}) and the explicit expression for $\alpha^g$. In order to deduce that every ${\estS}_{i}=\{\xi_i,\eta_i,\varphi_i\}$ for $i=1,2,3$, is a Sasakian structure, we are going to prove that the vector fields  $\{\xi_i\}_{i=1}^3$ are $g$-orthonormal. Note that $[\xi_{i}, \xi_{j}]= 2\epsilon_{ijk}\xi_{k}$. Now, from (\ref{eq_imp1}), we can compute
$
\kappa(\xi_{i}, \xi_{i})=\textrm{tr}\,(\textrm{ad}^{2}\xi_i )=
\textrm{tr}\,(\textrm{ad}^{2}\xi_i\vert_{\spf(1)} )+\textrm{tr}\,(\textrm{ad}^{2}\xi_i\vert_{\g_{1}})=-8+\textrm{tr} \,(\textrm{ad}^{2}\xi_i\vert_{\g_{1}}).
$
Take into account, by (\ref{eq_imp2}),    that the action of $\spf(1)^\C$ on $\g_{1}^\C\cong\C^2\otimes W$   is given by the matrix multiplication on column vectors, that is, 
$[\xi_i,\begin{pmatrix} x\\y\end{pmatrix}\otimes w]=\xi_i \begin{pmatrix} x\\y\end{pmatrix}\otimes w$. As $\xi_i^2=\begin{pmatrix}-1&0\\0&-1\end{pmatrix}$ for any $i$, we obtain $\textrm{ad}^2\xi_i\vert_{\g_{1}^\C}=-\id$,  and hence $\textrm{ad}^2\xi_i\vert_{\g_{1}}=-\id$. In particular, the Killing form takes the value $\kappa(\xi_i,\xi_i)=-8-\dim_\R \g_{1}=-8-4n$, so that each $\xi_i$ is a  $g$-unitary vector field.  
In a similar way,  for $i\neq j$, 
we get
$$
\kappa(\xi_{i}, \xi_{j})=\textrm{tr}\,(\textrm{ad}_{\xi_i} \circ \textrm{ad}_{\xi_j}\vert_{\spf(1)})+\textrm{tr}\,(\textrm{ad}_{\xi_i} \circ \textrm{ad}_{\xi_j}\vert_{\g_{1}})=0+\textrm{tr}\,(\textrm{ad}_{\xi_i} \circ \textrm{ad}_{\xi_j}\vert_{\g_{1}}).
$$
And again from (\ref{eq_imp2}), we have
$
\Big[\xi_i, \Big[\xi_j,\begin{pmatrix} x\\y\end{pmatrix}\otimes w\Big]\Big]=2 \epsilon_{ijk}\xi_k \begin{pmatrix} x\\y\end{pmatrix}\otimes w.
$
Therefore, it holds $\textrm{tr}_{\g_{1}^\C}\,(\textrm{ad}_{\xi_i} \circ \textrm{ad}_{\xi_j})=0$ and, as the trace does not depend on the field extension,  it yields $\kappa(\xi_{i}, \xi_{j})=0$.

Next, we want to show that every $\xi_{i}$, $i=1,2,3$, is a Killing vector field for $g$. The invariance properties simplify the computations, so we just need to check that for $X,Y\in \mm$, 
$$
g(\varphi_{i}(X),Y)+g(X,\varphi_{i}(Y))=0,
$$
but this formula can be easily deduced from (\ref{derivada de xi}). Finally, it remains to prove the following expression for the covariant derivative of $\varphi_i$,
$$
(\nabla^g_X\varphi_i)Y=g(X,Y)\xi_i-\eta_i(Y)X.
$$
By Lemma~\ref{le_comoderivar} and the invariance of the tensors, this is equivalent to prove, for $ X,Y\in \mm$, that
\begin{equation}\label{3sas}
\alpha^g(X,\varphi_i(Y))-\varphi_i(\alpha^g(X,Y))=g(X,Y)\xi_i-\eta_i(Y)X.  
\end{equation}
As explicit expressions for $\alpha^g$ and $\varphi_{i}$ are given in (\ref{eq_alfadeLevi}) and (\ref{derivada de xi}), respectively, we will check this equation case by case. First,  if  $X\in \spf(1)$ and $Y\in \g_{1}$, then both terms of (\ref{3sas}) vanish. Second, for a fixed $i\in\{1,2,3\}$, if $X \in \g_{1}$ and $Y=\xi_{j}\in \spf(1)$ with $j\neq i$, Eq.~(\ref{3sas}) reduces to check that
$$
\epsilon_{ijk}\alpha^g(X,\xi_{k})-\varphi_i([X,\xi_{j}])=-\epsilon_{ijk}\varphi_{k}(X)+ \varphi_{i}\circ \varphi_{j}(X)=0.
$$
But this holds   due to the equality $\epsilon_{ijk}\varphi_{k}= \varphi_{i}\circ \varphi_{j}$ on horizontal vectors for $j\neq i$. The subcase $j= i$ can be proved similarly. 
Third, (\ref{3sas}) is a straightforward computation when $X$ and $Y$ belongs to $\{\xi_j\}_{j=1}^3$. 
Finally, for $X,Y\in\g_{1}$, (\ref{3sas}) becomes 
$$
\alpha^g(X,\varphi_i(Y))-\varphi_i(\alpha^g(X,Y))=g(X,Y)\xi_i.
$$
As the left-side of this expression can be written as follows
$$
Z:=\frac 1 2 [X,[\xi_i,Y]]_{\spf(1)}-\frac14[\xi_i,[X,Y]]\in\spf(1),
$$
it suffices to prove   that $g(Z,\xi_j)=\delta_{ij}g(X,Y)$, or equivalently,
$\kappa(2Z,\xi_j)=\delta_{ij}\kappa(X,Y).$ But this is easy to check due to the associativity of the Killing form and the fact $\kappa(\hh,\spf(1))=0$. 
\end{proof}
\begin{remark}\label{re_lametrica}
\normalfont 
Theorem~\ref{le_laestructura!} exhibits, in a unified way, the algebraic structures involved in all $3$-Sasakian homogeneous manifolds. In particular, it clarifies our choice of the metric $g$. We should recall the explicit description of the metric tensor of the $3$-Sasakian homogeneous manifolds given by Bielawski in \cite[Theorem~4]{bialowski} as follows,
\begin{equation}\label{eq_metric}
g(X+U,X+U)=-\kappa(X,X)-\frac12\kappa(U,U),
\end{equation}
for all $X\in\mathfrak{sp}(1)$ and $U\in \g_{1}$. Needless to say, we have determined  the metric $g$ with more precision in Theorem~\ref{le_laestructura!}, 
since our metric $g$ is homothetic to the metric in \eqref{eq_metric}.

At this point, it is interesting to note that for $n=0$, we have $ \spf(1)=\mm$, that is, $\g_{1}=0$. For this particular case, the multiplication in $\mm$ corresponding to the Levi-Civita connection coincides with $\frac12[\ ,\ ]$ and so $M=G/H$ is naturally reductive for our reductive decomposition when $n=0$. However, the homogeneous spaces $M=G/H$ in Theorem~\ref{le_laestructura!} are not naturally reductive whenever $n>0$.
\end{remark}

\begin{remark}
\normalfont
Every 3-Sasakian manifold $M$ is Einstein with positive scalar curvature, \cite{kashiwada}. Remarkably, every 3-Sasakian manifold $M$, homogeneous or not, does always admit a second Einstein metric with positive scalar curvature \cite[Theorem~13.3.18]{galickiboyer}. That is, every 3-Sasakian manifold $M$ has at least two distinct homothety classes of Einstein metrics. This second metric, say $\tilde g$, was constructed by making a canonical variation of the original one along the fibers of the natural projection $M\to M/\mathcal{F}_{{Q}}$ on the leaf space. The Sasakian structures 
for the original metric are not Sasakian structures for $(M,\tilde g)$. For any $G$-homogeneous $3$-Sasakian manifold, the new metric $\tilde g$ is also $G$-invariant.
\end{remark}

\begin{remark}
{\normalfont Equation~\eqref{eq_imp2} implies that $\mm^\C$, as $\hh^\C$-module, is the sum of three trivial irreducible modules with two copies of certain module $W$. This will allow to apply Lemma~\ref{le_contando}\,c) to all the  3-Sasakian homogeneous manifolds  to compute the dimension   $\dim_\R\hom_{\mathfrak h}( \Lambda^3\mathfrak m,\R)=\dim_\C\hom_{\mathfrak h^\C}( \Lambda^3\mathfrak m^\C,\C)$. In spite of this unified treatment, we will make use of the concrete module $W$, which   will be computed next jointly with its irreducible summands. 
The $\SU$-case will be the last one to consider. It is the only case   in which the module $W$ is not irreducible, what causes the apparition of new tensors and hence, the existence of more invariant affine connections. To deal with these extra tensors and their covariant derivatives, the explicit expression given in (\ref{eq_alfadeLevi}) for the multiplication $\alpha^g$  on $\mm$ corresponding to the Levi-Civita connection will be
crucial.} 
\end{remark}

In the following subsections, the dimensions of the vector spaces in (\ref{eq_dimensionesqueridas})  will be computed for the homogeneous $3$-Sasakian manifolds listed in Theorem \ref{th_lasSashomogeneas}. These dimensions will provide us the cardinality of the set of  invariant affine connections in each case. 

\subsection{Case $\Sp(n+1)/\Sp(n)$}

Let $\HH$ be the algebra of quaternions with complex units $\{\j_1,\j_2,\j_3\}$. This means $\j_i^2=-1$ for all $1\leq i\leq 3$ and $\j_{i}\j_{j}=\epsilon_{ijk}\j_{k}$ for all permutation of the indices $(1,2,3)$ and $\epsilon_{ijk}$ the sign of the corresponding permutation.  For every $z=x_0+x_1\j_1+x_2\j_2+x_3\j_3\in\HH$, with $x_i\in\R$, $i=0,1,2,3$, we denote by $\overline{z}=x_0-x_1\j_1-x_2\j_2-x_3\j_3\in\HH$ its quaternionic conjugate  and by $\tr z=z+\bar z =2x_0$ its trace. Also, $\mathbb{H}_0$ denotes the set of zero trace quaternions,  that is, $\mathbb{H}_0=\mathbb{R}\j_1\oplus\mathbb{R}\mathbf{j}_2\oplus\mathbb{R}\j_3$. 
For any $n\geq 1$, let us consider the Euclidean metric $g$ on $\HH^{n}$ defined by 
$$
g(z,w)=\mathrm{Re}\Big(\sum_{k=1}^n z_k\overline{w_k} \Big),
$$ 
for $z=(z_{1},...,z_{n})^t, w=(w_{1},...,w_{n})^t\in \HH^n$.
The compact symplectic group $\Sp(n)$ is defined as  
$$
\Sp(n)=\{A\in \mathcal M _{n}(\mathbb H) :  A \overline{A}^{t}=I_{n} \}. 
$$
This Lie group acts on the left by matrix multiplication on $\HH^n$ (as column vectors) in such a way that the Euclidean metric $g$ is preserved.

Let us consider the $(4n+3)$-dimensional unit round sphere
$\esf^{4n+3} = \{ z\in \HH^{n+1} :g(z,z)=1 \}$, which inherits from $g$ the usual Riemannian metric of constant sectional curvature $+1$. 
The action of the Lie group $\Sp(n+1)$ on $\HH^{n+1}$ restricts to a transitive action on $\esfera$ and $\Sp(n+1)\subset  \mathrm{Iso}(\esfera,g)$.  
The isotropy group at the point
$o=(0,\dots,0,1)^t\in \esfera$ is 
\[H=\left\{\left(
\begin{array}{c|c}
  B&0    \\ \hline
  0& 1
\end{array}
\right): B\in \Sp(n)\right\}.\] 
Moreover, we have a reductive decomposition as in (\ref{eq_descomposicionreductiva}) given by
 \begin{equation}\label{eq_descreductivacasoH}
 \begin{array}{l}
 \mathfrak{g}=\mathfrak{sp}(n+1)=\{A\in \mathcal M _{n+1}(\mathbb H) :  A+\bar A^t=0 \},\\
 \mathfrak{h}=\left\{\left(
\begin{array}{c|c}
  B&0    \\ \hline
  0& 0
\end{array}
\right): B\in \mathfrak{sp}(n)\right\}\cong\mathfrak{sp}(n),\\
\mathfrak m=\left\{
\left(
\begin{array}{c|c}
  0&z    \\
  \hline
  -\bar z^t& a
\end{array}
\right) \in \mathcal M_{n+1}(\mathbb H) :   z\in \mathbb H^n,a\in\mathbb{H}_0
\right\}
\cong
\mathbb H^n\oplus \mathbb{H}_0=\g_{1}\oplus \spf(1).
\end{array}
\end{equation} 
By means of the identifications suggested by (\ref{eq_descreductivacasoH}), the action of $\mathfrak{h}$ on $\mathfrak{m}$ is the action of $\mathfrak{sp}(n)$
on $\mathbb H^n\oplus\mathbb{H}_0$ given by $B\cdot(z,a)=(Bz,0)$, for $B\in \mathfrak{sp}(n)$, $z\in\mathbb{H}^n$ and $a\in \mathbb{H}_0$. 
In particular, the action of $\hh$ on $\mathbb{H}_0$ is trivial.
 
As expected, the pair of Lie algebras $(\g, \hh)$ is a $3$-Sasakian data. Indeed,  we have the $\mathbb{Z}_{2}$-grading $\g =\g_{0} \oplus \g_{1}$ with   $\g_{0}=\mathbb{H}_0\oplus \hh$ and $\g_{1}=\mathbb{H}^{n}$. (Recall that $(\mathbb{H}_0,[\ ,\ ])\cong\spf(1)$ by means of the isomorphism $\j_i\mapsto\xi_i$ for $1\leq i \leq3$, with $\xi_i$ as in \eqref{eq_losxis}.) 
Then, Eq.~\eqref{eq_imp2} becomes clear from the next proof.
  
\begin{lemma}\label{base} Assume $\mathfrak{sp}(n+1)=\hh\oplus\mm$, $n\ge1$,  is the reductive decomposition   in \eqref{eq_descreductivacasoH}. Then 
\[\dim _{\mathbb R}\hom_{\mathfrak  h}(  \mathfrak m\otimes \mathfrak m,\mathfrak m ) =63,\, 
\dim_{\mathbb R}\hom_{\mathfrak  h}(\mathfrak m,\Lambda^2 \mathfrak m)=30,\, 
\dim_{\mathbb R}\hom_{\mathfrak  h}(\Lambda^3 \mathfrak m,\mathbb R)=10.\]
\end{lemma}
\begin{proof} Put $Z_n=\tiny{\left(\begin{array}{cc}0&I_n\\-I_n&0\end{array}\right)}$. Then, the complex Lie algebra  
$$
\mathfrak h^{\mathbb C}=\hh\oplus\mathbf{i}\hh\cong  \mathfrak{sp}(2n,\mathbb C)=\left\{A\in \mathcal M _{2n}(\mathbb C) :  AZ_n+Z_n A^t=0 \right\}
$$ 
is simple of type $C_n$. Moreover, $\mathfrak m^{\mathbb C}$ is isomorphic to the module $2V(\lambda_1)\oplus 3\mathbb C$, where $\lambda_i$, for $i=1,\dots,n$, denote the \textit{fundamental weights} for  $C_n$ (notation as in \cite{Draper:Humphreysalg}).
In order to apply Corollary~\ref{co_contar} to the module $U:=V(\lambda_1) $, we compute
\begin{align*}
S^2V(\lambda_1)\cong V(2\lambda_1)\ \textrm{for all $n$, and \ }
&\Lambda^2V(\lambda_1)\cong\left\{
\begin{array}{ll}
V(\lambda_2)\oplus\C&  \textrm{ if $n\ge2$,  }\\
\C&  \textrm{ if $n=1$.  }\\
\end{array}\right.
\end{align*}
Recall that these computations can also be done by using the LiE online service:
\begin{center}
http://wwwmathlabo.univ-poitiers.fr/~maavl/LiE/form.html
\end{center}
It is well-known that, if $\lambda$ and $\beta$ are dominant weights, i.e., belonging to $\Lambda^+=\{\sum_{i=1}^n m_i\lambda_i:m_i\in\mathbb Z_{\ge0}\}$, the irreducible modules $V(\lambda)$ and $V(\beta)$ are isomorphic if, and only if, $\lambda=\beta$. This implies that $\mathfrak m^{\mathbb C}$ satisfies the technical conditions in Corollary~\ref{co_contar}, from which $\dim _{\mathbb R}\hom_{\mathfrak  h}(  \mathfrak m\otimes \mathfrak m,\mathfrak m ) =63$ and $\dim_{\mathbb R}\hom_{\mathfrak  h}(\mathfrak m,\Lambda^2 \mathfrak m)=30$. 
To compute the third dimension, we check
$$
\Lambda^3U\cong\left\{
\begin{array}{ll}
V(\lambda_3)\oplus U&  \textrm{ if $n\ge3$,  }\\
U &  \textrm{ if $n=2$,  }\\
0&  \textrm{ if $n=1$,  }\\
\end{array}\right.
\Lambda^2U\otimes U\cong\left\{
\begin{array}{ll}
V(\lambda_3)\oplus V(\lambda_1+\lambda_2)\oplus 2U&  \textrm{ if $n\ge3$,  }\\
V(\lambda_1+\lambda_2)\oplus 2U &  \textrm{ if $n=2$,  }\\
U&  \textrm{ if $n=1$.}\\
\end{array}\right.
$$ 
Taking into account that any of these modules does not contain any copy of the trivial module $\C$, we  can again apply Corollary~\ref{co_contar} to deduce  $\dim_{\mathbb R}\hom_{\mathfrak  h}(\Lambda^3 \mathfrak m,\mathbb R)=10.
$
\end{proof}

\begin{remark} \normalfont Note that this lemma was proved with different arguments in \cite{DraperPalomoPadge}, but after a long chain of computations which made use of 63 parameters. Now, our argument in \cite{DraperPalomoPadge} is replaced with a short proof which avoids the explicit algebraical descriptions of the connections provided by Nomizu's Theorem.
Moreover, there is a slip-up  in the survey \cite{DraperPalomoPadge} in the expression of the  map $\alpha^g$ related to the Levi-Civita connection, which should  coincide with that one in \eqref{eq_alfadeLevi}.
\end{remark}

\begin{remark}
{\rm 
We have not studied here the case $n=0$ because the 3-dimensional sphere is   well-known. First,  Laquer \cite{otroLaquer} studied invariant affine connections on Lie groups. And second, our work \cite[Section~7]{DraperGarvinPalomo} studies the metric connections with skew  torsion on $\mathbb S^3$ and also, those which are Einstein with skew  torsion. 
}
\end{remark}

\begin{remark}\label{remark_proyectivo}\rm{
To end this subsection, we would like to mention the projective case too. This is the only $3$-Sasakian homogeneous manifold which is not simply connected. Here again, the group $\Sp(n+1)$ acts transitively on $\R P^{4n+3}\cong\esf^{4n+3}/\mathbb Z_2=\{[z]: z\in\C^{2n+2},g(z,z)=1\}$ by $A [z]=[A z]$ for $A \in \Sp(n+1)$ and $[z]\in \R P^{4n+3}$. Now, the isotropy group is not connected. To be precise, the isotropy group at the point $[(0,\dots,0,1)^t]$ is 
\[H=\left\{\left(
\begin{array}{c|c}
  B&0    \\ \hline
  0& \pm1
\end{array}
\right): B\in \Sp(n)\right\},\] 
and the reductive decomposition $\g=\hh\oplus\mm$ is the same as the given one by \eqref{eq_descreductivacasoH}. Bearing in mind that the projection $\esf^{4n+3} \to \R P^{4n+3}$ commutes with the actions of the Lie group  $\Sp(n+1)$ on $\esf^{4n+3}$ and $\R P^{4n+3}$, respectively, every $\Sp(n+1)$-invariant affine connection on $\esf^{4n+3}$ induces, in a natural way, an $\Sp(n+1)$-invariant affine connection on $\R P^{4n+3}$. Therefore, the set of invariant affine connections on $\esf^{4n+3}$ is in one-to-one correspondence with the corresponding set on $\R P^{4n+3}$.

Alternatively, this conclusion can be achieved algebraically.
Take $\sigma=\tiny\left(
\begin{array}{c|c}
  I_n&0    \\ \hline
  0& -1
\end{array}
\right)\in H$, which does not belong to the connected component $H_0$. As $\Ad(\sigma)(z,a)=(z,-a)$, it is easy to check $\Ad(\sigma)(\alpha(X,Y))=\alpha\big(\Ad(\sigma)(X),\Ad(\sigma)(Y)\big)$
for all $X,Y\in\mathfrak{m}$  and  any $\hh$-invariant bilinear map $\alpha\colon\mm\times\mm\to\mm$.
}
\end{remark}

\subsection{Case $\frac{\SO(k )}{\SO(k-4)\times \Sp(1)}$ for $k\ge7$}\label{subcasoSO}

As mentioned, every $3$-Sasakian homogeneous manifold is the
total space of a  $\mathrm{Sp}(1)$ or a $\SO(3)$ principal bundle over a symmetric space. In this case, we have the projection
$$
p:\frac{\SO(k )}{\SO(k-4)\times \Sp(1)} \longrightarrow \frac{\SO(k )}{\SO(k-4)\times \SO(4)},
$$
and the base manifold is the real Grassmann manifold of oriented linear $4$-subspaces of $\R^k$. To be precise, let us recall the following double covering group homomorphism 
$$
\mathrm{Sp}(1)\times \mathrm{Sp}(1)\to \SO(4),\quad (q, z)\mapsto T_{(q,z)}\colon\R^4 \to \R^4\cong \mathbb{H},
$$
where $ T_{(q,z)}v=qv\bar{z}$ and $\mathrm{Sp}(1)=\{q\in \mathbb{H} :  q\bar q=1\}$. The map $T_{(q,z)}=L_{q}\circ R_{\bar{z}}$ is the composition of the left and right quaternionic multiplications
$L_{q}(v)=q v$ and $ R_{\bar{z}}(v)=v\bar{z}$.
Therefore, there are two remarkable ways of considering the group $\mathrm{Sp}(1)$ as a subgroup of $\SO(4)$. Namely, $\mathrm{Sp}(1)^{-}\subset \SO(4)$
given by $q\mapsto T_{(1,q)}$ and $\mathrm{Sp}(1)^{+}\subset \SO(4)$
given by $q\mapsto T_{(q,1)}$.
The above projection $p$ is defined, in a natural way, from $\mathrm{Sp}(1)^{-}\subset \SO(4)$.

\begin{remark}
This point of view gives a geometrical meaning for the $3$-Sasakian manifold $\frac{\SO(k )}{\SO(k-4)\times \Sp(1)}$. 
Recall that we  have the induced isomorphism $\mathrm{Sp}(1)\mathrm{Sp}(1)\to \SO(4)$, where as customary $\mathrm{Sp}(1)\mathrm{Sp}(1):=(\mathrm{Sp}(1)\times \mathrm{Sp}(1))/\mathbb{Z}_{2}$. Now, let us consider $\Pi\in \frac{\SO(k )}{\SO(k-4)\times \SO(4)}$, with $\Pi=\sigma(\SO(k-4)\times \SO(4))$ for $\sigma\in\SO(k )$. Then, the last four column vectors of $\sigma\in\SO(k )$ give an oriented orthonormal basis $\mathcal{B}(\sigma)$ of the 4-subspace $\Pi$ which identifies $\mathcal{B}(\sigma)\colon\R^4\cong \mathbb{H}\to\Pi$. Thus, the fibre over $\Pi$ can be described by
$
p^{-1}(\Pi)=\{\mathcal{B}(\sigma)\circ T_{[1,q]}: q\in \mathrm{Sp}(1)/\mathbb{Z}_{2}\cong \SO(3)\}.
$
Hence, $\frac{\SO(k )}{\SO(k-4)\times \Sp(1)}$ can be seen as a distinguished family of oriented orthonormal basis on every oriented linear $4$-subspace of $\R^k$.
\end{remark}

Now we introduce the $3$-Sasakian data corresponding to this case. The $\mathbb Z_2$-grading on $\g=\sof(k)=\{A\in \mathcal M _{k}(\mathbb R) :  A+ A^t=0\}$  is given by
\begin{align*}
 \g_{0} &=\left\{\left(\begin{array}{c|c}B&0\\\hline0&C
\end{array}\right): B\in\sof(k-4), C\in\sof(4)\right\}, \\
\g_{1}&=\left\{\left(\begin{array}{c|c}0&D\\\hline-D^t&0
\end{array}\right): D\in \mathcal M_{(k-4)\times 4}(\mathbb R)\right\}.
\end{align*}
Next,  $\sof(4)$ is not a simple Lie algebra but decomposes as a sum of two copies of $\spf(1)$. Namely, 
$\sof(4)=I^-\oplus I^+$ for
$$\scriptsize
I^-=\left\{\left(\begin{array}{cccc}
0&\alpha_1&\alpha_2&\alpha_3\\
-\alpha_1&0&-\alpha_3&\alpha_2\\
-\alpha_2&\alpha_3&0&-\alpha_1\\
-\alpha_3&-\alpha_2&\alpha_1&0\end{array}\right):\alpha_i\in\mathbb R \right\},\,\,
I^+=\left\{\left(\begin{array}{cccc}
0&-\alpha_1&-\alpha_2&-\alpha_3\\
\alpha_1&0&-\alpha_3&\alpha_2\\
\alpha_2&\alpha_3&0&-\alpha_1\\
\alpha_3&-\alpha_2&\alpha_1&0\end{array}\right):\alpha_i\in\mathbb R \right\},
$$ 
where these matrices are just the matrices of $R_{\bar q}=-R_q$ and $L_q$, respectively, if $q=\alpha_1\mathbf{j}_1+\alpha_2\mathbf{j}_2+\alpha_3\mathbf{j}_3\in\mathbb H_0\cong\spf(1)$. 
In other words, the above decomposition of $\sof(4)$ has been obtained from the inclusions $\mathrm{Sp}(1)^{-}\subset \SO(4)$ for $I^{-}$ and  $\mathrm{Sp}(1)^{+}\subset \SO(4)$ for $I^{+}$.
In particular, we have the reductive decomposition of $\g=\sof(k)$
given by $\g=\hh\oplus\mm$ for 
$$
\begin{array}{l}\vspace{6pt}
\hh=\left\{\left(\begin{array}{c|c}B&0\\\hline0&C
\end{array}\right): B\in\sof(n),\, C\in I^-\right\},
\\
\mm=\left\{\left(\begin{array}{c|c}0&D\\\hline-D^t&X
\end{array}\right): D\in \mathcal M _{n\times 4}(\mathbb R),\,X\in I^+\right\}\cong I^{+}\oplus \g_{{1}},
\end{array}
$$
with $n=k-4$, which  of course satisfies $\g_{0}= I^{+}\oplus \hh\cong \spf(1) \oplus \hh$.
We use the natural identifications $\hh\cong \sof(n)\oplus I^-$ ($\cong \sof(n)\oplus\spf(1)$)
and $\mm\cong \mathcal M _{n\times 4}(\mathbb R)\oplus I^+$ ($\cong\mathcal M_{n\times 4}(\mathbb R)\oplus\spf(1)$)
 given by
 $$
 \left(\begin{array}{c|c}B&0\\\hline0&C
\end{array}\right)\mapsto (B,C) ,\qquad
\left(\begin{array}{c|c}0&D\\\hline-D^t&X
\end{array}\right)\mapsto (D,X).
$$
Thus, taking into account that $[I^{-}, I^{+}]=0$ (as $R_qL_z=L_zR_q$), the adjoint action of the semisimple Lie algebra $\hh$ on $\mm$ can be expressed as $(B,C)\cdot(D,X)=(BD-DC,0).$  If there is no ambiguity, we use $B$ for $(B,0)$ and $C$ for $(0,C)$. In particular, $\mm$ decomposes as the sum of $I^+$, which is a trivial 3-dimensional $\hh$-module, and $\mathcal M_{n\times 4}(\mathbb R)$, which is an irreducible $\hh$-module.

In order to apply Corollary~\ref{co_contar}, we  have previously to check that $
\mm^\C\cong3\C\oplus 2\U$   for some irreducible and nontrivial $\hh^\C$-submodule $\U$. So, we need to know the decomposition of
 $\mathcal M_{n\times 4}(\mathbb R)^\C\cong\mathcal M_{n\times 4}(\mathbb C)$ as a sum of  $\hh^\C$-irreducible submodules. More precisely, we are going to prove that such decomposition is
 $$
\mathcal M_{n\times 4}(\mathbb C)=\U_1\oplus \U_2,
$$
for
$\U_1:=\{(\textbf{i} a|a|\textbf{i} b|b): a,b\in\C^n\}$
and
$\U_2:=\{(-\textbf{i} a|a|-\textbf{i} b|b): a,b\in\C^n\}$
(notation by columns). 
First, $\U_1\oplus\U_2=\mathcal M_{n\times 4}(\mathbb C)$ is clear since $\C(\textbf{i},1)\oplus\C(-\textbf{i},1)=\C\times\C$. Second, we also have $\hh^\C\U_1\subset\U_1$ (all works analogously  for $\U_2$): 
If $B\in\sof(n)^\C\equiv \sof(n,\C)$, then $B\cdot (\textbf{i} a|a|\textbf{i} b|b)=(\textbf{i} Ba|Ba|\textbf{i} Bb|Bb)\in\U_1$;  and, if $C\in I^-$, then $C\cdot(\textbf{i} a|a|\textbf{i} b|b)\equiv(0,C)((\textbf{i} a|a|\textbf{i} b|b),0)$ $=(-(\textbf{i} a|a|\textbf{i} b|b)C,0)\in\U_1$, since 
\begin{equation}\label{eq_accion}
\begin{array}{l}
\j_1\cdot(\textbf{i} a|a|\textbf{i} b|b)=(a|-\textbf{i} a|-b|\textbf{i} b),
\\\j_2\cdot(\textbf{i} a|a|\textbf{i} b|b) = (\textbf{i} b|b|-\textbf{i} a|-a),
\\\j_3\cdot(\textbf{i} a|a|\textbf{i} b|b) =(b|-\textbf{i} b|a|-\textbf{i} a).
\end{array}
\end{equation}
Third, 
to check the irreducibility of each $\U_i$, we observe that both $\U_i$ are isomorphic to the $\hh^\C$-module
$\U=\C^n\times\C^n=\{(a,b):a,b\in\C^n\}$ endowed with the action given by
\begin{itemize}
\item for $B\in\sof(n,\C)$, $B\cdot (a,b):=(Ba,Bb)$  (so $\U$ is sum of two copies of the natural $\sof(n,\C)$-irreducible representation $\C^n$, each \emph{column} is a copy);
\item the action of $(I^-)^\C\cong\slf(2,\C)=\textrm{Span}\left\{ 
H=\tiny{ \left(\begin{array}{cc}1&0\\0&-1
\end{array}\right)} ,
E=\tiny{ \left(\begin{array}{cc}0&1\\0&0
\end{array}\right)},
F=\tiny{ \left(\begin{array}{cc}0&0\\ -1&0
\end{array}\right) }
\right 
\}$ is by right multiplication:
\begin{equation}\label{eq_acciontraducida}
(a,b)H=
(a,-b),\quad (a,b)E=(0,a),\quad (a,b)F=(-b,0);
\end{equation}
(so $\U$ is sum of $n$ copies of the two-dimensional $\slf(2,\C)$-module $V(1)$:\footnote{
Recall that there is exactly one irreducible  $\slf(2,\C)$-module of each dimension $n+1$, which is frequently denoted  by $V(n)$. It coincides with $V(n)\cong S^n(V(\lambda_1))$. In particular $V(1)\cong V(\lambda_1)\cong\C^2$ is given by the columns.
} each \emph{row} of $(a,b)$ viewed as a matrix in $\mathcal M_{n\times 2}(\mathbb C)$  is such a copy).
\end{itemize}
It is enough to call $H=\textbf{i}\j_1$, $E=\frac12\left(-\j_2+\textbf{i}\j_3\right)$ and $F=\frac12\left(\j_2+\textbf{i}\j_3\right)$ to pass from \eqref{eq_accion} to  \eqref{eq_acciontraducida}, thus getting that $\U_1$ is isomorphic to $\U$. In order to clarify the irreducibility of $\U$, let us assume that $0\ne W\ne \U$ is an $\hh^\C$-submodule of $\U$. Then $W$ must be an  $\sof(n,\C)$-submodule isomorphic to $\C^n$, that is, either 
$W=\{ (a,\alpha a):a\in\C^n  \}$ for some fixed $\alpha\in\C$ or $W=\{ (0, a):a\in\C^n  \}$. But none of these are invariant for the action of $\slf(2,\C)$ given by \eqref{eq_acciontraducida}. This finishes the proof that
$\mm^\C\cong3\C\oplus 2\U$ is the decomposition of $\mm^\C$ as a sum of irreducible $\hh^\C$-submodules.

\begin{lemma}\label{lemacasoSO} If $\g=\hh\oplus\mm$ is the reductive decomposition related to $M=\frac{\SO(k)}{\SO(k-4)\times \Sp(1)}$ for $k\ge7$, then
\[ \dim_{\mathbb R}\hom_{\mathfrak  h}(  \mathfrak m\otimes \mathfrak m,\mathfrak m ) =63,  \
\dim_{\mathbb R}\hom_{\mathfrak  h}(\mathfrak m,\Lambda^2 \mathfrak m)=30,\,  
\dim_{\mathbb R}\hom_{\mathfrak  h}(\Lambda^3 \mathfrak m,\mathbb R)=10.\]
\end{lemma}
\begin{proof}
As mentioned, we can also  apply Corollary~\ref{co_contar} to the complex (semisimple) algebra $\mathfrak  h^\C=\sof(n,\C)\oplus\slf(2,\C)$ for $n=k-4\ge3$  and to the $\mathfrak  h^\C$-module $\mm^\C$.
First, we observe that  $\hom_{\mathfrak  h^\C}(  \U\otimes \U,\U ) =0$, since, as $\slf(2,\C)$-module, $\U$ is isomorphic to $nV(1)$ and hence
$$
\U\otimes\U\cong nV(1)\otimes nV(1)\cong n^2V(2)\oplus n^2V(0). 
$$
Thus, $\hom_{\mathfrak  h^\C}(  S^2 \U,\U ) =\hom_{\mathfrak  h^\C}(  \Lambda ^2 \U,\U ) =0$. Second, we are going to see that 
$\hom_{\mathfrak  h^\C}(  S^2\U,\C )$ $=0$ and that the (complex) dimension of $\hom_{\mathfrak  h^\C}(  \Lambda^2\U,\C ) $ is one. Indeed, $\U$ is isomorphic to $2\C^n$ as $\sof(n,\C)$-module, so that $ \U\otimes \U\cong4\C^n\otimes \C^n$ contains four copies of the trivial module (namely, $\Lambda ^2\C^n$ is the adjoint module and $S^2\C^n$ is sum of a trivial one-dimensional module and $V(2\lambda_1)$).   
In other words, there is a four-parametric family of $\sof(n,\C)$-invariant maps $ \U\times \U\to\C$ given by
$$
\rho((a,b),(c,d))=s_1a^tc+s_2a^td+s_3b^tc+s_4b^td, 
$$
for some scalars $s_i\in\C$. The map $\rho$ is also $\slf(2,\C)$-invariant if, and only if,  $s_1=s_2+s_3=s_4=0$. For instance, the equalities $ s_2+s_3=s_4=0$ are obtained directly from the fact that
$$
\begin{array}{ll}
0&=\rho((a,b)E,(c,d))+\rho((a,b),(c,d)E)=\rho((0,a),(c,d))+\rho((a,b),(0,c))\\
&=s_3a^tc+s_4a^td+s_2a^tc+s_4b^tc
\end{array}
$$
for all $a,b,c,d\in\C^n$. Also $s_1=0$ is achieved by changing $E$ with $F$ above. Therefore, we obtain the unique (up to scalar) $\hh^\C$-invariant map $ \U\times \U\to\C$   given by
$$
\rho((a,b),(c,d))= a^td-b^tc,
$$
which is alternating.  This gives the first two desired dimensions.

For the third   dimension, that one of $\hom_{\mathfrak  h^\C}(\Lambda^3 \mathfrak m^\C,\mathbb C)$,
observe that the module
  $$
  \Lambda^3\mm^\C\cong2\Lambda^3\U\oplus2(\Lambda^2\U\otimes\U)\oplus 9\Lambda^2\U\oplus 3S^2\U\oplus 6\U\oplus\C
  $$
only contains trivial submodules on the summand $\Lambda^2\U$ (one copy of $\C$ on each $\Lambda^2\U$), because, paying attention to the $\slf(2,\C)$-action, there are copies of $V(0)$ neither in $V(1)^{\otimes3}$ nor in 
$(V(2)\oplus V(0))\otimes V(1)$ nor in $S^2V(1)\cong V(2)$. Consequently, there are just $9+1$ copies of the trivial one-dimensional module inside $\Lambda^3\mm^\C$.
\end{proof}

\subsection{Exceptional Cases 
$\frac{G_2 }{\Sp(1) },\quad  
\frac{F_4 }{\Sp(3) },\quad
\frac{E_6 }{\SU(6) },\quad  
\frac{E_7 }{\mathrm{Spin}(12) },\quad 
\frac{ E_8}{ E_7}$.
}

Here we provide a model of the reductive decompositions and hence also of the $3$-Sasakian structures related to the exceptional Lie algebras.
Our approach will be based on the famous Tits' unified construction. Actually, it is not necessary for our purposes of computing dimensions, but we include it here for completeness and beauty.

Let $\CC$  be a (finite-dimensional) real division   algebra, that is, $\CC\in\{\R,\C,\HH,\OO\}$. Then $\CC$ is endowed with a nonsingular quadratic form $n\colon \CC\to \R$, usually called the \emph{norm}, such that $n(xy)=n(x)n(y)$. That is, $n$ is multiplicative, or $\CC$ is a \emph{composition} algebra. Each element $a\in \CC$ satisfies a quadratic equation (with real coefficients) 
$
a^2-t_\CC(a)a+n(a)1=0,
$
where $t_\CC(a)=n(a+1)-n(a)-1$ is called the \emph{trace}. Denote by $\CC_0=\{a\in \CC: t_\CC(a)=0\}$ the subspace of traceless elements. Note that $[a,b]=ab-ba\in \CC_0$ for any $a,b\in \CC$, since $t_\CC(ab)=t_\CC(ba)$.
The map $-\colon \CC\to \CC$ given by $\bar a=t_\CC(a)1-a$ is an involution such that  $n(a)=a\bar a$ and $t_\CC(a)1=a+\bar a$ hold. Furthermore,  for any $a,b\in \CC$, the endomorphism
$
D_{a,b}:=[l_a,l_b]+[l_a,r_b]+[r_a,r_b]$
is a derivation of $\CC$, where $l_a(b)=ab$ and $r_a(b)=ba$ denote the left and right multiplication operators. These are quite representative derivations, since $\der(\CC)=\textrm{Span}\{D_{a,b}:a,b\in\CC\}\equiv D_{\CC,\CC}$. Their main properties are summarized here:
\begin{equation}\label{eq_propDxy}
D_{a,b}=-D_{b,a},\quad D_{ab,c}+D_{bc,a}+D_{ca,b}=0,\quad [d,D_{a,b}]=D_{d(a),b}+D_{a,d(b)},
\end{equation}
for any $a,b,c\in\CC$ and $d\in\der(\CC)$.

Recall that a basis of the   octonion algebra $\OO$ is  $\{1,\textbf{i},\textbf{j},\textbf{k},\textbf{l},\textbf{il},\textbf{jl},\textbf{kl}\}$, 
where the product is given by  $q_1(q_2\textbf{l})=(q_2q_1)\textbf{l}$, $(q_1\textbf{l})(q_2\textbf{l})=-\bar q_2q_1$ and $(q_2\textbf{l})q_1=(q_2\bar q_1)\textbf{l}$ for any $q_i\in\HH=\textrm{Span}\{1,\textbf{i},\textbf{j},\textbf{k}\} $.
The norm is determined by $n(\HH,\textbf{l})=0$ and $n(\textbf{l})=1$, being $n\vert_\HH$   the usual norm of the quaternion algebra.

A commutative  algebra $J$ satisfying the Jordan identity
$(x^2y)x$ $=x^2(yx)$ is called a  \emph{Jordan algebra}. The Jordan algebras relevant   for our purposes are $\R$ and 
$\mathcal{H}_3(\CC)=\{x=(x_{ij}) \in \mathcal{M}_{3 }(\CC): \bar x^t\equiv (\overline{x_{ji}})=  x \}$, for $\CC$ one of the previous composition algebras, where the  product in $J=\mathcal{H}_3(\CC)$ is given by
$$
x\cdot y=\frac12(xy+yx),
$$
denoting here by juxtaposition the usual product of matrices. 
  (This product $\cdot$ is usually called the \emph{symmetrized} product.)
We have a decomposition $J=\R I_3\oplus J_0$, for $J_0=\{x\in J: \tr(x)=0\}$ the subspace of traceless matrices. (This can be extended to $J=\R$ by considering $J_0=0$.)
 We also have a commutative multiplication $*$  on $J_0$,  defined by $x*y:=x\cdot  y-\frac13\tr(x\cdot y)I_3\in J_0$.
Denote by $R_x\colon J\to J$, $y\mapsto y\cdot x$  the multiplication operator, and observe that
   $[R_x,R_y]\in\der(J)$
  for any $x,y\in J$.  

The beautiful unified Tits' construction of all the exceptional simple Lie algebras, \cite{Tits}, is reviewed here only for  compact real exceptional Lie algebras, although it is valid in a wider context. For $\CC$ and $\CC'$ two real division composition algebras and  the Jordan algebra given by  either  $J=\R$ or $J=\mathcal{H}_3(\CC')$, consider the vector space
\begin{equation}\label{eq_TitsModel}
\T(\CC,J)=\der\,(\CC)\oplus (\CC_0 \otimes J_0) \oplus \der\,(J),
\end{equation}
which is made into a (compact) Lie algebra   by defining the (bilinear and anticommutative) multiplication $[\ ,\ ]$ on $\T(\CC,J)$ 
which agrees with the ordinary commutator in $\der(\CC)$ and $\der(J)$ 
and is specified by:
\begin{equation}\label{eq_TitsProduct}
\begin{array}{c}
 {[}\der(\CC),   \der(J)]=0, \qquad  {[}d, a\otimes x]=d(a) \otimes x, \qquad  [D, a\otimes x]=a \otimes D(x), \\
  {[}a\otimes x, b\otimes y]= \frac13\tr(xy) D_{a,b}+[a,b]\otimes (x\ast y)+ 2t_\CC(ab)[R_x,R_y],
\end{array}
\end{equation}
for all $d\in \der(\CC)$, $D\in \der(J)$, $a,b\in \CC_0$ and $x,y\in J_0$.  (If $J=\R$, note that $\T(\CC,\R)=\der\,(\CC)$.) Then:
 
\begin{center}
 {\small
 \begin{tabular}{c|ccccc}
 $ \T(\CC,J)$ & $\R$ &$\mathcal{H}_3(\R)$& $\mathcal{H}_3(\C)$& $\mathcal{H}_3(\HH)$& $\mathcal{H}_3(\OO)$\\
\hline \vrule width 0pt height 9pt
 $\HH$& $\spf(1) $&$\spf(3)  $&$\suf(6)$&$\sof(12)$&${\ef_7}$\\
 $\OO$& $\mathfrak{g}_2$&${\mathfrak{f}_4}$&${ \ef_6}$&${ \ef_7}$&${ \ef_8}$\\
 \end{tabular}}
 \end{center}  
 
 
For each Jordan algebra $J^1=\R$, $J^2=\mathcal{H}_3(\R)$,  $J^3=\mathcal{H}_3(\C)$,  $J^4=\mathcal{H}_3(\HH)$,  $J^5=\mathcal{H}_3(\OO)$, we take the Lie algebras constructed by the above process:  $\mathfrak{g}^s=\T(\OO,J^s)$ and $ \mathfrak{h}^s=\T(\HH,J^s)$, so that $\mathfrak{h}^s$ can be trivially considered as a subalgebra of $\mathfrak{g}^s$. If $\kappa\colon \mathfrak{g}^s\times \mathfrak{g}^s\to\R$ denotes the  (negative definite) Killing form, we will take $ \mathfrak{m}^s$ the orthogonal complement to $ \mathfrak{h}^s$ with respect to $\kappa$. In order to describe  $ \mathfrak{m}^s$ explicitly, we focus first on the case $J^1=\R$ and $ \mathfrak{g}^1=\der(\OO)\cong \mathfrak{g}_2$.

The $\mathbb Z_2$-grading on $\OO=\OO_{0}\oplus\OO_{1}=\HH\oplus \HH\mathbf{l}$ induces a
$\mathbb Z_2$-grading on $\der(\OO)$ with
$$
 \begin{array}{l}
 (\mathfrak{g}^1)_{0}=\der(\OO)_{0}=\{d\in \der(\OO): d(\HH)\subset \HH, d(\HH\mathbf{l})\subset \HH\mathbf{l}\}, \\
  (\mathfrak{g}^1)_{1}=\der(\OO)_{1}=\{d\in \der(\OO): d(\HH)\subset \HH \mathbf{l}, d(\HH\mathbf{l})\subset \HH\}.
 \end{array}
$$
As $\OO$ is generated (as an algebra) by $\HH\mathbf{l}$, any derivation $d\in  (\mathfrak{g}^1)_{0}$ is determined by its restriction to $\HH\mathbf{l}$, so that 
$$
 (\mathfrak{g}^1)_{0}\longrightarrow \mathfrak{so}( \HH\mathbf{l},n)\cong\sof(4),\qquad d\mapsto d\vert_{\HH\mathbf{l}}
$$
is an isomorphism of Lie algebras and then $ (\mathfrak{g}^1)_{0}$ is isomorphic to two copies of $\spf(1)$ as in Section~\ref{subcasoSO}. To be precise,  we introduce the derivations 
$$
\begin{array}{rccc}
d^-_a\colon&\OO&\to&\OO \\
&q&\mapsto&[a,q]\\
&q\textbf{l}&\mapsto&(qa)\textbf{l},
\end{array}\quad  
\qquad \begin{array}{rccc}
d^+_a\colon&\OO&\to&\OO \\
&q&\mapsto&0\\
&q\textbf{l}&\mapsto&(aq)\textbf{l},
\end{array}
$$
for any $a\in \HH_0$ ($q\in\HH$), so that $I^\sigma=\{d^{\sigma}_a : a\in\HH_0\}$ is isomorphic to   $(\HH_0,[\ ,\ ])\cong\spf(1)$ for each $\sigma\in\{\pm\}$
and $(\mathfrak{g}^1)_{0}=I^-\oplus I^+$ is sum of two simple ideals. Consider the   subalgebra
$\mathfrak{h}^1=I^-$,   isomorphic to $\der(\HH)=\T(\HH,\R)$. We are in the situation of \eqref{eq_imp1}, so that we have $\mathfrak{m}^1=I^+\oplus (\mathfrak{g}^1)_{1}$. In order to check Eq.~\eqref{eq_imp2}, we have to dive a little bit in the module structures. Taking into account that $D_{\OO_{i},\OO_{j}}\subset \der(\OO)_{{i+j}}$ (subindices in $\mathbb Z_2$), then $(\mathfrak{g}^1)_{1}= D_{\HH_0,\HH \textbf{l}}$. As $I^+=\{d\in\der(\OO): d(\HH)=0\}$ acts on $(\mathfrak{g}^1)_{1}$ by
$[d^+_a,D_{b,q\textbf{l}}]=D_{b,(aq)\textbf{l}}$, this tells that   $(\mathfrak{g}^1)_{1}=D_{\textbf{i} ,\HH \textbf{l}}\oplus D_{\textbf{j} ,\HH \textbf{l}}$ is sum of two (irreducible) $I^+$-modules, each one working as the $\HH_0$-module $\HH$ under the left multiplication. 
(For these arguments, we have used \eqref{eq_propDxy}.)
Hence, the complexification 
$(\mathfrak{g}^1)_{1}^\C$ breaks as 4 copies of the $(I^+)^\C\cong\slf(2,\C)$-module $V(1)$. Even more,
  it is not difficult to prove (\cite{esphomogeneosdeG2} for more details) that $(\mathfrak{g}^1)_{1}$ is an absolutely irreducible $(\mathfrak{g}^1)_{0}$-module whose complexification becomes
  $$
  (\mathfrak{g}^1)_{1}^\C\cong V(1) \otimes V(3), 
  $$
the tensor product of  the ${(I^{+})^\C}$-module of type $V(1)$ with the  ${(I^{-})^\C}$-module of type $V(3)$.
 In particular  Eq.~\eqref{eq_imp2}  is satisfied so that Theorem~\ref{le_laestructura!} tells   that $G_2/\Sp(1)$ is a 3-Sasakian manifold.
 Another consequence is that $(\mm^1)^\C\cong 3\C\oplus2 V(3)$ is the decomposition as a sum of $ (\mathfrak{h}^1)^\C$-irreducible submodules, and Corollary~\ref{co_contar} can be applied.
 
 All the remaining reductive decompositions for the exceptional cases can be obtained from the above case (note $\mathfrak{h}^r\subset\mathfrak{h}^s$ and $\mathfrak{g}^r\subset\mathfrak{g}^s$ if $r<s$), namely,
 $$
 \begin{array}{l}
 \mathfrak{h}^s=I^-\oplus \der(J^s) \oplus  \HH_0\otimes J^s_0,  \\
 \mathfrak{m}^s=I^+\oplus D_{\HH_0,\HH \textbf{l}}\oplus \HH \textbf{l}\otimes J^s_0.
 \end{array}
 $$
 Indeed, the next (symmetric) decomposition easily provides a $\mathbb Z_2$-grading on 
 $\mathfrak{g}^s=\T(\OO,J^s)$:
 $$
 \begin{array}{l}
 (\mathfrak{g}^s)_{0}=  \der(\OO)_{0} \oplus  (\OO_0)_{0}\otimes J^s_0\oplus \der(J^s)=I^+\oplus  \mathfrak{h}^s,\\
  (\mathfrak{g}^s)_{1}=  \der(\OO)_{1} \oplus  (\OO_0)_{1}\otimes J^s_0,
 \end{array}
 $$ in such a way that $[I^+, \mathfrak{h}^s]=0$, so that  \eqref{eq_imp1} holds. It is well-known that $ (\mathfrak{g}^s)_{1}$ is a $ (\mathfrak{g}^s)_{0}$-irreducible module (see \cite[Chapter~8]{Kac}), moreover, absolutely irreducible. To be precise, $(\mathfrak{g}^s)_{0}^\C\cong \slf(2,\C)\oplus (\mathfrak{h}^s)^\C$, 
where the Lie algebra $(\mathfrak{h}^s)^\C$ is isomorphic to
$$
\slf(2,\C)\, (A_1),\quad \spf(6,\C)\,  (C_1),\quad \slf(6,\C)\,  (A_5),\quad \sof(12,\C)\,  (D_6),\quad \mathfrak{e}_7^\C \, (E_7),
$$
if $s=1,2,3,4,5$ respectively; and $(\mathfrak{g}^s)_{1}^\C$ is isomorphic to the tensor product of  the natural $\slf(2,\C)$-module $\C^2\equiv V(1)$ with certain irreducible $(\mathfrak{h}^s)^\C$-module $W_s$ which can be identified  \cite[Eq.~2.23]{Alb} with the vector space 
$$
W_s=\left\{\begin{pmatrix}\alpha&x\\y&\beta\end{pmatrix}: \alpha,\beta\in\C, x,y\in (J^s)^\C\right\},
$$ 
of dimension $4$,  $14$, $20$, $32$ and $56$     respectively.  In terms of dominant weights,
$$
(\mathfrak{g}^s)_{1}^\C \cong\left\{\begin{array}{ll}
    \C^2\otimes V(3)&\text{if }s=1,\\
   \C^2\otimes V(\lambda_3)&\text{if }s=2,\\
   \C^2\otimes V(\lambda_3)&\text{if }s=3,\\
   \C^2\otimes V(\lambda_5)&\text{if }s=4,\\ 
    \C^2\otimes V(\lambda_7) &\text{if }s=5.
\end{array}\right.
$$  
This gives   Eq.~\eqref{eq_imp2}, so that Theorem~\ref{le_laestructura!} gives the $3$-Sasakian structure, where now  $\xi_i=d_{ \textbf{j}_i}^+\in\der(\OO)\cap\mm^s$ if $i=1,2,3$ (for all $s=1,\dots,5$). In particular, $ (\mathfrak{m}^s)^\C\cong3\C\oplus 2W_s$ is a $(\mathfrak{h}^s)^\C$-module isomorphism, which is the condition to apply Corollary~\ref{co_contar}. 

Two comments are in order. First, there is an abuse of notation, because  $\lambda_i$ is  used simultaneously for the fundamental weight relative to different Lie algebras. We think that this is clear from the context. For $s=4$, more relevant is the fact that there are  two valid decompositions, the other one being  $(\mathfrak{m}^4)^\C\cong 3\C\oplus 2V(\lambda_6)$. Note that  $V(\lambda_5)$ and $V(\lambda_6)$ are not isomorphic $D_6$-modules, but dual,  while the own adjoint module $D_6$ is self-dual. This explains why the existence of one reductive decomposition implies the other one.

With all this information, it is quite easy to compute the desired dimensions.

\begin{lemma} 
If $\g^s=\hh^s\oplus\mm^s$ is the reductive decomposition related to the homogeneous manifold $   {G_2 }/{\Sp(1) }$,   
${F_4 }/{\Sp(3) }$,  ${E_6 }/{\SU(6) }$, 
${E_7 }/{\mathrm{Spin}(12) }$ and 
${ E_8}/{ E_7}$, respectively, for $s=1,\dots, 5$, then 
$\mathrm{dim}_{\mathbb R}\hom_{\mathfrak  h^s}(  \mathfrak m^s\otimes \mathfrak m^s,\mathfrak m^s ) =63$,
$\mathrm{dim}_{\mathbb R}\hom_{\mathfrak  h^s}(\mathfrak m^s,\Lambda^2 \mathfrak m^s)$ $=30$, and   
$\mathrm{dim}_{\mathbb R}\hom_{\mathfrak  h^s}(\Lambda^3 \mathfrak m^s,\mathbb R)$ $=10.$
\end{lemma}
\begin{proof} Once again, we apply Corollary~\ref{co_contar} for $U$ the $(\mathfrak{h}^s)^\C$-irreducible module of type $V(3)$, $V(\lambda_3)$, $V(\lambda_3)$, $V(\lambda_5)$ and $V(\lambda_7)$ respectively (if $s=1,\dots,5$). Indeed,
\begin{enumerate}[\hspace{0.3truecm}($s=\,\,$1)]
\item $S^2U\cong V(6)\oplus V(2)\ $ and $\  \Lambda^2U\cong V(4)\oplus \C$;
\item $S^2U\cong V(2\lambda_1)\oplus V(2\lambda_3)\ $ and $\  \Lambda^2U\cong V(2\lambda_2)\oplus \C$;
\item $S^2U\cong V(\lambda_1+\lambda_5)\oplus V(2\lambda_3)\ $ and $\  \Lambda^2U\cong V(\lambda_2+\lambda_4)\oplus \C$;
\item $S^2U\cong V(2\lambda_5)\oplus V(\lambda_2)\ $ and $\  \Lambda^2U\cong V(\lambda_4)\oplus \C$; 
\item $S^2U\cong V(2\lambda_7)\oplus V(\lambda_1)\ $ and $\  \Lambda^2U\cong V(\lambda_6)\oplus \C$.
\end{enumerate}
Then, $\hom_{\mathfrak  h^\C}(   S^2U, U ) =\,\hom_{\mathfrak  h^\C}(   \Lambda^2U, U ) =\,\hom_{\mathfrak  h^\C}(  S^2 U,\C ) =0$, and the complex dimension of $\hom_{\mathfrak  h^\C}(  \Lambda^2U,\C )$ is one. Now, by Corollary~\ref{co_contar}, we conclude that $\mathrm{dim}_{\mathbb R}\hom_{\mathfrak  h^s}(  \mathfrak m^s\otimes \mathfrak m^s,\mathfrak m^s ) =63$ and $\mathrm{dim}_{\mathbb R}\hom_{\mathfrak  h^s}(\mathfrak m^s,\Lambda^2 \mathfrak m^s)=30$. The third searched dimension is equal to $10$, as a consequence of  the fact that neither $\Lambda^3U$ nor $\Lambda^2U\otimes U$ contains any trivial submodule:
\begin{enumerate}[\hspace{0.3truecm}($s=\,\,$1)]
\item $\Lambda^2U\otimes U\cong V(7)\oplus V(5)\oplus 2V(3)\oplus V(1)$ and $\Lambda^3U\cong V(3)$;
\item $\Lambda^2U\otimes U\cong V(2\lambda_2+\lambda_3) \oplus V(\lambda_1+2\lambda_2)\oplus V(2\lambda_1+\lambda_3)\oplus V(\lambda_1+\lambda_2)\oplus 2V(\lambda_3)$;
$\Lambda^3U\cong V(\lambda_1+2\lambda_2)\oplus V(\lambda_3)$;
\end{enumerate}
and so on.
\end{proof}

\subsection{Case $M=\frac{\SU(m)}{S(\mathrm U(m-2)\times \mathrm U(1))}$ with $m\ge3$.}\label{su}

This family of homogeneous manifolds has attached  the following reductive decomposition, for $n=m-2\ge1$:  
$$
\mathfrak g= \suf(m)=\{A\in \mathcal M _{m}(\mathbb C) :  A+\bar A^t=0 ,\, \text{tr}(A)=0\},$$
the subalgebra   (with matrices written by blocks $1+n+1$)
$$
\mathfrak h =\left\{ \left( \begin{array}{ccc}\frac{-\tr (B)}2&0&0\\0&B&0\\0&0&\frac{-\tr (B)}2\end{array}\right):  B\in\mathfrak{u}(n)\right\} 
$$
and the complementary subspace
$$
\mathfrak m=\left\{ \left( \begin{array}{ccc}  \alpha \textbf{i}&z_2^t&w\\-\bar z_2&0&z_1\\-\bar w&-\bar z_1^t&- \alpha \textbf{i}\end{array}\right): z_1,z_2\in\C^{n},w\in\C,\alpha\in\R\right\}.
$$
Observe that, in this case, $\hh\cong \mathfrak{u}(n)$ is not semisimple. In fact, we have the decomposition $\hh=Z(\hh)\oplus[\hh,\hh]$ with a one-dimensional center $Z(\hh)=\R \textbf{i}I_n$ and $[\hh,\hh]\cong\mathfrak{su}(n)$, which  is  simple if $n\ne 1$ while is $0$ if $n=1$. Thus, if $n=1$, the algebra $\hh$ is one-dimensional and hence abelian. Also, let us note that $\kappa (\hh, \mathfrak m)=0$ for  $\kappa$ the Killing form of $\mathfrak g$, and   $\dim_\R\mathfrak m=4n+3$.


In order to understand $\mm$ as $\hh$-module, we use the above suggested identifications 
$\hh\cong \mathfrak{u}(n)$ and $\mm\cong\C^n\oplus\C^n\oplus\suf(2)$ given by
\begin{equation}\label{eq_identificaciones}
\left( \begin{array}{ccc}\frac{-\tr (B)}2&0&0\\0&B&0\\0&0&\frac{-\tr (B)}2\end{array}\right) \mapsto B,\qquad
\left(\begin{array}{ccc}  \alpha \textbf{i}&z_2^t&w\\-\bar z_2&0&z_1\\-\bar w&-\bar z_1^t&- \alpha \textbf{i}\end{array}\right)\mapsto
\left(z_1,z_2\right)+\left(\begin{array}{cc} \alpha \textbf{i}&w\\-\bar w&-\alpha \textbf{i}\end{array}\right).
\end{equation}
Thus, the action of $\hh$ on $\mm$ is translated from the bracket $[\mathfrak h,\mathfrak m]\subset \mathfrak m$ and can be expressed in these terms as
\begin{equation}\label{eq_acciondeh}
B\cdot\left(\left(z_1,z_2\right)+\left(\begin{array}{cc} \alpha \textbf{i}&w\\-\bar w&-\alpha \textbf{i}\end{array}\right)\right)=\left(\left( B+\frac{\tr (B)}2 I_n\right)z_1,\left(\bar B-\frac{\tr (B)}2 I_n\right)z_2\right),
\end{equation}
for $I_n$  the identity matrix. 
(In particular, $\textbf{i}I_n\cdot(z_1,z_2)=\left(1+\frac n2\right)\textbf{i}(z_1,-z_2)$.)
In other words, $\mathfrak m$ can be decomposed  as $\mathfrak m=\mathfrak m_1 \oplus \mathfrak m_2 \oplus \mathfrak m_3$, the sum of the following  $\hh$-submodules:
\begin{itemize}
\item $\mathfrak m_1\equiv\{(z_1,0):z_1\in\C^n\}$, which is, as
$[\hh,\hh]\cong\mathfrak{su}(n)$-module, isomorphic to  the natural module $\C^n$ (hence irreducible),
while the center $Z(\hh)$ acts scalarly by 
$(\textbf{i} I_n) z_{1}=\left(1+\frac{n}2\right) \textbf{i} z_{1}$;

\item $\mathfrak m_2\equiv\{(0,z_2):z_2\in\C^n\}$, which is, as
$\mathfrak{su}(n)$-module, irreducible and isomorphic to $(\C^n)^*$,  the dual of the natural module,
while the center $Z(\hh)$ acts scalarly by $(\textbf{i} I_n) z_{2}=-\left(1+\frac{n}2\right) \textbf{i} z_{2}$;

\item $\mathfrak m_3\equiv\{X:X\in\suf(2)\}$, which is a trivial 3-dimensional $\mathfrak h$-module.
\end{itemize}
 
  Note that, if $n=1$, the decomposition $\mathfrak m=\mathfrak m_1 \oplus \mathfrak m_2 \oplus \mathfrak m_3$ still works, but now these pieces are not irreducible. To deal with this case, we simply forget the $[\hh,\hh]$-action, taking only into account that $(\textbf{i} I_1) z_{1}= \frac{3}2  \textbf{i} z_{1}$ and $(\textbf{i} I_1) z_{2}= -\frac{3}2  \textbf{i} z_{2}$.

We can do a initial comparison with our previous cases of reductive decompositions related to  3-Sasakian homogeneous manifolds. First, as  mentioned, $\hh$ is not semisimple now. Second, 
we   have the  $\mathbb Z_2$-grading required for a 3-Sasakian data
$\g_{0}=\hh\oplus\mm_3$ and $\g_{1}=\mm_1\oplus\mm_2$,
 since $\mm_3\cong\suf(2)\cong\spf(1)$ and $[\mm_3,\hh]=0$. 
 But  in this case $\g_{1}$ is not $\hh$-irreducible, since $\mm_1$ and  $\mm_2$ are proper $\hh$-submodules. This makes more delicate to prove
 Eq.~\eqref{eq_imp2}, as well as our computation of dimensions.

Thus, let us begin  to study how is $\mm^\C$ as $\hh^\C$-module. First we address the issue $n\ne1$.  
Denote by $\V$ the $\mathfrak{sl}(n,\C)$-natural module $\C^n$ (that is, the action is given by column multiplication). Let $\V_+$ and $\V^*_-$ be the $\mathfrak{gl}(n,\C)=(\mathfrak{sl}(n,\C)\oplus\C I_n)$-modules  which are   $\V$ and $\V^*$ as $\mathfrak{sl}(n,\C)$-modules, and where the action of $  I_n$ is $(1+\frac n2)\id$ and $-(1+\frac n2)\id$, respectively.
Our purpose is to prove that 
\begin{equation}\label{eq_complexifcasoSU}
\mm^\C\cong 3\C\oplus\,2 \V_+\oplus2\V^*_-
\end{equation}
is the decomposition of $\mm^\C$ as a sum of irreducible $\hh^\C$-submodules, so that we can apply Lemma~\ref{le_contando} again. More precisely, we prove $\mathfrak m_1^\C\cong\mathfrak m_2^\C\cong \V_+\oplus \V^*_-$, in spite that $\mathfrak m_1 \not\cong\mathfrak m_2$.

Recall that $\hh^\C=\hh\otimes_\R\C\cong\mathfrak{gl}(n, \C)
=\mathfrak{u}(n)\oplus\mathfrak{u}(n)\mathbf{i}$, since any $A\in\mathfrak{gl}(n, \C)$ can be written as  
$$
A=A_0+A_1\mathbf{i}\quad \text{ for }\quad A_0=\frac12(A-\bar A^t),\ A_1=-\frac{\mathbf{i}}{2}(A+\bar A^t)\in\mathfrak{u}(n).
$$
Besides, if $A\in\mathfrak{sl}(n, \C)$, then $A_0,A_1\in\mathfrak{su}(n)$.
 The  action of $\mathfrak{gl}(n, \C)$ on $(\C^n)^\C$ obtained as a complexification of an action $\diamond$ of $\mathfrak{su}(n)$ on $\C^n$
 is defined by
$$
A \diamond (z\otimes 1+w\otimes \mathbf{i})=(A_0  \diamond z-A_1\diamond  w)\otimes 1+(A_0 \diamond w+A_1\diamond  z)\otimes  \mathbf{i},
$$
for
 any $A\in\hh^\C$ (viewed as  $A_0\otimes 1+A_1\otimes  \mathbf{i}$), $z,w\in\C^n$. In particular, $I_n \diamond(z\otimes 1+w\otimes \mathbf{i})= (\mathbf{i}I_n \diamond w)\otimes 1-(\mathbf{i}I_n \diamond z)\otimes \mathbf{i}$.  This gives, jointly with Eq.~\eqref{eq_acciondeh},
\begin{equation}\label{eq_enmedio}
\begin{array}{l}
A\cdot (z_1\otimes 1-\mathbf{i}z_1\otimes \mathbf{i},0)= (Az_1\otimes 1-\mathbf{i}(Az_1)\otimes \mathbf{i},0), \\
A\cdot (z_1\otimes 1+\mathbf{i}z_1\otimes \mathbf{i},0)= (-\bar A^tz_1\otimes 1-\mathbf{i}(\bar A^tz_1)\otimes \mathbf{i},0) , \\
A\cdot (0,z_2\otimes 1-\mathbf{i}z_2\otimes \mathbf{i})= (0,-A^tz_2\otimes 1+\mathbf{i}(A^tz_2)\otimes \mathbf{i}),   \\
A\cdot (0,z_2\otimes 1+\mathbf{i}z_2\otimes \mathbf{i})=(0,\bar A z_2\otimes 1+\mathbf{i}( \bar Az_2)\otimes \mathbf{i}) ,
\end{array}
\end{equation}
for any $A\in\mathfrak{sl}(n, \C)$, $z_1,z_2\in\C^n$, while $  I_n $ acts scalarly with eigenvalue $\lambda:=1+\frac n2$ 
on 
\begin{equation}\label{eq_enmedio2}\{(z_1\otimes 1-\mathbf{i}z_1\otimes \mathbf{i},0):z_1\in\C^n\}\oplus \{(0,z_2\otimes 1+\mathbf{i}z_2\otimes \mathbf{i}):z_2\in\C^n\}  
\end{equation}
and with eigenvalue $-(1+\frac n2)$
on 
\begin{equation}\label{eq_enmedio3}
\{(z_1\otimes 1+\mathbf{i}z_1\otimes \mathbf{i},0):z_1\in\C^n\}\oplus \{(0,z_2\otimes 1-\mathbf{i}z_2\otimes \mathbf{i}):z_2\in\C^n\} . 
\end{equation}
Let us denote by $H\colon\C^n\times\C^n\to\C$ the usual Hermitian product given by $H(u,v)=u^t\bar v$, and by $\langle\ ,\ \rangle\colon\C^n\times\C^n\to\C$    the usual scalar product given by $\langle u,v\rangle=u^t  v$. Now, it is a direct consequence from (\ref{eq_enmedio}) that the following maps 
$$
 \begin{array}{rcccl}
\mm_1^\C&\cong&\C^n&\oplus&(\C^n)^*\\
(z_1\otimes 1-\mathbf{i}z_1\otimes \mathbf{i},0)&\mapsto& z_1&& \\
(z_1\otimes 1+\mathbf{i}z_1\otimes \mathbf{i},0)&\mapsto&&& H(-,z_1), 
\end{array} 
$$
and
$$
 \begin{array}{rcccl}
\mm_2^\C&\cong&(\C^n)^*&\oplus&\C^n\\
(0,z_2\otimes 1-\mathbf{i}z_2\otimes \mathbf{i})&\mapsto& \langle z_2,-\rangle&& \\
(0,z_2\otimes 1+\mathbf{i}z_2\otimes \mathbf{i})&\mapsto&&& \bar z_2,
\end{array} 
$$
are    isomorphisms of $[\hh ^\C,\hh ^\C]$-modules. Also, Eqs.~(\ref{eq_enmedio2}) and (\ref{eq_enmedio3}) tell that they are   isomorphisms of $\hh ^\C$-modules when we consider the  scalar  action of $  I_n\in Z(\hh^\C)$ on $\C^n$ and $(\C^n)^*$   with eigenvalue  $\lambda $ and $-\lambda $ respectively. This finishes the proof of Eq.~\eqref{eq_complexifcasoSU}.

The case $n=1$ has to be  considered separately. Denote by $\V_s$, $s\in\mathbb Z$, the one-dimensional $\C$-vector space in which $I_1$ acts scalarly with eigenvalue $3/2\, s$. Thus  $\mm^\C$ is isomorphic to the $\hh^\C\cong\C I_1$-module $2\V_{1}\oplus 3\V_{0}\oplus 2\V_{-1}$ (complex dimension 7). It is very useful to observe that $ \V_{s}\otimes \V_{t}\cong \V_{s+t}$ for all $s,t\in\mathbb Z$.

Now, we are in position to prove

\begin{lemma}\label{elnueve} If $\g=\hh\oplus\mm$ is the reductive decomposition related to the manifold  $M=\frac{\SU(m)}{S(\mathrm{U}(m-2)\times \mathrm{U}(1))}$, $m\ge3$, then 
$\dim_{\mathbb R}\hom_{\mathfrak  h}(  \mathfrak m\otimes \mathfrak m,\mathfrak m ) =99$,  
$\dim_{\mathbb R}\hom_{\mathfrak  h}(\mathfrak m,\Lambda^2 \mathfrak m)=45$, and   
$\dim_{\mathbb R}\hom_{\mathfrak  h}(\Lambda^3\mathfrak{m}  ,\mathbb R)=13.$
\end{lemma}

\begin{proof}
Consider first the general case $n=m-2\ne 1$.
On one hand,  by Lemma~\ref{le_contando}\,c) and b), 
\begin{align*}
\Lambda^2\mm^\C 
&\cong 3\Lambda^2(\V_+\oplus\V^*_-) \oplus S^2(\V_+\oplus\V^*_-) \oplus 6\V_+\oplus6\V^*_- \oplus 3\C\\
&\cong  3\Lambda^2 \V_+\oplus3\Lambda^2\V^*_-\oplus 4(\V_+\otimes\V^*_-)\oplus  S^2\V_+\oplus S^2\V^*_- \oplus 6\V_+\oplus6\V^*_- \oplus 3\C.
\end{align*}
Observe that $   I_n$ acts on  $\Lambda^2 \V_+$ and on $S^2 \V_+$ with eigenvalue $2\lambda$, on  $\Lambda^2 \V^*_+$ and on $S^2 \V^*_+$ with eigenvalue $-2\lambda$, and on $\V_+\otimes\V^*_-$ with eigenvalue $0$, so that the summand
$3\Lambda^2 \V_+\oplus3\Lambda^2\V^*_-\oplus 4(\V_+\otimes\V^*_-)\oplus  S^2\V_+\oplus S^2\V^*_- $
does not contain any copy of $\V_+$ or $\V^*_-$ and moreover, it just contains 4 copies of the trivial module, because the $\mathfrak{sl}(n,\C)$-module   $\V\otimes\V^*$ is isomorphic   to the adjoint module $\slf(\V)$ direct sum with a copy of the trivial one. Hence
\begin{align*}
 \mathrm{dim}_{\mathbb R}\hom_{\mathfrak  h}(  \mathfrak m, \mathfrak m\wedge\mathfrak m ) 
 &=3\, \mathrm{dim}_{\mathbb C}\hom_{\mathfrak  h^\C}( \Lambda^2\mathfrak m^\C,\C ) +
  2\,\mathrm{dim}_{\mathbb C}\hom_{\mathfrak  h^\C}( \Lambda^2\mathfrak m^\C,\V^+ ) \ 2\\
  &=3(4+3)+2\cdot 6\cdot 2=45.
\end{align*}  

Similarly, we check that
$$ \mathfrak m^\C\otimes\mathfrak m^\C \cong 
 4\Lambda^2(\V_+\oplus\V^*_-) \oplus 4S^2(\V_+\oplus\V^*_-) \oplus 12\V_+\oplus12\V^*_- \oplus 9\C,
$$
and hence, taking into consideration that in $S^2(\V_+\oplus\V^*_-) \cong S^2(\V_+)\oplus S^2(\V^*_-) \oplus   (\V_+\otimes\V^*_-)   $ there is just one copy of $\C$, then
\begin{align*}
\mathrm{dim}_{\mathbb R}\hom_{\mathfrak  h}(  \mathfrak m\otimes \mathfrak m,\mathfrak m ) 
 &=3\, \mathrm{dim}_{\mathbb C}\hom_{\mathfrak  h^\C}( \mathfrak m^\C\otimes\mathfrak m^\C,\C ) \\ & \quad 
\quad+ 2\, \mathrm{dim}_{\mathbb C}\hom_{\mathfrak  h^\C}( \mathfrak m^\C\otimes\mathfrak m^\C,\V_+ ) \, 2\\
  &=3(4+4+9)+2\cdot 12\cdot 2=99.
 \end{align*}

Finally, the $\mathfrak{gl}(n, \C)$-module $\Lambda^3\mm^\C$ decomposes as a sum of:
\begin{itemize}
\item $2\Lambda^3(\V_+\oplus\V^*_-)\oplus\, 2\big(\Lambda^2(\V_+\oplus\V^*_-)\otimes (\V_+\oplus\V^*_-)\big)$, which does not contain any copy of the trivial module $\C$, since $ I_n\in\mathfrak{gl}(n, \C)$ acts with eigenvalue $\pm\lambda\pm\lambda\pm\lambda\ne0$ on any of the submodules;
\item 9 copies of $\Lambda^2(\V_+\oplus\V^*_-)$, each one with a trivial 1-dimensional submodule;
\item 3 copies of $S^2(\V_+\oplus\V^*_-)$, each one with a trivial irreducible submodule;
\item   $6\V_+\oplus6\V^*_-\oplus \C$.
\end{itemize}
In particular $\Lambda^3\mm^\C$ just contains $9+3+1=13$ copies of $\C$ and the result follows.

Consider finally the case $n=1$. We obtain the same dimensions as for $n\ne1$ for the three sets of homomorphisms, after computing
$$
\begin{array}{l}
\mm^\C\otimes \mm^\C\cong 4 \V_{2 }\oplus12\V_{ 1}\oplus17\V_{0}\oplus12\V_{ -1}\oplus4\V_{ -2}; \\
\Lambda^2\mm^\C\cong  \V_{2 }\oplus6\V_{ 1}\oplus7\V_{0}\oplus6\V_{ -1}\oplus\V_{ -2}; \\
\Lambda^3\mm^\C\cong  3\V_{2 }\oplus8\V_{ 1}\oplus13\V_{0}\oplus8\V_{ -1}\oplus3\V_{ -2}; 
\end{array}
$$
and applying Lemma~\ref{le_comocontar}\,ii).
\end{proof}

\begin{remark}\label{la phi0}
{\rm It is a remarkable fact that there are more invariant $3$-forms on $\mm$ than in the remaining 3-Sasakian homogeneous manifolds studied in the previous  subsections. 
We can provide a explicit description of the related 13 linear independent 3-forms.
Consider first    $\{\xi_i,\eta_i,\varphi_i\}_{i=1}^3$ and the metric $g$ as in Theorem~\ref{le_laestructura!}. Use our previous identification of $\mm$ with $\C^n\oplus\C^n\oplus\suf(2)$,   writing the elements as a sum of a pair of vectors and a matrix in $\suf(2)$. As the Killing form of $\suf(m)$ is $\kappa(x,y)=2m \tr(xy)$,  
it is easy to check that
\[
 g((z,w),(u,v)) = \frac12\mathrm{Re}(z^t\bar{u}+w^t\bar{v}), 
\]
for any $z,w,u,v\in\C^n$, and
\begin{equation}\label{eq_accionesficasoU}
\varphi_1(z,w)=(\textbf{i} z,\textbf{i} w),\quad  
\varphi_2(z,w)=(-\bar w,\bar z),\quad
\varphi_3(z,w)=(-\textbf{i} \bar w,\textbf{i} \bar z).
\end{equation}
This allows to check that the 2-forms $\Phi_i$'s restricted to $\C^n\oplus\C^n\le\mm$ are given by
$$
\Phi_i((z,w),(u,v))=\left\{\begin{array}{ll}
 \mathrm{Im} (z^t\bar u+w^t\bar v)/2&i=1,\\
 \mathrm{Re} (-z^t v+w^t u)/2&i=2,\\
\mathrm{Im} (-z^t v+w^t u)/2\quad & i=3.\\
\end{array}
\right.
$$
Finally, take $h=\left(1+\frac n2\right)^{-1}\tiny\begin{pmatrix}-\frac{n\mathbf{i}}2& 0& 0\\ 0 &\mathbf{i}I_n& 0 \\ 0& 0 &-\frac{n\mathbf{i}}2\end{pmatrix}\in\hh$
and the endomorphism $\varphi_0=\ad h\vert_{\mm}\in\ad\hh\vert_{\mm}\subset \textrm{End}_\hh(\mm)$. By Eq.~\eqref{eq_acciondeh},  
\begin{equation}\label{eq_accionesficasoU_0}
\varphi_0(z,w)=(\textbf{i} z,-\textbf{i} w)\quad\textrm{and}\quad\varphi_0\vert_{\mathfrak{su}(2)}=0,
\end{equation}
so that the related $\hh$-invariant 2-form $\Phi_0(-,-)=g(-,\varphi_0(-)) $  satisfies
$$\Phi_0((z,w),(u,v))=
\frac12 \mathrm{Im} (z^t\bar u-w^t\bar v).$$
As $[\hh,\spf(1)]=0$, we obtain $\varphi_0\varphi_i = \varphi_i\varphi_0$, for any $i=1,2,3$.  It is also clear that $ g(\varphi_0X,Y)+g(X,\varphi_0Y)=0$. As a consequence, the extra 3-forms are $\eta_i\wedge\Phi_0$, $i=1,2,3$, which are of course linearly independent. }
\end{remark}
 
In order to use Theorem~\ref{le_laestructura!} to assert that $\{\xi_i,\eta_i,\varphi_i\}_{i=1}^3$ is in fact a 3-Sasakian structure, we need to show  that the $\g_{0}^\C$-module $\g_{1}^\C$ is isomorphic to the $\g_{0}^\C\cong \slf(2,\C)\oplus \slf(n,\C)$-module 
$$
\g_{1}^\C\cong \C^2\otimes ( \V_+\oplus \V^*_-).
$$
We know that this is true as $\hh^\C$-modules, by  Eq.~\eqref{eq_complexifcasoSU}. We also know  (\cite{Draper:Humphreysalg}) that any  irreducible $\g_{0}^\C$-submodule $W_i$ of $\g_{1}^\C$ is the tensor product of an irreducible $\slf(2,\C)$-module $V_i$ with an irreducible $\slf(n,\C)$-module $U_i$, so that $\g_{1}^\C=\sum_i V_i\otimes U_i$ is isomorphic  as $\slf(n,\C)$-module to $\sum_i(\dim V_i)U_i\cong2\V\oplus2\V^*$. We would like to prove that $\dim V_i=2$ for any $i$. By dimension count, the only other possibility would be $\dim V_i=1$ for some $i$, but then there would be a trivial $\slf(2,\C)$-submodule of  $\g_{1}^\C$ of dimension $n$. This would give a contradiction, because  the $\slf(2,\C)$-action on $\g_{1}^\C$ is complexified of the action of $\varphi_i$'s described in Eq.~\eqref{eq_accionesficasoU}, obviously never trivial: $\varphi_i^2\vert_{\g_{1}}=-\id$.

\subsection{Several general algebraical comments.}  We have shown that $\hom_{\mathfrak  h}(  \mathfrak m\otimes \mathfrak m,\mathfrak m )$ is a real vector space of dimension $63$ in all the cases with $G\ne\SU(m)$. This was observed in the case $G_2/\SU(2)$ (\cite{esphomogeneosdeG2})  and on the spheres $\mathbb S^{4n+3}$ under the action of the symplectic group  (\cite{DraperPalomoPadge}). 
We are going to describe this set in an unified notation, which can obviously be extended to obtain concrete expressions of the invariant affine connections.

\begin{proposition}\label{pr_descriunificada}
Assume the previous situation with    $\g=\g_{0}\oplus\g_{1}\ne\mathfrak{su}(m)$. Given $\pi_{_Q}\colon\mm\to\mm$  the projection  onto the second factor relative to the decomposition $\mm=\spf(1)\oplus\g_{1}$,  and $X^h=\pi_{_ Q}(X)$ if $X\in\mm$. Let $\theta\colon \mathfrak m\times \mathfrak m\to\mathfrak m\times \mathfrak m$ be the interchanging map $\theta(X, Y)=(Y, X)$ and let $\kappa$ be the Killing form of $\g$. Then a basis of the vector space $\hom_{\mathfrak  h}(  \mathfrak m\otimes \mathfrak m,\mathfrak m )$ is provided by the following bilinear maps
\begin{equation}\label{eq_labase}
\Big\{\alpha_{rst},\beta_{0s},\beta_{rs},\beta_{0s}\circ\theta,\beta_{rs}\circ\theta,\gamma_{0s},\gamma_{rs}: r,s,t=1,2,3\Big\},
\end{equation}
which are defined, for any $X,Y\in\mm$, by
\begin{align*}
& \alpha_{rst}(X, Y)=\eta_r(X)\eta_s(Y)\xi_t,\\
& \beta_{0s}(X,Y)=\eta_s(X)Y^h, \qquad 
\beta_{rs}(X, Y)=\eta_s(X)\varphi_r(Y^h),& \\
&\gamma_{0s}(X, Y)=\kappa(X^h,Y^h)\xi_s, \qquad 
 \gamma_{rs}(X, Y)=\eta_r([X^h,Y^h]_{\spf(1)})\xi_s.&
\end{align*}
\end{proposition}

\begin{proof}
This   is simply a corollary of all the previous results, since the set in \eqref{eq_labase} provides 63 independent elements in  $\hom_{\mathfrak  h}(  \mathfrak m\otimes \mathfrak m,\mathfrak m )$.
\end{proof}

We would like to dive a little bit more into   several algebraical facts concerning  the  3-Sasakian data.

\begin{remark}\label{re_sts} \normalfont
The complex module $W$ involved  a 3-Sasakian data (Definition~\ref{def_data}, Eq.~\eqref{eq_imp2}) is a  (complex) \emph{symplectic triple system}, as defined in \cite{sts}. This means that there exist a triple product $[\ ,\ ,\ ]$ in $W$ and a symplectic form $(\ ,\ )\colon W\times W\to\C$ satisfying the following list of identities, for any $x,y,z,u,v,w\in W$: 
\begin{align*}
  &[x,y,z]=[y,x,z],\quad ([x,y,u],v)=-(u,[x,y,v]),\\
  &[x,y,z]-[x,z,y]=(x,z)y-(x,y)z+2(y,z)x,\\
  &[x,y,[u,v,w]]=[[x,y,u],v,w]+[u,[x,y,v],w]+[u,v,[x,y,w]],
  \end{align*}
Conversely, given any symplectic triple system $(W, [\ ,\ ,\ ],(\ ,\ ))$, consider the set of inner derivations $\mathfrak{inder}(W)$, that is, the linear span of the operators $[x,y,-]\in\mathfrak{sp}(W, (\ ,\ ))$. Then, $\mathfrak{inder}(W)$ is a Lie subalgebra of $\mathfrak{gl}(W)$ and  $\g=\g_{0}\oplus  \g_{1}$ is a $\mathbb Z_2$-graded Lie algebra for 
\[  \g_{0}=\mathfrak{sp}(V)\oplus \mathfrak{inder}(W),\qquad \g_{1}= V\otimes W,
\]
being $V$ a two dimensional vector space endowed with a  nonzero symplectic form. According to \cite[Theorem~5.3]{triples}, $\g$ is a simple algebra if, and only if,  $W$ is simple as a symplectic triple system, which is our situation.
  
These symplectic triple systems appeared as ingredients in the constructions of 5-graded Lie algebras $\g=\oplus_{i=-2}^2\g_i$ with one dimensional corners: $\dim\g_{\pm2}=1$. They are strongly related to Freudenthal triple systems and to Faulkner ternary algebras (see \cite{triples} for more details and references). A natural question on 3-Sasakian homogeneous manifolds is how the curvature tensor could be naturally expressed by means of this ternary product.    
\end{remark}

\begin{remark}\label{remark6}    
\normalfont We have proved that we have a 3-Sasakian data on each of our   cases, so that the 3-Sasakian  structure is described by Theorem~\ref{le_laestructura!}. 
We have used a case-by-case test proving $\g_{1}^\C\cong\C^2\otimes W$. But now  we would like to prove  that some conditions are sufficient to assure Eq.~\eqref{eq_imp2}  starting with Eq.~\eqref{eq_imp1}, that is, with
the $\mathbb Z_2$-grading $\g=\g_{0}\oplus\g_{1}$    whose even part is sum of the ideals $\g_{0}= \spf(1)\oplus\hh$. Namely,
\begin{center}
\textit{  If $\g$ is simple, $\hh$ is semisimple and $\g_{1}$ is an $\hh$-irreducible module\\ (i.~e. $G\ne \SU(m)$), then Eq.~\eqref{eq_imp2} holds. }
\end{center}
Indeed, as  $[\spf(1),\hh]=0$ and  $[\spf(1),\g_{1}]\subset[\g_{0},\g_{1}]\subset \g_{1}$, then it holds  $\ad(\spf(1))\vert_{\g_{1}}\subset\textrm{End}_\hh(\g_{1})$.   Besides $[\spf(1),\g_{1}]\ne0$ since otherwise $\spf(1)$ would be an ideal of $\g$, hence $\spf(1)\cong \ad(\spf(1))\vert_{\g_{1}}$. As the (finite-dimensional) centralizer $\textrm{End}_\hh(\g_{1})$ is a real division associative algebra (consequently of dimension either 1 or 2 or 4) containing a copy of $\mathbb H_0$, then such centralizer has to be isomorphic to $\mathbb H$. If $\hh$ is simple, \cite[Lemma~5.7]{mi G2} tells that $ \g_{1}^\C $ is sum of two isomorphic irreducible $\hh^\C$-modules (and it has real dimension multiple of 4). On the other hand, every irreducible module for a complex semisimple algebra is isomorphic to the tensor product of irreducible modules for the simple components. That is, we know that, as $\g_{0}^\C$-modules,  $ \g_{1}^\C\cong\sum_{i} V_i\otimes U_i $ with $V_i$ an irreducible $\spf(1)^\C$-module and $U_i$ an irreducible $\hh^\C$-module. Hence, as $\hh^\C$-module, $ \g_{1}^\C\cong \sum_i(\dim V_i) U_i $, so that all $U_i$ are isomorphic $\hh^\C$-modules (to certain $W$) and $\sum_i\dim V_i=2$, which was the number of irreducible $\hh^\C$-submodules of $ \g_{1}^\C$.
We are not in the situation of two summands with $\dim V_i=1$ for $i=1,2$, since in such a case, the $\spf(1)$-action on $\g_1$ would be trivial. Thus $ \g_{1}^\C\cong  V\otimes W$ for $V$ an irreducible $\spf(1)^\C$-module of dimension 2. But there is only one nontrivial $\spf(1)^\C=\mathfrak{sl}(2,\C)$-module of dimension 2, the natural module $\C^2$, so that Eq.~\eqref{eq_imp2} holds. The proof in the semisimple case is the same one,   taking into account that \cite[Lemma~5.7]{mi G2} can be adapted to such hypothesis: only the complete reducibility was used through the proof. 
 
Observe also that we can change the hypothesis \lq\lq$\g$ simple\rq\rq\  with the   hypothesis  \lq\lq$\spf(1)$ is not an ideal of $\g$\rq\rq.   
\end{remark}

\section{Invariant Affine Connections with Skew-Symmetric Torsion}\label{affine}

In this   section, first we   geometrically describe the set of invariant metric affine connections with skew  torsion. Second, we  determine which of  them are Einstein with skew  torsion. Third, we characterize those 
with symmetric Ricci tensor. Fourth, we consider a  notion generalizing that one of Einstein with skew  torsion, namely, when the Ricci tensor is a multiple of the metric on the canonical horizontal and vertical distributions, but with different factors.  Fifth, we  find a distinguished affine connection in the above family, characterized by parallelizing all  the Reeb vector fields associated with the $3$-Sasakian  structure.
Finally, we study the invariant affine connections with parallel skew  torsion.

\subsection{Invariant Affine Connections }\label{sec_invariantes} Here we give a basis of the space $\hom_{\mathfrak  h}(\wedge^3\mathfrak{m},\mathbb R)$, which is used to describe explicitly the set of all invariant affine connections with skew  torsion (in particular, they are metric) on a homogeneous $3$-Sasakian manifold. 

Let $(M=G/H,g)$ be a $3$-Sasakian homogeneous manifold with a fixed reductive decomposition as in \eqref{eq_descomposicionreductiva}. We will focus on dimension at least $7$, since the $3$-dimensional sphere $\mathbb S^3$ has been previously studied in \cite{DraperGarvinPalomo}. Let $\{\xi_i,\eta_i,\varphi_i\}_{i=1,2,3}$ be  three compatible Sasakian structures on $M$. Also recall, for any $X,Y\in\mathfrak{X}(M)$, the 2-forms $\Phi_i(X,Y)= g(X,\varphi_iY)$. We put $X^{v}$ the component of $X$ in  the vertical distribution $ Q^{\perp}$, i.~e., $X^v=\sum_{j=1}^3\eta_j(X)\xi_j$,  and   we consider the unique vectorial product  {here} such that  $\xi_1\times \xi_2=\xi_3$ (that is, $\times=\frac12[\ ,\ ]$). In particular, $\eta_r\wedge\Phi_s(X_1,X_2,X_3) =  \, \sum_{i=1}^3 \eta_r(X_{i}) \Phi_s (X_{i+1},  X_{i+2} )$ 
(indices modulo 3).\footnote{Our convention for  the exterior product of a $p$-form $\omega_1$ and a $q$-form $\omega_2$ is the following
\[
\omega_1\wedge \omega_2(X_1,\dots,X_{p+q})=   
\frac1{p!q!}\sum_{\sigma\in  S_{p+q}}(-1)^{[\sigma]} 
\omega_1(X_{\sigma(1)},\dots,X_{\sigma(p)})\omega_1(X_{\sigma(p+1)},\dots,X_{\sigma(p+q)}).
\]}
Denote by $T^o$ and $ T^{rs}$  the $(1,2)$-tensors given by  
\begin{equation}\label{def_torsiones}
 \eta_1\wedge\eta_2\wedge\eta_3(X,Y,Z)=g(T^o(X,Y),Z),\quad 
\eta_r\wedge\Phi_s(X,Y,Z)=g(T^{rs}(X,Y),Z),
\end{equation}
for any $X,Y,Z\in\mathfrak X(M)$ and $r,s \in \{1,2,3\}$. 
For the $\SU$-case, it will also be convenient
 to introduce the tensors $T^{r0}$ given by (see Remark \ref{la phi0})
$$
\eta_r\wedge\Phi_0(X,Y,Z)=g(T^{r0}(X,Y),Z).
$$

\begin{lemma}\label{bases}
\begin{enumerate}[i)]
\item For $G\ne \SU(m)$, the set $\{\eta_1\wedge\eta_2\wedge\eta_3\}\cup\{\eta_r\wedge\Phi_s : r,s=1,2,3\}$ at the point $o$  is a basis   of $\,\hom_{\mathfrak  h}(\wedge ^3\mathfrak{m},\mathbb R)$. 
 \item For $G=\SU(m)$, $m\geq 3$, the set $\{\eta_1\wedge\eta_2\wedge\eta_3\}\cup\{\eta_r\wedge\Phi_s : r=1,2,3, \ s=0,1,2,3\}$ at the point $o$ is a basis of $\hom_{\mathfrak  h}(\wedge^3\mathfrak{m},\mathbb R)$.
\item For any $X,Y\in\mathfrak{X}(M)$,  {$r=1,2,3$, $s=0,1,2,3$, we have }
\[ T^o(X,Y)=X^v \times Y^v,\qquad 
 T^{rs}(X,Y)= -\eta_r(X)\varphi_sY +\eta_r(Y)\varphi_sX +\Phi_s(X,Y)\xi_r. \]
\end{enumerate}\end{lemma}
\begin{proof}
For item i), recall that  all $G$-homogeneous $3$-Sasakian manifolds with $G\ne \SU(m)$ satisfy  
$\dim_{\mathbb R}\hom_{\mathfrak  h}(\wedge^3\mathfrak{m},\mathbb R)=10.$ Thus, we only need to check the linear independence of the set given in $i)$. Indeed, consider a linear combination
$a\, \eta_1\wedge\eta_2\wedge\eta_3+\sum_{r,s=1}^3 b_{rs}\,  \eta_r\wedge\Phi_s =0.$ Given a unit vector $X\in\mm$    in the horizontal subspace $\g_1$, we take $(Y_1,Y_2,Y_3)=(\xi_i,\varphi_jX,X)$. Then, taking into account that $g(\varphi_jX,\varphi_sX)=\delta_{js}g(X,X)=\delta_{js}$, we compute 
$$
\begin{array}{rl}
 0&= \Big( a\, \eta_1\wedge\eta_2\wedge\eta_3+\sum_{r,s=1}^3 b_{rs} \,\eta_r\wedge\Phi_s\Big)(Y_1,Y_2,Y_3) \\
&=\sum_{r,s=1}^3 b_{rs} \, \sum_{k=1}^3 \eta_r(Y_k)\Phi_s(Y_{k+1},Y_{k+2}) 
=\sum_{r,s=1}^3 b_{rs}\,  \delta_{ri}g(\varphi_jX,\varphi_sX) = b_{ij}.
\end{array}
$$
Now $0=a\, \eta_1\wedge\eta_2\wedge\eta_3(\xi_1,\xi_2,\xi_3)=a$. This  shows the linear independence in case i).   
The proof of item~ii) is similar, bearing in mind that  $\dim_{\mathbb R}\hom_{\mathfrak  h}(\wedge^3\mathfrak{m},\mathbb R)=13$.

Finally, we get the metrically equivalent torsions to the 3-forms $\eta_1\wedge\eta_2\wedge\eta_3$ and $\eta_r\wedge\Phi_s $. For any $X_1,X_2,X_3\in\mathfrak{X}(M)$, we have
$$
\begin{array}{l}
   \eta_1\wedge\eta_2\wedge \eta_3(X_1,X_2,X_3) = 
\sum_{\sigma\in S_3 } (-1)^{[\sigma]}  \eta_{1}(X_{\sigma(1)}) \eta_{2}(X_{\sigma(2)}) \eta_{3}(X_{\sigma(3)}) \vspace{4pt}\\
\  = \det ( X_1^v, X_2^v,X_3^v) = g( X_1^v \times X_2^v,X_3^v)=g( X_1^v \times X_2^v,X_3),
\end{array}
$$
and then $T^o(X,Y)=X^v \times Y^v$.
Also, since $\Phi_s$ is alternating,  
\begin{align*}     \eta_r\wedge \Phi_s(X_1,X_2,X_3)  
& = \eta_r(X_{1})g(X_{2},\varphi_sX_{3})+\eta_r(X_{2})g(X_{3},\varphi_sX_{1})+\eta_r(X_{3})g(X_{1},\varphi_sX_{2}) \\
& = g\left( X_3, -\eta_r(X_1)\varphi_sX_2  + \eta_r(X_2)\varphi_sX_1 +g(X_1,\varphi_sX_2)\xi_r 
\right),
\end{align*}
so we get the desired expression for $T^{rs}$.  
\end{proof}
 
\begin{remark}
{ \rm 
Note that our hypothesis $\dim M\ge 7$ is essential for item i). In fact, for $\dim M = 3$ we have $\dim\hom (\Lambda^3\mathfrak{m},\R)=1$.}  
\end{remark}
The above explicit description of the basis of $\hom_{\mathfrak{h}}(\Lambda^3\mathfrak{m},\R)$ and Lemma
\ref{igualdadconjuntos} provide the following geometric description for all $G$-invariant  affine connections with skew-symmetric torsion on $3$-Sasakian homogeneous manifolds $M$ with $\dim M \ge 7$.

\begin{corollary}\label{co_lasskew}
Let $M=G/H$ be a 3-Sasakian homogeneous manifold with $\dim M \ge 7$. Then, any $G$-invariant  affine connection on $M$ with skew-symmetric torsion is given by   
\begin{enumerate}[i)]
\item If $G\neq \SU(m)$, 
\begin{equation}\label{eq_todasskewtorsion}
\nabla=\nabla^g+\frac12\left( a T^o +\sum_{r,s=1}^3 b_{rs} T^{rs}\right), 
\end{equation}
for some $a\in\R$ and $B=(b_{rs})\in  \mathcal{M}_{3}(\R)$.
\item If $M=\SU(m)/ S(\mathrm{U}(m-2)\times \mathrm{U}(1))$ with $m\ge3$, 
\begin{equation}\label{eq_todasskewtorsion2}
\nabla=\nabla^g+ \frac12\left( a T^o +\sum_{r,s=1}^3 b_{rs} T^{rs}+\sum_{l=1}^3c_lT^{l0}\right), 
\end{equation}
for some $a\in\R$, $(c_1,c_2,c_3)^t\in\R^3$ and $B=(b_{rs})\in  \mathcal{M}_{3}(\R)$.
\end{enumerate} 
\end{corollary}

As usual, for $B=(b_{rs})\in  \mathcal{M}_{3}(\R)$ and $c\in\R^3$, we use the notations $\|B \|^2=\sum_{r,s=1}^3b_{rs}^2$ and $\|c\|^2=\sum_{l=1}^3c_l^2$.     

 \subsection{Einstein with skew  torsion affine connections}\label{sec_Einstein}
 
In this section we determine which of the invariant affine connections with skew  torsion on a homogeneous $3$-Sasakian manifold   satisfy the Einstein with skew  torsion condition \eqref{einstein}.

Let us recall that the tensor $\s\in \mathcal{T}^{(0,2)}(M)$ defined in (\ref{eq_defS}) plays a key role in
the curvature identities for Riemann-Cartan manifolds with totally skew-symmetric torsion. The following result provides the expressions of the tensors $\s$ for the affine connections given in  (\ref{eq_todasskewtorsion2}).  Recall that the horizontal distribution is denoted by $ Q$ and the vertical distribution by $ Q^{\perp}=\mathrm{Span}\{\xi_1,\xi_2,\xi_3\}$.

\begin{proposition}\label{prop_Sys} 
Let $M^{4n+3}=G/H$ be a $3$-Sasakian homogeneous manifold as in Theorem \ref{th_lasSashomogeneas}, with $\dim M\geq 7$. Let $\nabla$ be an affine connection on $M=G/H$ given by (\ref{eq_todasskewtorsion2}), where $c=0$ when $G\neq \SU(m)$, $m\geq 3$.  
\begin{enumerate}
\item[i)] \label{tensorS}  The tensor $\s$ is given, if $X,Y $ are horizontal vector fields,  by 
$$
\begin{array}{l}
\s\vert_{ Q\times  Q^{\perp}}=0,\\
\s(\xi_i,\xi_k) =  2\delta_{ik}(a-\mathrm{tr}(B))^2+4n\sum_{s=1}^3 b_{is}b_{ks}+4nc_ic_k,\\
  \s(X,Y)=
2(\|B\|^2+\|c\|^2)g(X,Y)+4\sum_{s,j=1}^3 b_{js}c_jg(\varphi_0X,\varphi_sY).
\end{array}
$$
\item[ii)] The symmetric part of the Ricci tensor is given, if $X,Y $ are horizontal,  by 
\[
\begin{array}{l}
  \mathrm{Sym(Ric^{\nabla})}\vert_{   Q\times  Q^{\perp}}  =0,\\ 
\mathrm{Sym(Ric^{\nabla})}(\xi_i,\xi_k) = \left(4n+2-\frac12\left(a-\mathrm{tr}(B)\right)^2\right)\delta_{ik}-n\sum_{s=1}^3 b_{is}b_{ks}-nc_ic_k, \\
    \mathrm{Sym(Ric^{\nabla})}(X,Y)= \left(4n+2-\frac{\|B\|^2+\|c\|^2}{2}\right)g(X,Y)-\sum_{s,j=1}^3 b_{js}c_{j}g(\varphi_0X,\varphi_sY). 
    \end{array}
\]

\item[iii)] The scalar curvature is equal to
$$s^\nabla=
(4n+2)(4n+3) -	\frac32(a-\mathrm{tr}(B))^2-3n\left(\|B\|^2+\|c\|^2\right).
$$
\end{enumerate}
\end{proposition}
\begin{proof}
Denote by $T=a T^o +\sum_{r,s=1}^3 b_{rs} T^{rs}+\sum_{l=1}^3c_lT^{l0}$ the torsion of the connection $\nabla$, but put $c_l=0$ and $\varphi_0=0$ when $G\neq \SU(m)$. Recall that  
   the vertical distribution  is preserved by any $\varphi_s$, $s=0,\dots,3$. From Lemma~\ref{bases}, it is easy to check that 
\begin{equation}\label{eq_torsiones}
\begin{array}{ll}
T^o(X,Y)=0,\qquad&  T^{rs}(X,Y) =\Phi_s(X,Y)\xi_r,        \\
  T^o(X,\xi_j)=0,&   T^{rs}(X,\xi_j)= \delta_{rj}\varphi_sX,       \\
  T^o(\xi_i,\xi_{i+1})=\xi_{i+2},\qquad&   T^{rs}(\xi_i,\xi_{i+1})=-\delta_{rs}\xi_{i+2},   \\
\end{array}
\end{equation}
for any $X,Y\in\mathfrak{m}$   horizontal vectors,  $r=1,2,3$, $s=0,1,2,3$ (so $\delta_{r0}=0$). In particular, 
$T(X,Y)$ is vertical, 
$T(X,\xi_j)$ is horizontal, 
and $T(\xi_i,\xi_j)$ is vertical.

Then, by taking $\{\xi_1,\xi_2,\xi_3,V_1,\ldots,V_{4n}\}$ an orthonormal basis  of $\mathfrak{m}$,
it holds 
\[ \s(X,\xi_k) = \sum_{j=1}^3 g(T(X,\xi_j),T(\xi_k,\xi_j)) + \sum_{j=1}^{4n} g(T(X,V_j),T(\xi_k,V_j))=0. 
\]
For the sake of simplicity,  we write $b_{r0}=c_r$. Thus, we have 
$$ T(X,\xi_j)=\sum_{s=0}^3 b_{js}\varphi_sX,\quad
T(X,Y)=\sum_{r=1}^3\left[ \sum_{s=0}^3 b_{rs}g(X,\varphi_sY) \right]\xi_r.
$$
According to the definition of the tensor $\s$ in \eqref{eq_defS},
\begin{align*}
 &\ \ \s(X,Y)= \sum_{r=1}^3 g(T(X,\xi_r),T(Y,\xi_r)) + \sum_{j=1}^{4n} g(T(X,V_j),T(Y,V_j)) \nonumber \\
& \hspace{-3pt}= \sum_{r=1}^3\sum_{s,u=0}^3  g\left( b_{rs}\varphi_sX , b_{ru}\varphi_uY \right)    + \sum_{j=1}^{4n}\sum_{r,t=1}^3 g\left(  \sum_{s=0}^3b_{rs} g(X,\varphi_sV_j) \xi_r,
\sum_{u=0}^3 b_{tu}g(Y,\varphi_uV_j) \xi_t\right)\hspace{-1pt}. \nonumber
\end{align*} 
Taking into account that $\varphi_s\in\sof(\mm,g)$ and that   $g(X,Y)= \sum_{j=1}^{4n}  g( X,V_j)g( Y,V_j)$ holds for horizontal vectors, then both summands of the above formula are equal and then,  
\begin{align*}
\s(X,Y)&= 2\left[\sum_{s,r,u=1}^3 b_{rs}b_{ru} g(\varphi_sX,\varphi_u Y)
+ \sum_{s,r=1}^3 b_{rs}c_r\left( g(\varphi_sX,\varphi_0Y)+g(\varphi_0X,\varphi_sY) \right) \right] \\
&\ +2\|c\|^2 g(X,Y) .
\end{align*}
If we recall that $\varphi_s\varphi_u+\varphi_u\varphi_s=-2\delta_{su}\textrm{id}_{ Q}$ for $s,u =1,2,3$ and $\varphi_0\varphi_s=\varphi_s\varphi_0$, we get
$$
\s(X,Y)  =2(\|B\|^2+\|c\|^2)g(X,Y)+4\sum_{s,j=1}^3 b_{js}c_jg(\varphi_0X,\varphi_sY).  
$$

For the next steps, we recall from Section \ref{su} that 
$\varphi_0\varphi_1(z,w)=\varphi_0(\ii z,\ii  w)=(-z,w)$,  $\varphi_0\varphi_2(z,w)=\varphi_0(-\bar{w},\bar{z})=(-\mathbf{i}\bar{w},-\mathbf{i}\bar{z})$, $\varphi_0\varphi_3(z,w)=\varphi_0(-\mathbf{i}\bar{w},\mathbf{i}\bar{z})=(\bar{w},\bar{z})$, and also $\varphi_0\vert_{\mm_3}=0$. From the point of view of real linear spaces, the eigenvalues of $\varphi_0\varphi_1\vert_{\mm_1\oplus\mm_2}$ are $ 1,-1$, with equal number of $+1$ and $-1$. Then, $\mathrm{tr}(\varphi_0\varphi_1)=0$. Also, if we take an orthonormal $\R$-basis of $\mm_1\oplus\mm_2$ of the form $((e_1,0),\ldots,(e_n,0)$, $(\ii  e_1,0)$, $\ldots$, $(\ii  e_n,0)$, $(0,e_1)$, $\ldots,(0,e_n),(0,\ii  e_1),\ldots,(0,\ii  e_n))$, a simple computation shows that  $\mathrm{tr}(\varphi_0\varphi_2)=\mathrm{tr}(\varphi_0\varphi_3)=0$. 

We will compute  $S(\xi_i,\xi_k)$ in two steps. On one hand, 
\begin{align*}
\sum_{j=1}^{4n} &g(T(\xi_i,V_j),T(\xi_k,V_j)) 
= \sum_{j=1}^{4n}\left[ \sum_{s,u=0}^3 b_{is}b_{ku} g(\varphi_sV_j,\varphi_uV_j)\right]\\
& 
= \sum_{j=1}^{4n} \left[\sum_{s,u=1}^3 b_{is}b_{ku}\delta_{su}
+c_k \sum_{s=1}^3 b_{is}g(\varphi_sV_j,\varphi_0V_j)+ c_i \sum_{u=1}^3 b_{ku}g(\varphi_uV_j,\varphi_0V_j)
+ c_ic_k \right]  \\
&= 4n\sum_{s=1}^3 b_{is}b_{ks}-\mathrm{tr}\Big( \sum_s^3(c_kb_{is}+c_ib_{ks}) \varphi_s \varphi_0\Big)
+4nc_ic_k  = 4n\sum_{s=1}^3 b_{is}b_{ks}+4nc_ic_k.
\end{align*}
On the other hand, $T(\xi_i,\xi_{i+1})=-T(\xi_{i+1},\xi_{i})= \big(a-\sum_{r=1}^3{b_{rr}}\big)\xi_{i+2}.$ Then, it holds 
$$
\sum_{j=1}^3g(T(\xi_i,\xi_j),T(\xi_k,\xi_j))=\delta_{ik}( \| T(\xi_i,\xi_{i+1}) \|^2+ \| T(\xi_i,\xi_{i+2}) \|^2)=2\delta_{ik}(a-\mathrm{tr}(B))^2,
$$
which gives the required expression for   $\s(\xi_i,\xi_k)$.

Next, taking into account that every $3$-Sasakian manifold is Einstein (in the usual sense), 
 Eq.~(\ref{formulicas}) can be claimed to obtain the expression for the symmetric part of the Ricci tensor. Finally, the scalar curvature becomes
\begin{align*}
 & s^{\nabla}
=\sum_{k=1}^{4n+3}\mathrm{Sym(Ric}^{\nabla})(e_{k},e_{k})\\
& = 3\left[ 4n+2-\frac12\left(a-\mathrm{tr}(B)\right)^2\right]-n\|B\|^2-n\|c\|^2 +\left(4n+2-\frac12\left(\|B\|^2+\|c\|^2\right)\right)4n\\ &\quad +\sum_{s,j}b_{js}c_j\mathrm{tr}(\varphi_0\varphi_s) 
=(4n+2)(4n+3) -	\frac32(a-\mathrm{tr}(B))^2-3n\left(\|B\|^2+\|c\|^2\right).
\end{align*}
\end{proof} 

With this information, we can characterize which of our $3$-Sasakian homogeneous manifolds admit invariant Einstein with skew  torsion connections, 
and describe such connections $\nabla$ in terms of the coefficients $a\in\R$, $B=(b_{rs})\in \mathcal{M}_{3}(\R)$ and $c\in \R^3$ in (\ref{eq_todasskewtorsion2}). 

On one hand, as already stated, any 3-Sasakian manifold $(M^{4n+3},g)$ is Einstein with $\mathrm{Ric}^g=2(2n+1)g$. 
Therefore, by \eqref{formulicas}, the symmetrized of $\mathrm{Ric}^{\nabla}$  is a multiple of the metric $g$  if, and only if, the corresponding tensor  $\s$ is so. On the other hand, the scalar curvature $s^\nabla$ is a constant for every affine connection in (\ref{eq_todasskewtorsion2}). Hence, given $\nabla$ an invariant  Einstein with skew  torsion connection on the homogeneous $3$-Sasakian manifold $M=G/H$, there should exist $\lambda\in\R$ such that $\s=\lambda g$. Now, the identities    \eqref{einstein} and (\ref{formulicas})  imply that the tensor $\s$ corresponding with  such connection $\nabla$ satisfies
$$
\s=4\Big(2(2n+1)-\frac{ s^{\nabla}}{4n+3}\Big) g.
$$
By substituting here the expression of $s^\nabla$ provided by  Proposition~\ref{prop_Sys}\,iii)  (keep in mind that for convenience we are assuming $c=0$ whenever $G\neq \mathrm{SU}(m)$), we get 
\begin{equation}\label{unlanda}
\lambda=\frac{6\left( a-\mathrm{tr}(B)\right)^2 + 12n (\|B\|^2+\|c\|^2)}{4n+3}.
\end{equation}
A direct computation with the expression of $\s$ provided by  Proposition~\ref{prop_Sys}\,i) gives
\begin{equation}\label{treslanda}
3\lambda= \sum _{i=1}^3\s(\xi_{i},\xi_{i})=6\left( a-\mathrm{tr}(B)\right)^2  + 4n(\|B\|^2+\|c\|^2).
\end{equation}
From (\ref{unlanda}) and (\ref{treslanda}), we deduce that  every  Einstein with skew  torsion  invariant affine connection  $\nabla$ on the homogeneous $3$-Sasakian manifold $M^{4n+3}=G/H$ satisfies
\begin{equation}\label{eq_paradem1}
3\left( a-\mathrm{tr}(B)\right)^2  + (2n-3)(\|B\|^2+\|c\|^2)=0.
\end{equation}
This forces either $a=0$, $B=0$, $c=0$ (Levi-Civita connection) or $2n-3\le0$, so that $n=1$, in other words,  $\dim M=7$.

Now, we assume $\dim M=7$ and $3\left( a-\mathrm{tr}(B)\right)^2 =\|B\|^2+\|c\|^2$. By Proposition~\ref{prop_Sys}\, iii), 
$\mathrm{Sym}(\mathrm{Ric}^{\nabla})=\left(6- \frac{1}{2}(\|B\|^2+\|c\|^2)\right)g.$ 
Next, we use the formula of $ \mathrm{Sym}(\mathrm{Ric}^{\nabla})$ in Proposition~\ref{prop_Sys}\,ii), where we insert $(\xi_i,\xi_k)$, obtaining 
$\sum_sb_{is}b_{ks}+c_ic_k=\frac13(\|B\|^2+\|c\|^2)\delta_{ik}$, so that    
\begin{equation}\label{eq_paradem2}
BB^t+cc^t=\frac13(\|B\|^2+\|c\|^2) I_3.
\end{equation}
By inserting $X,Y$ horizontal, we get $c^tB=0$. Conversely, these   conditions (that is, $n=1$, \eqref{eq_paradem1}, \eqref{eq_paradem2} and $c^tB=0$) imply  that $\mathrm{Sym}(\mathrm{Ric}^{\nabla})=\frac{s^\nabla}{7} g$, as a direct application of Proposition~\ref{prop_Sys}\,ii).  We summarize these arguments  in the following theorem.

\begin{theorem}\label{dimn=1} 
Let $M=G/H$ be a 3-Sasakian homogeneous manifold of dimension at least $7$. 
\begin{enumerate}
\item[i)] If $M$ admits a nontrivial $G$-invariant affine Einstein with skew  torsion  connection, then  $\dim M=7$, that is, either $M=\Sp(2)/\Sp(1)\cong \mathbb{S}^7$, or   
 $M=\Sp(2)/(\Sp(1)\times\mathbb{Z}_2)\cong \R P^7$,  or  
$M=\SU(3)/S(\mathrm U(1)\times U(1))=\SU(3)/\mathrm U(1)$
is the Aloff-Wallach space $\mathfrak W_{1,1}^7$. 

\item[ii)] There is a bijective correspondence between the Lie group $\mathbb Z_2\times  {\mathrm{CO}}(3)$  and the set of nontrivial $\Sp(2)$-invariant affine Einstein with skew  torsion  connections on $M= \mathbb{S}^7$, given by 
\begin{equation}\label{eq_sonvalidas}
 (\pm1,B=(b_{rs}))\mapsto \nabla^g+ \frac12\left[\left(   \mathrm{tr}(B)\pm\sqrt{\frac{\|  B \|^2 }3}\right)T^o+ \sum_{r,s=1}^3 b_{rs}T^{rs}\right]. 
\end{equation} 
In addition, $s^\nabla= {42-\frac72\|  B \|^2}$. This also holds for the projective space $M=\R P^7$.
 
\item[iii)]  There  is a bijective correspondence between   the set of nontrivial   $\SU(3)$-invariant 
Einstein with skew  torsion affine connections  on  $M=\mathfrak W_{1,1}^7$ and the set 
\[\mathbb Z_2\times  \{(c,B ) \in\R^3\times\mathcal{M}_{3}(\R):  B  B^t+cc^t\in\R I_3,\,c^tB=0\},\] given  by   
\begin{equation}\label{eq_sonvalidas2}
 \big(\pm1,c=(c_l), B=(b_{rs})\big)\mapsto \nabla^g+ 
 \frac12 a_{\pm}T^o+\frac12 \sum_{r,s }  b_{rs}T^{rs}+ \frac12\sum_{l }  c_{l}T^{l0}, 
\end{equation}
where $a_{\pm}=\tr(B)\pm\sqrt{\frac{\| B \|^2+\| c \|^2 }3}$. Moreover, $s^\nabla=  {42-\frac72(\|  B \|^2+\| c \|^2)}$.
\end{enumerate}
\end{theorem}

 \begin{remark}
\rm{A relevant fact is that,  for an arbitrary 3-Sasakian manifold of dimension   $7$, any affine connection as in \eqref{eq_sonvalidas} is Einstein with skew  torsion.    Thus, the homogeneity is only required to assure that  the list is complete.}
 \end{remark}
 
In Theorem~\ref{dimn=1}, among the Einstein with skew  torsion invariant affine connections, the scalar curvature $s^\nabla$ equals zero if, and only if,  
$\|  B \|^2+\| c \|^2=12$.  Also, $ \mathrm{Sym}(\mathrm{Ric}^{\nabla})$ is positive definite (resp. negative definite) when $\|  B \|^2+\| c \|^2<12$ (resp. $>$). We are including here the cases   $\mathbb S^7$ and $\R P^7$, with $c=0$.   

Note that the set of Einstein with skew  torsion invariant connections in the Aloff-Wallach space $\mathfrak W_{1,1}^7$  contains strictly   $\mathbb Z_2\times\mathrm{CO}(3)$. For instance, we can take the  nonzero vector $c=(1,0,0)^t$ and the  noninvertible matrix $  B=\tiny\begin{pmatrix} 0&0&0\\ \cos\theta&-\sin\theta&0\\ -\sin\theta&-\cos\theta&0 \end{pmatrix}$.  Indeed, $c^tB=0$ and $BB^t+cc^t=I_3$. 

\begin{remark}\label{re_accion}
\normalfont Formulae \eqref{eq_sonvalidas} and \eqref{eq_sonvalidas2} are given in terms of an  arbitrary choice of three Sasakian compatible structures $\{\xi_{i} ,\eta_{i}, \varphi_{i}\}_{i=1,2,3}$ in ${\estS}= \{\xi_{\tau} ,\eta_{\tau}, \varphi_{\tau}\}_{\tau\in\mathbb S^2} $. {Given} a second choice $\{\xi'_{i} ,\eta'_{i}, \varphi'_{i}\}_{i=1,2,3}$ in ${\estS}$ such that $g(\xi'_{i},\xi'_{j})=\delta_{ij}$ and $[\xi'_{i},\xi'_{j}]=2 \epsilon_{ijk}\xi'_{k}$, the first condition is equivalent to the fact that the change of basis given by $\xi'_i=\sum_jp_{ij}\xi_j$ is an orthogonal transformation, that is, $P=(p_{ij})\in\mathrm {O}(3)$. The second condition is equivalent to $\det(P)=1$. Thus  a 3-Sasakian structure determines implicitly an orientation on the  vertical distribution,   and  we have a natural action of the special orthogonal group $\mathrm {SO}(3)$ on the set of compatible triplets of Sasakian structures.

If we  write a torsion  
$T= a T^o +\sum_{r,s=1}^3 b_{rs} T^{rs}+\sum_{l=1}^3c_lT^{l0}= a' T'^o+\sum_{r,s=1}^3 b'_{rs} T'^{rs}$ $+\sum_{l=1}^3c'_lT'^{l0}$, by a straightforward computation, we see that
$$
a'=\det(P)a=a,\qquad   B'=PBP^{-1},\qquad c'=Pc.
$$
This implies that the sets of Einstein with skew  torsion affine connections on $\mathbb S^7$ ($\R P^7$) and $\mathfrak W_{1,1}^7$ in Theorem~\ref{dimn=1} should be invariant under the actions 
$P\cdot(s,B)=( s, PBP^{-1})\in\mathbb Z_2\times \mathrm {CO}(3)$ and $P\cdot(s,c,B)=( s, Pc, PBP^{-1})$, respectively. This holds trivially, but we point it out because it   gives a clue to understand the great amount of connections in \eqref{eq_sonvalidas} and \eqref{eq_sonvalidas2}, which is a little bit surprising initially. Other properties which will be invariant under the $\mathrm {SO}(3)$-action, that is, independent of the  choice of three Sasakian compatible structures, are  the symmetry of the Ricci tensor  (Section~\ref{sec_riccisimetrico}) and  the condition of having parallel skew  torsion (Section~\ref{sec_paralela}).
\end{remark}
 
\begin{remark}\label{G2} \normalfont 
Many reported connections   fit in the set described in Theorem~\ref{dimn=1}. We mention some of these examples to serve as a link with the existent literature.
\begin{itemize}
\item[i)]
Recall from  \cite{AgriFri2} the existence of metric connections with skew  torsion and parallel spinors in any 3-Sasakian manifold in dimension   $7$, whose holonomy group is a subgroup of $G_2$. There, the connections were 
identified via $G_2$-geometry. Observe now that there is a possible choice of such connections in the family \eqref{eq_sonvalidas}.  Indeed, Agricola and Friedrich found in  \cite[Theorem~9.1 and Proposition~9.1]{AgriFri2} a family of metric connections $\nabla^T$ with skew  torsion $T=-\frac{1}{6}T^0+\sum_{r,s=1}^{3}b_{rs}T^{rs}$ which admit parallel spinors when the matrix $B$ satisfies 
$$ B=\rho\begin{pmatrix}
 a^2-b^2-c^2+d^2&2(ab+cd)&2(ac-bd)\\
 2(ab-cd)&-a^2+b^2-c^2+d^2&2(bc+ad)\\
 2(ac+bd)&2(bc-ad)&-a^2-b^2+c^2+d^2
 \end{pmatrix},
$$
$\rho=\frac{1}{6(a^2+b^2+c^2+d^2)}$, 
for certain $a,b,c,d$. A direct computation shows that $B\in  {\mathrm{CO}}(3)$ with $\|  B \|^2=1/12$.
Hence $ \tr(B)\pm\frac16=-\frac 16$ has solution if either $a^2+b^2+c^2=3d^2$ or $3(a^2+b^2+c^2)=d^2$.
Therefore, from Theorem \ref{dimn=1} and  \cite[Theorem~9.1 and Proposition~9.1]{AgriFri2}, every $3$-Sasakian manifold $M$ of dimension $7$ possesses Einstein with skew  torsion affine connections with parallel spinors.

\item[ii)] On the other hand,  a two-parametric family of metrics on the Aloff-Wallach space $\mathfrak W_{1,1}^7$ is introduced in \cite[Section 8]{AgriFri2}. The particular case $s=1$ and $y=2$ is just the $3$-Sasakian metric on  $\mathfrak W_{1,1}^7$. The seminal paper \cite{AgriFerr} includes some examples  of Einstein  with skew  torsion affine connections for the $3$-Sasakian metric on $\mathfrak W_{1,1}^7$. They are denoted $\nabla^3, \nabla^4$ and $\nabla^6$   in \cite[Section~8]{AgriFri2} with $s=1$ and $y=2$. Following the notation of  Theorem~\ref{dimn=1}, it is not difficult to check that these Einstein  with skew  torsion affine connections correspond to 
$$\begin{array}{ccccc}
\Big(-1, & c=0,& B=\frac{1}{6}\begin{pmatrix}
 0&-1& 0\\
 1& 0&0\\
 0&0&-1
 \end{pmatrix}\Big)&\mapsto& \nabla^3,\\
 \Big(-1, & c=0,& B=\frac{1}{6}\begin{pmatrix}
 0&1& 0\\
 -1& 0&0\\
 0&0&-1
 \end{pmatrix}\Big)&\mapsto& \nabla^4,\\
\Big(-1, & c=0,& B=\frac{1}{6}\begin{pmatrix}
 0&1& 0\\
 -1& 0&0\\
 0&0&1
 \end{pmatrix}\Big)&\mapsto& \nabla^6.
\end{array}$$
In particular, all three scalar curvatures   of $\nabla^3, \nabla^4$, and $\nabla^6$ are  equal to $s^{\nabla}=7\left(6-\frac{1}{24}\right).$ 
 
\item[iii)] Einstein with skew  torsion invariant connections on $\mathbb{S}^7$ have been studied  in \cite{griego} and \cite{DraperGarvinPalomo}, where the sphere has been regarded as $\mathbb{S}^7\cong \mathrm{Spin}(7)/G_2$ and $\mathbb{S}^7\cong \SU(4)/\SU(3)$, respectively. Since $\Sp(2)\subset \SU(4)\subset \mathrm{Spin}(7)$, those connections should appear in our list.
According to these works, the torsion of any nontrivial $\SU(4)$-invariant Einstein with skew  torsion connection\footnote{Be aware that the  expression of $\omega_{_{\nabla}}$ on the preliminar version arXiv:1503.08401v1  have a wrong coefficient $\frac12$, corrected in the   version \cite{DraperGarvinPalomo}.} 
is
$$
\omega_{_{\nabla}}=  r\, \eta_{1}\wedge d\eta_{1}+ \mathrm{Re}(q) \left( \eta_{2}\wedge d\eta_{2}- \eta_{3}\wedge d\eta_{3}\right)- \mathrm{Im}(q) \left(\eta_{2}\wedge d\eta_{3}+\eta_{3}\wedge d\eta_{2}\right),
$$
for $0\ne r\in\R$ and $q\in\mathbb C$ such that  $|q|^2=r^2$.
 Moreover, it is $\mathrm{Spin}(7)$-invariant if, and only if,  $q=-r$. That is, there is $\theta\in[0,2\pi)$ ($q=re^{\mathbf{i}\theta}$) such that they correspond to  the following pairs $(s,B)$ in $\mathbb Z_2\times \mathrm{CO}(3)$:
\begin{equation}\label{eq_lasSU4inv}
\Big( s=-1,\, B=2r\left(\begin{array}{ccc}
1&0&0\\0&\cos\theta&-\sin\theta\\0&-\sin\theta&-\cos\theta
\end{array}\right)\Big)
\ \textrm{ and }\ 
\Big(s=-1,\, B=2r\left(\begin{array}{ccc}
1&0&0\\0&-1&0\\0&0&1
\end{array}\right)\Big)
\end{equation}
respectively. Bearing in mind Remark~\ref{re_accion} about the special orthogonal group and its action on the set of Einstein with skew  torsion $\mathrm{Sp}(2)$-invariant affine connections, this means that any $\SU(4)=\mathrm{Spin}(6)$-invariant Einstein connection on $\mathbb{S}^7$ is $\mathrm{Spin}(7)$-invariant   for a convenient choice of $\{\xi_i,\eta_i,\varphi_i\}_{i=1}^3$.

\item[iv)]  Also, the canonical $G_2$-structure on $\mathbb{S}^7$ studied in \cite{AgriFri} is given by
\begin{equation}\label{eq_laG2est}
\omega_{_{G_2}}=\frac12  \left( \eta_{1}\wedge d\eta_{1}+   \eta_{2}\wedge d\eta_{2}+ \eta_{3}\wedge d\eta_{3}\right)+4  \eta_{1} \wedge \eta_{2}\wedge \eta_{3},
\end{equation}
so that $\omega_{_{G_2}}$ corresponds to $(s,B)=(+1,I_3)$ in \eqref{eq_sonvalidas}.
In particular, the affine connection with torsion $\omega_{_{G_2}}$ is Einstein with skew  torsion.

On the other hand, there exists a unique affine connection with skew  torsion  preserving such $G_2$-structure (the  {characteristic connection} of the $G_2$-structure) whose associated torsion 3-form is given by
\begin{equation*} 
\eta_{1}\wedge d\eta_{1}+   \eta_{2}\wedge d\eta_{2}+ \eta_{3}\wedge d\eta_{3} .
\end{equation*}
The corresponding affine connection   was generalized to arbitrary dimension in \cite{AgriDileo}, where is called \emph{the canonical connection of the 3-Sasakian manifold} (see Remark~\ref{re_canonica}).  We will denote it by $\nabla^{c}$. 
As it  does not correspond to any choice of $(s,B)$ in Theorem~\ref{dimn=1}, then  $\nabla^{c}$  is not an Einstein with skew  torsion affine connection. In spite of this fact, $\nabla^{c}$ plays a special role in any 3-Sasakian manifold (arbitrary dimension), as  Theorem~\ref{th_paralela} shows. 
\end{itemize}
\end{remark}  


We think  that further study of the     orbits under the action given in Remark~\ref{re_accion}, different from  those ones in \eqref{eq_lasSU4inv}, could be convenient. In fact,  their representatives could behave rather differently. If we denote by $R_\theta$ the rotation $R_\theta=\tiny\begin{pmatrix} \cos\theta&-\sin\theta\\\sin\theta&\cos\theta\end{pmatrix}$, 
the orbits of $\mathbb Z_2\times\mathrm{CO}(3)$ under the $\mathrm{SO}(3)$-action are precisely
$$\tiny\left\{(s,r\begin{pmatrix}1&0\\0&R_\theta\end{pmatrix}),(s,r\begin{pmatrix}-1&0\\0&R_\theta\end{pmatrix}):s={\pm1},r\in\R,\, r> 0,\theta\in[0,2\pi) \right\},
$$
where $r$ has been taken positive to avoid duplicity of the orbits. The orbits which have already appeared 
 in
  \eqref{eq_lasSU4inv} are $\tiny(-1, r\begin{pmatrix}-1&0\\0&R_{0}\end{pmatrix})$ and $\tiny(-1, r\begin{pmatrix}1&0\\0&R_{\pi}\end{pmatrix})$,   and, as mentioned,  $\tiny(1, \begin{pmatrix}1&0\\0&R_{0}\end{pmatrix})$ in \eqref{eq_laG2est}.

Coming back to the Aloff-Wallach space and its extra  Einstein  with skew  torsion invariant  connections, that is, those ones with $c\ne0$,  we can assume that $c=(r, 0, 0)^t$, $r>0$, due to the action of the special orthogonal group. 
Now, the first row of $B$ is equal to zero, and the two others can be any pair of orthogonal vectors in $\R^3$ of length $r$.
So $B$ is determined by an element in the Stiefel manifold $V_{3,2}$. 
 
\begin{remark}\normalfont 
Recall that there is a canonical way to associate a Lorentzian  metric with any Sasakian manifold, \cite[11.8.1]{galickiboyer}. If $(\xi,\eta,\varphi)$ is a Sasakian structure on a Riemannian manifold $(M,g)$, define $g_L=g-2\eta\otimes\eta$. This metric was considered on the sphere $\mathbb S^{2n+1}$ in \cite{casoLorentz}, where it was proved that there exist $\SU(n+1)$-invariant connections which are Einstein with skew  torsion  only for dimensions $3$ and $7$. According to \cite{casoLorentz} (respectively \cite{DraperGarvinPalomo}), the set of $\SU(4)=\mathrm{Spin}(6)$-invariant connections which are Einstein with skew  torsion on $(\mathbb S^7,g_L)$ (resp. $(\mathbb S^7,g)$) is parametrized by an ellipsoid (resp. by a cone). Obviously, the set of $\Sp(2)$-invariant connections must contain the set of $\SU(4)$-invariant connections, so that it arises the natural question of finding the whole  set,
in a similar way to Theorem~\ref{dimn=1}. It is not difficult to prove that the set of  $\Sp(2)$-invariant Einstein  with skew  torsion connections  on  $(\mathbb S^7,g_L)$  is parametrized by the group $\mathbb Z_2\times \mathrm{CO}(2,1)$, where  
 \[ \mathrm{CO}(2,1)=\{A\in \mathcal M_{3 }(\mathbb{R}) : AHA^t\in\mathbb{R} H, \ \det(A)\neq 0 \}, \qquad 
 H=\tiny\begin{pmatrix}-1&0\\0&I_2\end{pmatrix}.
\]
Here we are considering $g_L=g-2\eta_\tau\otimes\eta_\tau$ for any fixed $\tau\in\mathbb S^2$.
\end{remark}

\subsection{On the Symmetry of the Ricci Tensor}\label{sec_riccisimetrico} We study the symmetry of the Ricci tensors corresponding to our families of   affine connections with skew-symmetric torsion. In fact, from (\ref{formulicas}), the Ricci tensor of   $\nabla$ is symmetric just when the divergence of the torsion tensor of $\nabla$ vanishes.

\begin{proposition}\label{pr_div} The divergences of the torsion tensors defined in \eqref{def_torsiones} are:
$$\mathrm{div}(T^o)=0,\quad
\mathrm{div}(T^{rs})=\left\{\begin{array}{ll}
0&\textrm{ if   $s\in\{r,0\}$,}\\
2\Phi_{r+2} +(4n+2)\eta_r\wedge\eta_{r+1}&\textrm{ if   $s=r+1 $,}\\
-2\Phi_{r+1} +(4n+2)\eta_r\wedge\eta_{r+2}&\textrm{ if   $s=r+2 $.}
\end{array}
\right.
$$
In particular,  the tensor $T=aT^o+\sum_{r,s=1}^{3} b_{rs}T^{rs}+\sum_rc_rT^{r0}$ has zero divergence  if, and only if,  $B=(b_{rs})$ is a symmetric matrix.
\end{proposition}

\begin{proof} Given a point $p\in M$, let us consider an orthonormal local frame $(e_1,\ldots,e_{4n+3})$  at $p$ such that $\big(\nabla^g_{e_i}e_j\big)_p=0$. For $X,Y\in \mathfrak{X}(M)$, recall that
\[ \mathrm{div}(T)(X,Y) = \sum_{k=1}^{4n+3} g\left(e_k, \left(\nabla^g_{e_k}T\right)(X,Y)\right). 
\]
Since this is a tensorial expression, we reduce the computation to
\[\mathrm{div}(T)_p(e_i,e_j)= \sum_{k=1}^{4n+3} g_p\left(e_k, \nabla^g_{e_k}T(e_i,e_j)\right). 
\]
We will not write down the point for the sake of simplicity. We start with the computation of $\mathrm{div}(T^{rs})$, for $r,s\in\{1,2,3\}$. First,
\begin{align*}
& \sum_k g\left(e_k,\nabla^g_{e_k}\left(\eta_r(e_j)\varphi_se_i\right)\right) =   \sum_k \left[ 
g\left(\nabla^g_{e_k}\xi_r,e_j\right) g\left(e_k,\varphi_se_i\right) 
+\eta_r(e_j)g\left(e_k,\left(\nabla^g_{e_k}\varphi_s\right)e_i\right) \right] \\
& = \sum_k \left[ 
-g\left(\varphi_re_k,e_j\right) g\left(e_k,\varphi_se_i\right) 
+\eta_r(e_j)g\big(e_k,g(e_k,e_i)\xi_s-\eta_s(e_i)e_k\big)
\right] \\
& = g(\varphi_re_j,\varphi_se_i) +\eta_r(e_j)\left(g(e_i,\xi_s)-(4n+3)\eta_s(e_i) \right)\\
& = 
g(\varphi_re_j,\varphi_se_i) -(4n+2)\eta_r(e_j)\eta_s(e_i). 
\end{align*}
Next, 
\begin{align*}
&  \sum_k g\left(e_k, \nabla^g_{e_k}\left(g(e_i,\varphi_se_j)\xi_r\right)\right)   \\
& =\sum_k\left[  g\left(e_i,\left(\nabla^g_{e_k}\varphi_s\right)e_j\right)g(e_k,\xi_r)
+ g(e_i,\varphi_se_j) g\left( e_k, \nabla^g_{e_k} \xi_r\right) \right] \\
& = g(e_i,  \big(\nabla_{\sum_kg(e_k,\xi_r)e_k }^g\varphi_s\big) e_j) 
+g(e_i,\varphi_se_j) \mathrm{div}(\xi_r)  =g(e_i,\big(\nabla_{ \xi_r  }^g\varphi_s\big) e_j)  \\
& =  \eta_s(e_i)\eta_r(e_j)-\eta_s(e_j)\eta_r(e_i) = \eta_s\wedge\eta_r(e_i,e_j).  
\end{align*}
By joining the information, we get 
\begin{align*}
& \mathrm{div}(T^{rs})(e_i,e_j) = \sum_k g\left(e_k,\nabla^g_{e_k}\left( -\eta_r(e_i)\varphi_se_j+\eta_r(e_j)\varphi_se_i+g(e_i,\varphi_se_j)\xi_r 
\right)\right) \\
& =  g(e_i,(\varphi_r\varphi_s-\varphi_s\varphi_r)e_j)   +(4n+1)\eta_r\wedge\eta_s(e_i,e_j). \end{align*}
From here,  it is trivial that
$ \mathrm{div}(T^{rr}) =0.$ 
Also, as $[\varphi_r,\varphi_{r+1}]=2\varphi_{r+2}-\eta_r\otimes\xi_{r+1}+\eta_{r+1}\otimes\xi_{r}$, then 
\begin{align*}
\mathrm{div}(T^{r\,r+1})(e_i,e_j) 
=  2\Phi_{r+2}(e_i,e_j)   + (4n+2)\eta_r\wedge \eta_{r+1}(e_i,e_j).
 \end{align*}
Analogously, we obtain the required expression for $\mathrm{div}(T^{r,r+2})$.

Next, we compute $\mathrm{div}(T^o)$ in some steps. As $T^o(X,U)=0$ for any $X,U\in  \mathfrak{X}(M)$ 
whenever $X$ is  horizontal, then $\mathrm{div}(T^o)(X,Y)=0$  when $X,Y$  are horizontal. Also, the compatibility with the metric implies 
$g\left(\xi_l,\nabla^g_{e_k}X \right)=-g\left(\nabla^g_{e_k}\xi_l,X \right) = g\left(\varphi_l e_k,X \right) = -g\left(\varphi_lX,e_k \right)  $, so that
 \begin{align*}  &  
 \mathrm{div}(T^o)(X,\xi_i) =  
-\sum_k g\left( e_k, T^o(\nabla^g_{e_k}X,\xi_i)\right) 
   = \sum_k g\left(T^o\left(e_k,\xi_i\right),\nabla^g_{e_k}X \right) \\
& = \sum_{k,l} g\left(T^o\left(e_k,\xi_i\right),\xi_l\right)g\left(\xi_l,\nabla^g_{e_k}X \right) 
= -\sum_{k,l} g\left(T^o\left(\xi_i,\xi_l\right),e_k\right)g\left(\varphi_lX,e_k \right)=0, 
\end{align*}
since $\sum_l g\left(T^o\left(\xi_i,\xi_l\right),\varphi_lX\right)\in g( Q^\perp, Q)=0$.
For the last step,   consider a local orthonormal frame $\{e_k\}_{k=1}^{4n+3}$ such that  $e_k=\xi_k$, $k=1,2,3$. Recall that $\eta_i(e_k)=0$ for $k\geq 4$. As $\mathrm{div}(T^o(\xi_i,\xi_j))=0$, we get 
\begin{align*}
& \mathrm{div}(T^o)(\xi_i,\xi_j) 
 =  -\sum_{k=1}^{4n+3} \left[ \eta_1\wedge\eta_2\wedge\eta_3(\nabla^g_{e_k}\xi_i,\xi_j,e_k) 
+\eta_1\wedge\eta_2\wedge\eta_3(\xi_i,\nabla^g_{e_k}\xi_j,e_k) 
\right] \\
& = \sum_{k=1}^3 \left[ \eta_1\wedge\eta_2\wedge\eta_3(\varphi_i\xi_k,\xi_j,\xi_k) 
+\eta_1\wedge\eta_2\wedge\eta_3(\xi_i,\varphi_j\xi_k,\xi_k) 
\right] \\
& = 2 g(\xi_{i+1}\times\xi_{i+2},\xi_j)
- 2 g(\xi_{j+1}\times\xi_{j+2},\xi_i) = 2 g(\xi_i,\xi_j) -2 g(\xi_j,\xi_i) =0.
\end{align*}

Finally, in order to prove that $\mathrm{div}(T^{r0})=0$, $r=1,2,3$, we apply the algebraical tools in Section~\ref{se_3} 
to obtain the covariant derivative of the torsion $T^{r0}$. Lemma~\ref{le_comoderivar} allows to work algebraically, by using the bilinear map $\alpha^g$ related to the Levi-Civita connection obtained in \eqref{eq_alfadeLevi}. 
Again, $\{e_k\}_{k=1}^{4n+3}$ denotes an orthonormal basis of $\mm=T_oM$ such that $\xi_k=e_k$ for $k=1,2,3$. It is very simple to see $(\nabla^g_{\xi_k}T^{r0})(\xi_i,\xi_j)=0$. Also, given  $X,Y,Z\in\mm$ horizontal vectors, we compute
\begin{align*}
(\nabla^g_XT^{r0})(\xi_i,\xi_j)&=\alpha^g(X,0)- T^{r0}([X,\xi_i],\xi_j)-T^{r0}(\xi_i,[X,\xi_j]) \\
&=-\delta_{rj}\varphi_0([X,\xi_i])+\delta_{ri}\varphi_0([X,\xi_j])
=\delta_{rj}\varphi_0\varphi_iX-\delta_{ri}\varphi_0\varphi_jX, \\
\mathrm{div}(T^{r0})(\xi_i,\xi_j) & =\sum_kg(e_k,(\nabla^g_{e_k}T^{r0})(\xi_i,\xi_j))=\tr(\delta_{rj}\varphi_0\varphi_i-\delta_{ri}\varphi_0\varphi_j)=0.
\end{align*}
Second, we have 
\begin{align*}
(\nabla^g_XT^{r0})(Y,\xi_j)&=\tfrac12\delta_{rj}[X,\varphi_0Y]_\mm+g(Y,\varphi_0\varphi_j(X))\xi_r,\vspace{2pt}\\
(\nabla^g_{\xi_k}T^{r0})(Y,\xi_j)&=-\tfrac12T^{r0}(Y,[\xi_k,\xi_j])=-\varepsilon_{kjs}T^{r0}(Y,\xi_s)\in  Q, 
\end{align*}
so, as  $\varphi_0Y$   is horizontal and $\ad( Q)$ interchanges $ Q$ and $ Q^\perp$, then 
\[
 \mathrm{div}(T^{r0})(Y,\xi_j) =\textstyle \sum_{k=4}^{4n+3}\frac{1}{2}g(e_k,\delta_{rj}[e_k,\varphi_0Y]_\mm)\in g( Q, Q^\perp)=0.
\]
Third, we see 
$T^{r0}(\alpha^g(X,Y),Z)= \frac{1}{2}\sum_ig([X,Y]_\mm,\xi_i)T^{r0}(\xi_i,Z)=-\frac12g([X,Y]_\mm,\xi_r)\varphi_0(Z)$ 	
$=\Phi_r(X,Y)\varphi_0Z$,  which gives
\begin{align*}
(\nabla^g_XT^{r0})(Y,Z)&= -\Phi_0(Y,Z)\varphi_rX- \Phi_r(X,Y)\varphi_0Z+ \Phi_r(X,Z)\varphi_0Y, \\
(\nabla^g_{\xi_k}T^{r0})(Y,Z)&=\tfrac12\Phi_0(Y,Z)[\xi_k,\xi_r]= \varepsilon_{krs}\Phi_0(Y,Z)\xi_s,
\end{align*}
and since $\varphi_0$ commutes with $\varphi_r$, and these maps have zero trace, then
\begin{align*}
\mathrm{div}(T^{r0})(Y,Z)&=\textstyle \sum_{k=4}^{4n+3}g\big(e_k,-\Phi_0(Y,Z)\varphi_re_k- \Phi_r(e_k,Y)\varphi_0Z+ \Phi_r(e_k,Z)\varphi_0Y\big)\vspace{2pt}\\
&=-\Phi_0(Y,Z)\tr\varphi_r- \textstyle \sum_kg(e_k,\varphi_rY)g(e_k,\varphi_0Z)+ g(\varphi_rZ,\varphi_0Y)\vspace{2pt}\\
&= g(Y,(\varphi_r\varphi_0-\varphi_0\varphi_r)Z)=0.
\end{align*}
Bearing all this in mind,  the divergence of $T=aT^o+\sum_{r,s=1}^{3} b_{rs}T^{rs}+\sum_rc_rT^{r0}$ is zero if, and only if, $B$ is symmetric, because $\{\eta_r\wedge\eta_{r+1},\Phi_r:r=1,2,3\}$ is a linearly independent set. 
\end{proof}

These results are coherent with \cite[Remark~5.16]{DraperGarvinPalomo}, where the divergences of the $\SU(4)$-invariant skew-symmetric torsions  are studied. But of course, there is much more generality here.

\begin{corollary}\label{co_ricsime}
 The invariant connection $\nabla=\nabla^g+aT^o+\sum_{r,s=1}^{3} b_{rs}T^{rs}+ \sum_{l=1}^3c_lT^{l0}$ has symmetric Ricci tensor if, and only if,  $B=(b_{rs})$ is a symmetric matrix.
\end{corollary}

\subsection{$\estS$-Einstein affine connections}\label{sec_S-Einstein} Our aim is to analyse whether there are connections such that their Ricci tensors are multiple of the metric both in the horizontal and vertical distributions, but in general with different parameters. The following notion is clearly indebted to the notion of $\eta$-Einstein as in \cite[Definition 11.1.1]{galickiboyer}.
\begin{definition}\label{SEinstein}
If  $\estS=\{\xi_{\tau} ,\eta_{\tau}, \varphi_{\tau}\}_{\tau\in \mathbb{S}^{2}}$  is a    {$3$-Sasakian structure} on $(M,g)$, we will say that a (metric) affine connection with skew  torsion $\nabla$ is  \emph{${\estS}$-Einstein} if there are constants $\alpha,\beta\in\mathbb{R}$ such that
$$
\mathrm{Ric}^\nabla=\alpha g+\beta\sum_{k=1}^3\eta_k\otimes\eta_k.
$$
\end{definition}
Surprisingly,   there is a big amount of such connections in every  3-Sasakian    manifold.

\begin{theorem} \label{segundo}
Let $(M=G/H,g)$ be a homogeneous   3-Sasakian manifold with a 3-Sasakian structure $\estS$, being $\dim M\geq 7$. 
\begin{enumerate}[i)]
\item If $G\ne\SU(m)$, any nontrivial ${\estS}$-Einstein $G$-invariant affine connection on $M$  is given by 
$$
  \nabla^g+\frac12\left( aT^o+ \sum_{r,s=1}^3 b_{rs}T^{rs}\right),
$$ 
 for $a\in\R$ and $B\in\mathrm{CO}(3)$ such that $B=B^t$.

\item  If $G=\SU(m)$, any   ${\estS}$-Einstein $G$-invariant affine connection   
  on $M$  is given by 
$$
  \nabla^g+\frac12\left(  aT^o+ \sum_{r,s=1}^3 b_{rs}T^{rs}+\sum_{l=1}^3 c_lT^{l0}\right),
$$ 
 for $a\in\R$, $B\in \mathcal{M}_{3}(\R)$ and $c\in\R^3$ such that 
 \begin{equation}\label{eq_condicionSEins}
 B=B^t,\   BB^t+cc^t\in\R I_3, \ c^tB=0.
 \end{equation}
\end{enumerate}
\end{theorem}
 
\begin{proof} Take $T=aT^o+\sum_{r,s=1}^{3} b_{rs}T^{rs}+\sum_lc_lT^{l0}$ a skew-symmetric torsion related to a ${\estS}$-Einstein invariant affine 
${\estS}$-Einstein connection on $M$. We denote $c_l=0$ if $G\ne\SU(m)$.
First,   the Ricci tensor must be   symmetric, so that  $B=(b_{rs})$ is a symmetric matrix by Corollary~\ref{co_ricsime}.
As $M$ is Einstein (in the usual sense), the   ${\estS}$-Einstein condition can be written as  $S= {\lambda}g+\beta\sum_k\eta_k\otimes\eta_k$ for some scalars $\lambda, \beta\in\R$.
For vertical vectors,  Proposition~\ref{prop_Sys} gives
$$
(\lambda+ \beta)\delta_{ij}  =S(\xi_i,\xi_j)= 2(a-\mathrm{tr}(B))^2\delta_{ij}+4n(\textstyle \sum_{s=1}^3b_{is}b_{js}+c_ic_j),
$$
so that $cc^t+BB^t=\frac{2(a-\mathrm{tr}(B))^2-\lambda- \beta}{-4n} I_3$.
But, if $X,Y $ are horizontal, the identity 
$$
 \lambda g(X,Y)= S(X,Y)=
g\left(X,2(\|B\|^2+\|c\|^2) Y -4\textstyle \sum_{s,j=1}^3 b_{js}c_j\varphi_0\varphi_sY\right),
$$
jointly with the  nondegeneracy of $g$ and the linear independence of $\{\varphi_s\}_{s=0}^3$,
imply $\lambda=2(\|B\|^2+\|c\|^2)$ and $\sum_{s,j=1}^3 b_{js}c_j=0$, that is, $c^tB=0$.   The obtained conditions are of course sufficient.
\end{proof}

Observe that the set described in Eq.~\eqref{eq_condicionSEins} includes $c=0$, $B\in \mathrm{CO}(3)$, $B=B^t$. Besides, for $c\ne0$, condition 
\eqref{eq_condicionSEins} forces $\|B\|^2=2\|c\|^2$,  $BB^t+cc^t=\|c\|^2I_3$ and $\det(B)=0$.  

Again, we emphasize that the homogeneity is only necessary to guarantee that the found set is the whole set of invariant ${\estS}$-Einstein affine connections. But  the connections on Theorem~\ref{segundo}\,i)  are  ${\estS}$-Einstein affine connections  in any 3-Sasakian manifold of dimension strictly bigger than $3$. 

 The obtained set of invariant ${\estS}$-Einstein connections is invariant for the $\mathrm{SO}(3)$-action described in Remark~\ref{re_accion}. We expected this fact, because Definition~\ref{SEinstein} does not depend on the choice of the three compatible Sasakian structures.

\subsection{A Distinguished Connection}\label{distinguida}
Recall again the importance of the connections with skew  torsion on Sasakian manifolds (Remark~\ref{re_labuena}), due to the fact that the Levi-Civita connection does not parallelize the Reeb vector fields, while there are connections with skew  torsion which do it. We would like to give an  answer to the question of finding a connection that parallelizes all Reeb vector fields. 

\begin{theorem} \label{new}
Let $(M,g)$ be a  3-Sasakian manifold of dimension at least $7$, with   a 3-Sasakian structure ${\estS}=\{\xi_{\tau} ,\eta_{\tau}, \varphi_{\tau}\}_{\tau\in \mathbb{S}^{2}}$.
There exists an affine connection with skew  torsion $\nabla^{\estS}$  on $M$  satisfying
 $$
 \nabla^{\estS} \xi_\tau=0
 $$
 for any $\tau\in \mathbb{S}^{2}$, whose  torsion  is  given by 
 $$
T^{\nabla^{\estS}}= 4T^o+ 2\sum_{r=1}^3 T^{rr}. 
$$ 
The above connection $\nabla^{\estS}$ is:
\begin{itemize}
\item  Einstein   with skew  torsion, with symmetric Ricci tensor, if $\dim M=7$;
\item ${\estS}$-Einstein for arbitrary dimension.
\end{itemize}
 Furthermore, if $M=G/H$ is homogeneous, this is the unique $G$-invariant affine connection with skew  torsion satisfying that every $\xi_\tau$ is parallel.
 \end{theorem}

\begin{proof}
Let $\nabla$ be the affine connection on $M=G/H$ given by (\ref{eq_todasskewtorsion2}), where $c=0$ when $G\neq \SU(m)$, $m\geq 3$. 
Using Eq.~\eqref{eq_torsiones}, we get, for a horizontal vector field $X\in Q $,
\begin{equation}\label{eq_cuenta}
\nabla_X\xi_i=-\varphi_iX+\frac12\left(\sum_{s=1}^3b_{is}\varphi_sX+c_i\varphi_0X\right).
\end{equation}
Thus, if $\nabla_X\xi_i=0$ for any $i$, the linear independence of $\{\varphi_s\}_{s=0}^3$ gives $c=0$ and $B=2I_3$. By this, 
$$
\nabla_{\xi_{i+1}}\xi_i=-\varphi_i\xi_{i+1}+\frac12\left(aT^o(\xi_{i+1},\xi_{i})+\sum_{r=1}^32 T^{rr}(\xi_{i+1},\xi_{i})\right)=
\left(2-\frac{a}{2}\right)\xi_{i+2}, 
$$
which vanishes if, and only if,  $a=4$. This proves the uniqueness of $\nabla^{\estS}$ in the homogeneous case. All that remains is clear. 
\end{proof}

This connection $\nabla^{\estS}$   does not coincide with   any of the connections mentioned in Remark~\ref{G2}, since now $(s,B)=(-1,2I_3)\in\mathbb Z_2\times  {\mathrm{CO}}(3)$ following the correspondence  \eqref{eq_sonvalidas}.    

It is important to remark  that the connection $\nabla^{\estS}$ does not parallelize the endomorphism fields $\varphi_\tau$.        
{But in the next Corollary~\ref{le_ficompatible} we check that $\nabla^{\estS}$ satisfies the  following slightly weaker condition proposed in \cite[Definition 3.1.1]{AgriDileo} for a larger class of manifolds, the almost 3-contact metric manifolds. 
\begin{definition}
Let $(M,g, {\estS})$ be a  3-Sasakian manifold and $\varphi=\varphi_\tau$ for some $\tau\in\mathbb S^2$. A metric connection with skew  torsion $\nabla$ is called a \emph{$\varphi$-compatible connection}  if
\begin{itemize}
\item[i)] $\nabla_XY$ is horizontal if $Y$ is a horizontal vector field, and is vertical if $Y$ is a vertical vector field,  for any $X\in \mathfrak{X}(M)$; 
\item[ii)] $(\nabla_X\varphi)(Y)=0$ for any $X,Y$ horizontal vector fields.
\end{itemize}
\end{definition}
On an almost 3-contact metric manifold, the torsion of any $\varphi$-compatible connection is determined by $\gamma=\omega_{_\nabla}(\xi_1,\xi_2,\xi_3)$, according to \cite[Theorem~3.1.1 and Remark~3.1.2]{AgriDileo}. 
Our aim of finding explicit comparisons to the current literature, makes us wonder which of our invariant affine connections are $\varphi$-compatible.
\begin{corollary}\label{le_ficompatible} 
Let $\nabla=\nabla^g+ \frac12\left( a T^o +\sum_{r,s=1}^3 b_{rs} T^{rs}+\sum_{l=1}^3c_lT^{l0}\right)$ be an affine connection on a 3-Sasakian homogeneous manifold. Then 
\begin{itemize}
\item[a)] $\nabla$ is $\varphi$-compatible for  $\varphi=\varphi_\tau$ for some $\tau\in\mathbb S^2$ if and only if $B=2I_3$ and $c=0$. In such  case,  
$\gamma=\omega_{_\nabla}(\xi_1,\xi_2,\xi_3)=a-6$. Moreover, all these connections are $\varphi_\tau$-compatible for any $\tau\in\mathbb S^2$ and are ${\estS}$-Einstein too.
\item[b)] If $\dim M=7$, $\nabla$ is $\varphi$-compatible for  $\varphi=\varphi_\tau$ for some $\tau\in\mathbb S^2$ and Einstein with skew  torsion if and only if $a\in\{4,8\}$, $B=2I_3$ and $c=0$. The first connection is 
$\nabla^{\estS}$ and the second one has torsion 3-form equal to
$$
\omega=2\omega_{_{G_2}}=\left( \eta_{1}\wedge d\eta_{1}+   \eta_{2}\wedge d\eta_{2}+ \eta_{3}\wedge d\eta_{3}\right)+8  \eta_{1} \wedge \eta_{2}\wedge \eta_{3},
$$
which is a $G_2$-structure on $M$ (compare with \eqref{eq_laG2est}).
\end{itemize}  
\end{corollary}
\begin{proof} For a horizontal vector field $X$, by \eqref{eq_cuenta}, $\nabla_X\xi_i$ is horizontal too.
If our connection  $\nabla$ satisfies condition i) in the above definition,   then $\nabla_X\xi_i=0$. Hence, as in the previous proof, $c=0$ and $B=2I_3$. Now a) is easily achieved. For b) we use Theorem~\ref{dimn=1}, so that
$
a= \mathrm{tr}(B)\pm\sqrt{\|  B \|^2 /3}=6\pm2\in\{4,8\}.$
\end{proof}

\begin{remark}\label{re_canonica}\normalfont 
The  above results agree with \cite{AgriDileo}. Indeed,  following its notation, 
we can apply \cite[Proposition~3.2.2]{AgriDileo} ($\alpha=\delta=1$) to 
any 3-Sasakian homogeneous manifold $M$, so that $M$ admits $\varphi$-compatible connections such that $\nabla_X\xi_i=0$  for $\gamma=-2$. 
We can identify a relevant affine connection in that paper, called the {canonical connection}.
According to \cite[Theorem~4.1.1]{AgriDileo}, the \emph{canonical connection} is the (metric)  affine connection with skew  torsion   characterized by 
\begin{equation}\label{eq_laquesalemucho}
\nabla_X\varphi_i=\beta(\eta_{i+2}(X)\varphi_{i+1}-\eta_{i+1}(X)\varphi_{i+2})
\end{equation}
for the Reeb Killing function $\beta$. In fact, $\nabla$ is $\varphi$-compatible. Let us check that this canonical connection   coincides with $\nabla^{c}$ in Remark~\ref{G2} iv), whose torsion is
\begin{equation*} 
\omega^c= \sum_{i=1}^3 \eta_{i}\wedge d\eta_{i}.
\end{equation*}
According to the uniqueness, it is enough to check that $\nabla^{c}$  satisfies Eq.~\eqref{eq_laquesalemucho}. Instead, we are checking that, if $\nabla$ is an invariant $\varphi$-compatible connection satisfying Eq.~\eqref{eq_laquesalemucho}, then necessarily $\nabla=\nabla^{c}$. First, $\nabla$ has the form \eqref{eq_todasskewtorsion2} with $B=2I_3$ and $c=0$, as above. On one hand, with the help of \eqref{eq_torsiones}, 
$$
\begin{array}{l}
\nabla^g_{\xi_{i+1}}\varphi_i(Y) =\eta_{i+1}(Y)\xi_{i}-\eta_{i}(Y)\xi_{i+1}=\varphi_{i+2}(-Y^v),\vspace{1pt}\\
 T^o(\xi_{i+1},\varphi_i(Y))-\varphi_i(T^o(\xi_{i+1},Y))=\varphi_{i+2}(-Y^v),\vspace{1pt}\\
 T^{rr}(\xi_{i+1},\varphi_i(Y))-\varphi_i(T^{rr}(\xi_{i+1},Y))= \varphi_{i+2}(Y^v+2\delta_{r,i+1}Y^h),\end{array}
$$
for any $r=1,2,3$ and $Y^v$ and $Y^h$ the vertical and horizontal projection of any vector field $Y$. Hence $\nabla_{\xi_{i+1}}\varphi_i(Y)=(-1-\frac{a}2+3)\varphi_{i+2}(Y)$ if $Y$ is vertical, while  $\nabla_{\xi_{i+1}}\varphi_i(Y)=2\varphi_{i+2}(Y)$ if $Y$ is horizontal. On the other hand,    \eqref{eq_laquesalemucho} gives $\nabla_{\xi_{i+1}}\varphi_i=-\beta\varphi_{i+2}$.
Hence  $2=-\beta=-1-\frac{a}2+3$ and $a=0$, 
that is, $\nabla=\nabla^{c}$.
\end{remark}

Thus, the results in Corollary~\ref{le_ficompatible} a) say that an invariant affine connection with skew  torsion,  $\nabla$,
on a 3-Sasakian homogeneous manifold is $\varphi$-compatible if and only if $T^\nabla-T^{\nabla^{c}}\in\mathbb R T^o$.

Observe that, although the different ways of distinguishing an affine connection well adapted to the 3-Sasakian structure have   led to different connections, there are a few candidates which appear frequently.  For instance,    $\nabla^{c}$ is not an Einstein with skew  torsion connection, but  $\nabla^{c}T^{\nabla^{c}}=0$ and it also satisfies Eq.~\eqref{eq_laquesalemucho}. All of these \emph{persistent} affine connections  turn to be invariant metric with skew  torsion (and ${\estS}$-Einstein) connections in the homogeneous case.

\begin{remark}\label{111}
{\rm 
More general metric connections on $3$-Sasakian manifolds with totally skew-symmetric torsion only for tangent vectors in the horizontal distribution   have been proposed in \cite{Cappeletti}.
From (\ref{eq_torsiones}), it is a direct computation that the torsion tensor $T^{\nabla^{\estS}}$ given in Theorem~\ref{new} satisfies, for any $X,Y $ horizontal vector fields,
$$\begin{array}{l}
T^{\nabla^{\estS}}(X,Y)=2\sum_{r=1}^3\Phi_{r}(X,Y)\xi_{r},\quad 
 T^{\nabla^{\estS}}(X, \xi_{i})=4\varphi_{i}(X) , \\
 T^{\nabla^{\estS}}(\xi_{i}, \xi_{i+1})=2 \xi_{i}\times \xi_{i+1}-4\xi_{i+2}=-2 \xi_{i+2}.
 \end{array}
 $$
This means that $\nabla^{\estS}$ does not satisfy the conditions in  \cite[Definition 4.1]{Cappeletti}. In fact, it can be checked that  none of the invariant affine connections in \eqref{eq_todasskewtorsion2} does it.
}
\end{remark}

\subsection{Parallel torsions} \label{sec_paralela} This subsection is devoted to   the invariant affine connections on $3$-Sasakian manifolds with parallel skew-symmetric torsion. Some examples of invariant affine connections  with nonzero parallel skew  torsion on $3$-Sasakian manifolds have  previously appeared. Namely,    $\nabla_{\tau}^{ch}$ in Remark~\ref{re_labuena} and $\nabla^{c}$ in Remark~\ref{G2}.  There is also a family of such connections on the sphere $\mathbb S^7$, recently  found in  \cite[Remark~5.16]{DraperGarvinPalomo}. Next we will prove that there exists a similar family  in the Aloff-Wallach space 
$\mathfrak{W}^{7}_{1,1}$ too, and that these examples essentially exhaust the invariant examples on the homogeneous 3-Sasakian manifolds.
\begin{lemma} \label{cuentasauxiliares} Given $X,Y,Z\in \mm$, 
\begin{enumerate}[(a)]
\item $(\nabla_XT)(Y,Z) = (\nabla^g_XT)(Y,Z)+\displaystyle \frac{ T(X,T(Y,Z)) + T(Y,T(Z,X))+T(Z,T(X,Y))}{2}.$ 

\item For $X,Z\in\g_{1}$, $\eta_r([X,Z]_{\mm}) =2g(\varphi_r X,Z).$
\end{enumerate}
\end{lemma} 
\begin{proof} Firstly, given $X,Y,Z\in \mm$, 
\begin{align*} 
& (\nabla_XT)(Y,Z) = \nabla_XT(Y,Z)- T(\nabla_XY,Z)-T(Y,\nabla_XZ) \\ \nonumber 
& = \nabla^g_XT(Y,Z)+\frac12 T(X,T(Y,Z)) - T(\nabla^g_XY,Z) -\frac12 T(T(X,Y),Z)  -T(Y,\nabla^g_XZ) \\
&\quad -\frac12 T(Y,T(X,Z)) \\
& =(\nabla^g_XT)(Y,Z)+\frac12\left[ T(X,T(Y,Z)) + T(Y,T(Z,X))+T(Z,T(X,Y))\right]. 
\end{align*}
Secondly, for $X,Z\in\g_{1}$, then  
$$\eta_r([X,Z]_{\mm}) =g(\xi_r,[X,Z])=2g([\xi_r,X],Z)=2g(\varphi_r X,Z),$$
by recalling \eqref{eq_nuestrag} and the associativity of the Killing form $\kappa$.
\end{proof}

\begin{theorem}\label{th_paralela}  
Given $M=G/H$ a 3-Sasakian homogeneous manifold of dimension at least 7, let $\nabla$ be a $G$-invariant metric affine connection on $M$ with nonzero skew  torsion $T$. Then $T$ is parallel with respect to $\nabla$, i.~e. $\nabla T=0$ if, and only if, $\nabla$ is one of the following:
\begin{itemize}
\item[i)] The canonical connection of the 3-Sasakian manifold, that is, the one with torsion $ \omega^c= \sum_i\eta_i\wedge d\eta_i$;
\item[ii)] The characteristic connection of one of the involved Sasakian structures, that is, the one with torsion
$\omega_{\tau}^{ch}= \eta_\tau\wedge d\eta_\tau$ for some $\tau\in\mathbb S^2$;
\end{itemize}
and, under the additional assumption $\dim M=7$, also
\begin{itemize}
\item[iii)] The connection with associated torsion   
\begin{equation}\label{eq_paralela7}
 \omega_{_{\nabla}}=  \frac13\left(\eta_{\tau_1}\wedge d\eta_{\tau_1}+ \eta_{\tau_2}\wedge d\eta_{\tau_2}- \eta_{\tau_3}\wedge d\eta_{\tau_3}\right),
\end{equation}
for some positively oriented orthonormal basis $\{{\tau_1},{\tau_2},{\tau_3}\}$ of $\R^3$. 
\end{itemize}
\end{theorem}
\begin{proof}  We choose any of the torsions in \eqref{eq_todasskewtorsion},  or \eqref{eq_todasskewtorsion2}, 
$T=aT^o+\sum_{r=1,s=0}^3 b_{rs}T^{rs}$, and construct the corresponding invariant affine connection $\nabla=\nabla^g+\frac12 T$. (We substitute $c_l$ by $b_{l0}$ to simplify the notation.)
We have to prove that,  if $\nabla T=0$, then $a=0$, $c=(b_{10},b_{20},b_{30})^t=0$,  the matrix $0\ne B=(b_{rs})\in  \mathcal{M}_{3}(\R)$ is symmetric, and either:
\begin{enumerate}[(i)] 
\item \label{a2} $B=2I_3$; 
\item \label{a3} $B=P\,\mathrm{diag}\{2,0,0\}P^{t}$ for some $P\in \SO(3)$;
\item \label{a4} $\dim M=7$,  and $B=\frac23P\,\mathrm{diag}\{1,1,-1\}P^{t}$ for some $P\in \SO(3)$.
 \end{enumerate}

The symmetry of $B$ is consequence of Proposition \ref{pr_div}, because, as already pointed out, $\nabla T=0$ implies $\mathrm {div}(T)=0$. From now, until the end of the proof, $X,Y,Z\in \g_1$.  

The second term of the formula of item (a), Lemma \ref{cuentasauxiliares},  is a cyclic sum, so that\begin{equation}\label{eq_a=0}
0=(\nabla_XT)(Y,\xi_k)-(\nabla_{\xi_k}T)(X,Y) = (\nabla^g_XT)(Y,\xi_k)-(\nabla^g_{\xi_k}T)(X,Y).
\end{equation}
We develop the following  expressions by using   Lemma~\ref{le_comoderivar}.
$$
\begin{array}{cll}
\circ&
 (\nabla^g_XT^o)(Y,\xi_k)&=\alpha^g(X,T^o(Y,\xi_k))-T^o(\alpha^g(X,Y),\xi_k)-T^o(X,\alpha^g(Y,\xi_k))\vspace{1pt} \\
&& = -\frac12 T^o([X,Y]_{\mm},\xi_k)-T^o(X,[Y,\xi_k]_{\mm}) = \frac14 [[X,Y]_{\mm},\xi_k] ;  \\ 
\circ&(\nabla^g_{\xi_k}T^o)(X,Y)& = 
0;\vspace{1pt}\\
\circ&(\nabla^g_XT^{rs})(Y,\xi_k)&=  \delta_{rk}\alpha^g(X,\varphi_sY) - \frac12 T^{rs}([X,Y]_{\mm},\xi_k) + T^{rs}(Y,\varphi_kX) \vspace{1pt}\\
&& = \frac{\delta_{rk}}{2}[X,\varphi_sY]_{\mm}+\frac{\delta_{rs}}{4}[[X,Y]_{\mm},\xi_k]+\Phi_{s}(Y,\varphi_kX)\xi_r;\vspace{1pt}\\
 \circ&(\nabla^g_ {\xi_k}T^{rs})(X,Y)&=\alpha^g(\xi_k,\Phi_s(X,Y)\xi_r)=\frac12  \Phi_{s}(X,Y)[\xi_k,\xi_r] .
\end{array}
$$
Now, we   insert them in \eqref{eq_a=0},
\begin{align*}
0&=a ((\nabla^g_XT^o)(Y,\xi_k)-(\nabla^g_{\xi_k}T^o)(X,Y) )  +\sum_{s\ge0,r}b_{rs}((\nabla^g_XT^{rs})(Y,\xi_k)-(\nabla^g_{\xi_k}T^{rs})(X,Y) ) \\
& =\scriptsize{\frac{\tr(B)-a}{4}}\large[[X,Y]_{\mm},\xi_k] + \sum_{s\ge0} \frac{b_{ks}}{2} [X,\varphi_sY]_{\mm} \\
&\qquad +\sum_{s\ge0,r} b_{rs}g(X,\varphi_k\varphi_sY)\xi_r
 -\sum_{s\ge0,r}\frac{b_{rs}}{2}\Phi_s(X,Y)[\xi_k,\xi_r].
\end{align*}
By applying $\eta_{k+1}$, and remembering that $B$ is symmetric, 
\begin{align*}
0&= (\tr(B)-a)g(\varphi_{k+2}X,Y)\\ 
&\quad +\sum_{s\ge0}\left[ b_{ks}g(\varphi_{k+1}X,\varphi_sY) + b_{k+1,s}g(X,\varphi_k\varphi_sY) 
+ b_{k+2,s}\Phi_{s}(X,Y)\right] \\
&=g\Big(X,(a-\tr(B))\varphi_{k+2}Y -\sum_{s\ge0} b_{ks}\varphi_{k+1}\varphi_sY +\sum_{s\ge0} b_{k+1,s}\varphi_k\varphi_sY 
+\sum_{s\ge0} b_{k+2,s}\varphi_sY\Big) \\
& =  g\left(X, a\varphi_{k+2} Y-b_{k,0}\varphi_{k+1}\varphi_{0}Y
+b_{k+1,0}\varphi_{k}\varphi_{0}Y
+b_{k+2,0} \varphi_{0}Y\right).
\end{align*}
Since the restrictions to $\mathfrak g_1$ of the maps $\varphi_{k+2}$, $\varphi_0$, $\varphi_0\varphi_k$ and $\varphi_{0}\varphi_{k+1}$ are linearly independent in case $G=\SU(m)$ (see Eqs.~\eqref{eq_accionesficasoU} and \eqref{eq_accionesficasoU_0}), then
$a=0=c_r$ for any $r=1,2,3$ and any $G$. 

Thus, the torsion $T= \sum_{r,s}  b_{rs}T^{rs}$ satisfies
\begin{equation}\label{tensoresT} 
\begin{array}{l}
T(X,Y)=\sum_{r,s\ge1} b_{rs} g(X,\varphi_sY)\xi_r,  \\
  T(X,\xi_k)=\sum_{s\ge1} b_{ks} \varphi_sX, \\  T(\xi_k,\xi_{k+1})=-\tr(B)\xi_{k+2}.
\end{array}
\end{equation}
The following step is to prove
\begin{equation}\label{eq_laclave}
\left(\frac{\tr(B)}{2} -1\right)B^2=\det(B)I_3.
\end{equation}
  To that end, we use $g\big((\nabla_XT)(Y,\xi_k),\xi_{k+1}\big)=0$. First we compute 
\begin{align*}
& g\big(\xi_{k+1},T(X,T(Y,\xi_k))+T(Y,T(\xi_k,X))+T(\xi_k,T(X,Y))\big)\vspace{2pt} \\
& = -g\big(T(Y,\xi_k),T(X,\xi_{k+1}\big) 
-g\big(T(\xi_k,X),T(Y,\xi_{k+1}\big)  -g\big(T(X,Y),T(\xi_k,\xi_{k+1}\big) \vspace{2pt}\\
& = \sum_{r,s} \left[b_{kr}b_{k+1,s}-b_{ks}b_{k+1,r}\right] g(\varphi_sY,\varphi_rX)
+ \tr(B)\sum_s b_{k+2,s}g(X,\varphi_sY) ;\\
&g\big((\nabla^g_XT)(Y,\xi_k),\xi_{k+1}\big)
 \vspace{2pt}\\
& = \frac12\sum_s b_{ks} g(\xi_{k+1},[X,\varphi_sY]_{\mm}) 
+\frac12 g([X,Y]_{\mm},T(\xi_{k+1},\xi_k)) 
-g(\xi_{k+1},T(Y,-\varphi_kX))  \vspace{2pt}\\
& = \sum_s b_{ks} g(\varphi_{k+1}X,\varphi_sY) 
+\frac12\tr(B) g([X,Y]_{\mm},\xi_{k+2})) 
+\sum_{r,s} b_{rs} g(Y,\varphi_s\varphi_kX)g(\xi_{k+1},\xi_r)  \vspace{2pt}\\
& = -\sum_s b_{ks} g(X,\varphi_{k+1}\varphi_sY) 
-\tr(B) g(X,\varphi_{k+2}Y) 
+\sum_{s} b_{k+1,s} g(X,\varphi_k\varphi_sY).
\end{align*}

All together, and using again that $B$ is symmetric, 
\begin{align*}
0&=g\big(\xi_{k+1},(\nabla_XT)(Y,\xi_k)\big) \vspace{4pt}\\ 
& =- \sum_s b_{ks} g(X,\varphi_{k+1}\varphi_sY) 
-\tr(B) g(X,\varphi_{k+2}Y) 
+  \sum_{s} b_{k+1,s} g(X,\varphi_k\varphi_sY) \\
& \quad +\frac12 \sum_{r,s} \left[b_{kr}b_{k+1,s}-b_{ks}b_{k+1,r}\right] g(\varphi_sY,\varphi_rX)
+ \frac{\tr(B)}{2}   \sum_s b_{k+2,s}g(X,\varphi_sY) \vspace{4pt}\\
&=-\textstyle \sum_s b_{ks} g(X,\varphi_{k+1}\varphi_sY) 
-\tr(B) g(X,\varphi_{k+2}Y) 
+\textstyle \sum_{s} b_{k+1,s} g(X,\varphi_k\varphi_sY) \\
& \quad -\textstyle \sum_{r<s} \left[b_{kr}b_{k+1,s}-b_{ks}b_{k+1,r}\right] g(X,\varphi_r\varphi_sY)
+ \frac{\tr(B)}{2}\sum_s b_{s,k+2}g(X,\varphi_sY)  \\
\end{align*}
\begin{align*}
& = 
\textstyle\left[\left(\frac{\tr(B)}{2}-1\right)b_{k,k+2}-b_{k,k+1}b_{k+1,k+2}+b_{k,k+2}b_{k+1,k+1}\right]\Phi_k(X,Y) \vspace{5pt}\\
&\textstyle\quad+ \left[ 
\left(\frac{\tr(B)}{2}-1\right)b_{k+1,k+2} +b_{kk}b_{k+1,k+2}-b_{k,k+2}b_{k,k+1}\right]\Phi_{k+1}(X,Y) \\
&\textstyle\quad +\left[ 
\left(\frac{\tr(B)}{2}-1\right)b_{k+2,k+2}+b_{k,k+1}^2-b_{kk}b_{k+1,k+1}\right] \Phi_{k+2}(X,Y).
\end{align*}
Since $\Phi_1,\Phi_2,\Phi_3$ are linearly independent, putting $\alpha=\tr(B)/2-1$, then 
\[\left\{
\begin{array}{lclcll}
\alpha\, b_{k,k+2}&-&b_{k,k+1}b_{k+1,k+2}&+&b_{k,k+2}b_{k+1,k+1}&=0 \\
\alpha\, b_{k+1,k+2}& +&b_{kk}b_{k+1,k+2}&-&b_{k,k+2}b_{k,k+1}&=0 \\
\alpha\, b_{k+2,k+2}&-&b_{kk}b_{k+1,k+1}&+&b_{k,k+1}^2&=0.
\end{array}
\right.\]
By denoting by $\mathrm{Adj}(B)$ the adjugate matrix of $B$,  these equations are equivalent to
$\alpha B-\mathrm{Adj}(B)=0$. By myltiplying both sides of this equation by $B$, we obtain Eq.~\eqref{eq_laclave}.
\smallskip

$\ast$ If $\alpha=0$, then $\tr(B)=2$ and $\mathrm{Adj}(B)=0$. Since all cofactors of order 2 vanish, then $\mathrm{rank}(B)\leq 1$. As $\tr(B)=2$, $\mathrm{rank}(B)=1$, so that the two eigenvalues of $B$ must be $\lambda_1=2$ and $\lambda_2=0$ (double). 
As $B$ is symmetric (hence diagonalizable), there is $P\in \SO(3)$ such that  
\[ B=P\left(\begin{matrix} 2 & 0 & 0 \\ 0 & 0 & 0 \\ 0 & 0 & 0 \end{matrix}\right) P^t,
\]
and we are in case \ref{a3}).

$\ast$ If $\alpha\neq 0$, then   $B^2=\frac{\det(B)}{\alpha}I_3$.
That is, if  $\rho:=\frac{\det(B)}{\alpha}$, and for all $r,s\in\{1,2,3\}$,
\begin{equation} \label{Bcuadrado} 
\sum_i b_{ri}b_{is}=\rho \delta_{rs}.
\end{equation}
Now, let us prove that 
$\tr(B)=3\rho/2$. 
Take $X\in\g_1$ unit. 
By \eqref{Bcuadrado}, 
$$\begin{array}{ll}
\circ& T(X,T(\varphi_1X,\varphi_2X))+T(\varphi_1X,T(\varphi_2X,X))+T(\varphi_2X,T(X,\varphi_1X)) \vspace{1pt}\\
& =T(X, -\sum_r b_{r3}\xi_r  )+T(\varphi_1X,  \sum_r b_{r2}\xi_r  )+T(\varphi_2X,  -\sum_r b_{r1}\xi_r  )\vspace{1pt} \\
& = -\sum_{r,s}b_{r3}b_{rs}\varphi_sX+\sum_{r,s} b_{r2} b_{rs}\varphi_s\varphi_1X -\sum_{r,s}b_{r1} b_{rs} \varphi_s\varphi_2X \vspace{1pt}\\
& = - \sum_s\rho\delta_{3s}\varphi_sX + \sum_{s} \rho \delta_{2s}\varphi_s\varphi_1X 
-\sum_{s}\rho\delta_{1s} \varphi_s\varphi_2X 
= -3\rho\varphi_3X; \vspace{3pt}\\
\circ& (\nabla^g_XT)(\varphi_1X,\varphi_2X) \\
&  =  
  -\sum_r b_{r3} \alpha^g(X,\xi_r) - \frac12 T([X,\varphi_1X]_{\mm},\varphi_2X) -\frac12 T(\varphi_1X,[X,\varphi_2X]_{\mm}) \vspace{1pt}\\
& = \sum_r b_{r3} \varphi_rX - T(\xi_1,\varphi_2X) - T(\varphi_1X, \xi_2)\vspace{1pt} \\
&= \sum_r b_{r3} \varphi_rX + \sum_s b_{1s}\varphi_s\varphi_2X - \sum_s b_{2s}\varphi_s\varphi_1X =
\tr(B)\varphi_3X.
\end{array}
$$
All together, $0 =(\nabla_XT)(\varphi_1X,\varphi_2X) 
=   \left(\tr(B)-\frac32 \rho\right)\varphi_3X$,
 and we get $\tr(B)=3\rho/2$.

Moreover, from $\alpha\rho=\det(B)$, we obtain $\alpha^2\rho^2=(\det(B))^2=\det(B^2)=\det(\rho I_3)$ $=\rho^3$. Thus, there are two possibilities, $\rho=0$ or $\rho=\alpha^2\neq 0$.   In this second case, $\rho=\alpha^2=(\tr(B)/2-1)^2 = \left(\frac{3\rho}{4}-1\right)^2$. The solutions to this equation are
uniquely  $\rho=4$ and $\rho=4/9$.

Since  $B=B^t$, there exists $P\in \SO(3)$ such that $B=P\,\mathrm{diag}\{\lambda_1,\lambda_2,\lambda_3\}P^t$, where $\lambda_i$, $i=1,2,3$, are the eigenvalues of $B$. Then $\rho I_3=B^2 =P\,\mathrm{diag}\{\lambda_1^2,\lambda_2^2,\lambda_3^2\}\,P^t$, and $\lambda_i^2=\rho$ for all $i$.
\begin{itemize}
\item \underline{If $ \rho=0$}, then 
$\lambda_i=0$ for all $i$, so that $B=0$ and $T=0$, which is not our case. 

\item \underline{If $\rho=4$}, then $\lambda_i\in\{\pm 2\}$, $i=1,2,3$. Since $\sum_i\lambda_i=\tr(B)=3\rho /2=6$, then the only possibility is $\lambda_i=2$ for any $i=1,2,3$. That is to say, $B=2I_3$, the case \ref{a2}).

\item  \underline{If $\rho=4/9$}, then $\lambda_i\in \{\pm 2/3\}$, $i=1,2,3$. But now, $\sum_i\lambda_i=\tr(B)=3\rho/2 = 2/3$. The unique possibility is $\lambda_1=2/3=\lambda_2=-\lambda_3$ (up to reordering the indices). To finish, we only have to check that this case \ref{a4}) is only possible when $\dim M=7$. 
\end{itemize}

Assume that $\dim M>7$. We can choose  $X,Z\in\g_{1}$ such that $X$ is unit, $Z\neq 0$, and $g(X,Z)=g(X,\varphi_kZ)=0$ for any $k=1,2,3$. Next, we  compute $(\nabla_XT)(\varphi_kX,Z)$. In this case, it holds 
\[
[X,Z]_{\mm}=2\sum_rg(\varphi_rX,Z)\xi_r=0, \quad 
[X,\varphi_kX]_{\mm}=2\sum_rg(\varphi_rX,\varphi_kX)\xi_r=2\xi_k.
\]
Now, by \eqref{tensoresT}, $T(X,Z)=0=T(\varphi_kX,Z)$ and $T(X,\varphi_kX)=\sum_{r,s}b_{rs}g(X,\varphi_s\varphi_kX)\xi_r = -\sum_r b_{rk}\xi_r$. And, by \eqref{Bcuadrado}, we get 
$$
\begin{array}{ll}
\circ& T(X,T(\varphi_kX,Z))+T(\varphi_kX,T(Z,X))+T(Z,T(X,\varphi_kX))= -\sum_r b_{rk}T(Z,\xi_r) \vspace{1pt} \\
& = -\sum_r b_{rk}\sum_s b_{rs} \varphi_sZ = -\sum_{r,s} b_{kr}b_{rs}\varphi_sZ=-\sum_s\rho \delta_{ks}\varphi_sZ=-\rho \varphi_kZ;\vspace{3pt} \\
\circ& (\nabla^g_XT)(\varphi_kX,Z) = \alpha^g(X,T(\varphi_kX,Z))-T(\alpha^g(X,\varphi_kX),Z)-T(\varphi_kX,\alpha^g(X,Z)) \vspace{1pt} \\
& = -\frac12 T([X,\varphi_kX]_{\mm},Z) = T(Z,\xi_k)=\sum_sb_{sk}\varphi_sZ.
\end{array}
$$ 
Therefore,  
$0=(\nabla_XT)(\varphi_kX,Z)=   ( \sum_sb_{sk}\varphi_s  -\frac\rho2 \varphi_k)Z$  for any $k=1,2,3$.
This implies $b_{ks}=0$ for $k\neq s$ and $b_{kk}=\rho/2$ for any $k=1,2,3$. The matrix $B$ becomes $B=(\rho/2)I_3$. But as $B^2=\rho I_3$, then $\rho=\rho^2/4$. That is, either $\rho=0$ or $\rho=4$. Thus, the case $\rho=\frac49$ is ruled out.

Conversely, we can check that any of these torsions $T$ is indeed parallel (for the corresponding affine connection $\nabla=\nabla^g+\frac12T$).  The ideas have been already shown, so that these long but tedious computations are left to the reader.  
\end{proof}

 \begin{corollary} Let $M=G/H$ be a 3-Sasakian homogeneous manifold, $\dim M\geq 7$. Let $\nabla$ be a  $G$-invariant  affine connection  which is Einstein with nonzero parallel skew  torsion. Then,   $\dim M=7$ and $\nabla=\nabla^g+\frac12\sum_{r,s}b_{rs}T^{rs}$, where $B=(b_{rs})$ satisfies $B=B^t$, $I_3\ne -\frac32 B\in \SO(3)$. 
\end{corollary}

Note that the results agree with those ones in the sphere $\mathbb{S}^7$ in \cite[Remark~5.16]{DraperGarvinPalomo}, where the family in Theorem~\ref{th_paralela} iv) appeared from a different approach.

\section*{Acknowledgements}

The authors would like to thank the referee for the deep reading of this manuscript,  for his/her valuable  suggestions to include Section~\ref{sec_paralela} on parallel torsion and for several comments on the bibliography.

\end{document}